\documentclass[12pt,A4,leqno]{amsart}
\usepackage{amsfonts}
\usepackage{mathrsfs}
\usepackage[T1]{fontenc}
\usepackage{amssymb,amscd}
\usepackage{upgreek}
\usepackage{color}
\usepackage{epsf}
\usepackage{graphicx}
\usepackage{extarrows}
\usepackage[curve]{xypic}

\theoremstyle{plain}
\newtheorem{thm}{Theorem}[section]
\newtheorem{lem}[thm]{Lemma}

\newtheorem{cor}[thm]{Corollary}

\theoremstyle{definition}
\newtheorem{defi}[thm]{Definition}
\newtheorem{exam}[thm]{Example}
\newtheorem{rem}[thm]{Remark}

\newcommand{\R}{\mathbb R}
\newcommand{\Z}{\mathbb Z}

\newcommand{\largetimes}{\mbox{\LARGE$\times$}}

\newcommand{\normaltimes}{\scalebox{1}{$\times$}}

\makeatletter
\renewcommand{\@secnumfont}{\bfseries}
\makeatother

\makeatletter
\def\section{%
  \@startsection{section}{1}
    {\z@}
    {2.0ex plus 0.8ex minus .1ex}
    {1.0ex plus .2ex}
    {\bfseries\large\centering\MakeUppercase}%
}
\makeatother

\begin{document}

\title [\ ] {A generalization of moment-angle manifolds with non-contractible orbit spaces}

\author{Li Yu}
\address{Department of Mathematics, Nanjing University, Nanjing, 210093, P.R.China.
  }
 \email{yuli@nju.edu.cn}





\begin{abstract}
 We generalize the notion of moment-angle manifold over a simple convex polytope to an arbitrary nice manifold with corners.
For a nice manifold with corners $Q$,
we first compute the stable decomposition of the moment-angle manifold $\mathcal{Z}_Q$ via a construction called rim-cubicalization of $Q$. From this, we derive a formula to compute the integral cohomology group of $\mathcal{Z}_Q$  via the strata of $Q$. This generalizes the Hochster's formula for the moment-angle manifold over a simple convex polytope. 
 Moreover, we obtain a description of the integral cohomology ring of $\mathcal{Z}_Q$ using the idea of partial diagonal maps. In addition, we define the notion of polyhedral product of a sequence of based CW-complexes over $Q$ and obtain similar results for these spaces as we do for $\mathcal{Z}_Q$. 
Using this general construction, we can compute the 
equivariant cohomology ring of $\mathcal{Z}_Q$ with respect to its canonical torus action from the Davis-Januszkiewicz space of $Q$. The result leads to the definition of a new notion called the topological face ring of $Q$, which generalizes the notion of face ring of a simple polytope. Moreover, the topological face ring of $Q$ computes the equivariant cohomology of all locally standard torus actions with $Q$ as the orbit space when the corresponding principal torus bundle over $Q$ is trivial. 
Meanwhile, we obtain some parallel results for the real moment-angle manifold $\R\mathcal{Z}_Q$ over $Q$ as well. 
  \end{abstract}

\maketitle

 \section{Introduction} 
  The construction of a moment-angle manifold over a simple polytope is first introduced in
  Davis-Januszkiewicz~\cite{DaJan91}. Suppose $P$ is a simple (convex) polytope with $m$ \emph{facets} (codimension-one faces).
  A convex polytope in a Euclidean space is called \emph{simple} if every codimension-$k$ face is the intersection of exactly $k$ facets of the polytope.   
   The moment-angle manifold $\mathcal{Z}_P$ over $P$ is a closed connected manifold
  with an effective action by the compact torus $T^m=(S^1)^m$ whose orbit space is $P$. It is shown 
  in~\cite{DaJan91} that many important topological invariants of $\mathcal{Z}_P$ can be computed easily from the combinatorial structure of $P$. These manifolds play an important role in the research of toric topology. The reader is referred to Buchstaber-Panov~\cite{BP02, BP15} for more discussions on the topological and geometrical aspects of moment-angle manifolds.  
  
 The notion of moment-angle manifold over a simple convex polytope has been generalized in many different ways. For example, Davis and Januszkiewicz~\cite{DaJan91} define a class of topological spaces now called \emph{moment-angle complexes} (named by Buchstaber and Panov in~\cite{BP00}) where the simple polytope is replaced by a simple polyhedral complex. Later, L\"u and Panov~\cite{LuPanov11}
    defined the notion of moment-angle complex of a \emph{simplicial poset}. In addition, Ayzenberg and Buchstaber~\cite{AntonBuch11} defined the notion of moment-angle spaces over arbitrary convex polytopes (not necessarily simple). Note that in all these generalizations, the orbit spaces of the canonical torus actions are all contractible.    
  Yet an even wider class of spaces called
\emph{generalized moment-angle complexes} or \emph{polyhedral products}
over simplicial complexes were introduced by Bahri, Bendersky, Cohen and Gitler in~\cite{BBCG10}, which has become the major subject in the homotopy theoretic study of toric topology.
 
  In this paper, we generalize the construction of moment-angle manifolds by replacing the simple polytope $P$ by a nice manifold with corners $Q$ which is not necessarily contractible. Such a generalization has
  been considered by Poddar and Sarkar~\cite{PodSark15} for polytopes with simple holes.
  
  A motive for the study of this generalized construction is to compute the equivariant cohomology ring of locally standard torus actions. Recall that an action of a compact torus $T^n$ on a smooth compact manifold $M$ of dimension
$2n$ is called locally standard if it is locally modeled on the standard representation of $T^n$
on $\mathbb{C}^n$. Then the orbit space $Q= M\slash T^n$ is  a manifold with corners. Conversely, every manifold with a
locally standard $T^n$-action and with $Q$ as the orbit space is equivariantly homeomorphic to the quotient construction
$Y\slash \sim$, where $Y$ is a principal $T^n$-bundle over $Q$ and $\sim$ is an equivalence relation
determined by the characteristic function on $Q$ (see~\cite{Yos11}). Generally speaking, it is difficult to compute the equivariant cohomology ring of $M$ from the corresponding principal bundle $Y$ and the characteristic function on $Q$. 
But we will see in Corollary~\ref{Cor:Equiv-Cohom-Free-Quotient} that when $Y$ is the trivial $T^n$-bundle  over $Q$, the equivariant cohomology ring of $M$ can be computed from the strata of $Q$ directly. Examples of such kind include many toric origami manifolds (see~\cite{CanGuilPires11, HolmPires13, AntonMasParkZeng17}) with coorientable folding hypersurface where the faces of the orbit spaces may be non-acyclic.
 
 Recall that an \emph{$n$-dimensional manifold with corners} $Q$ is a
   Hausdorff space with a maximal
  atlas of local charts onto open subsets of $\R_{\geq 0}^n $
  such that the transitional functions are homeomorphisms which preserve the codimension of each point. Here
  the \emph{codimension} $c(x)$ of a point $x=(x_1,\cdots,x_n)$ in $\R_{\geq 0}^n$ is the number of
  $x_i$ which are $0$. So we have a well defined map
  $c: Q\rightarrow \Z_{\geq 0}$ where $c(q)$ is the codimension of a point $q\in Q$. In particular, the interior $Q^{\circ}$ of $Q$ consists of points of codimension $0$, i.e. $Q^{\circ} = c^{-1}(0)$.

   Suppose $Q$ is an $n$-dimensional manifold with corners with $\partial Q\neq \varnothing$.
  An \emph{open face} of $Q$ of codimension $k$
  is a connected component of $c^{-1}(k)$. A (closed) 
  \emph{face}
  is the closure of an open face. A face of codimension one is called a \emph{facet} of $Q$. Note that a face of codimension zero in $Q$ is just a connected component of $Q$.

  A manifold with corners $Q$ is said to be \emph{nice} if either its boundary $\partial Q$ is empty or
  $\partial Q$ is non-empty and any codimension-$k$ face of $Q$ is a component of the intersection of
  $k$ different facets in $Q$. 
  
   Let $Q$ be a nice $n$-manifold with corners.
   Let $\mathcal{F}(Q) =\{F_1,\cdots, F_m \}$ 
   be the set of facets of $Q$.  For any subset $J\subseteq [m]=\{1,\cdots,m\}$, let
\[ 
 F_J = \bigcup_{j\in J} F_j, \ F_{\varnothing} = \varnothing ;  \ \ F_{\cap J} =  \bigcap_{j\in J} F_j, \ F_{\cap \varnothing} = Q.
\]

  It is clear that 
 $$F_J\subseteq F_{J'}, \ \ F_{\cap J'}\subseteq F_{\cap J}, \ \ F_{\cap J}\subseteq F_J,\ \forall J\subseteq J'\subseteq [m].$$
 
   Let $\lambda: \mathcal{F}(Q) \rightarrow
    \Z^m$ be a map such that $\{\lambda(F_1),\cdots,\lambda(F_m)\}$ is a unimodular basis of $\Z^m\subset \R^m$.
   Since $S^1=\{ z\in \mathbb{C}\,|\, \|z\|=1\}$,
  we can identify the $m$-torus $(S^1)^m=\R^m\slash \Z^m$. The \emph{moment-angle manifold over $Q$} is defined by:
    \begin{equation} \label{Equ:glue-back-complex}
       \mathcal{Z}_{Q} = Q\times (S^1)^m \slash \sim
     \end{equation} 
     where $(x,g) \sim (x',g')$ if and only if $x=x'$ and $g^{-1}g' \in \mathbb{T}^{\lambda}_x$
   where $\mathbb{T}^{\lambda}_x$ is the subtorus of $(S^1)^m$ determined by
   the linear subspace of $\R^m$ spanned by the set $\{ \lambda(F_j) \, |\, x\in F_j \}$. There is a canonical action of $(S^1)^m$ on $\mathcal{Z}_{Q}$ defined by:
   \begin{equation} \label{Equ:Canon-Action-Complex}
    g' \cdot [(x,g)] = [(x,g'g)], \ x\in Q,\ g,g'\in (S^1)^m. 
   \end{equation} 
 
  Since the manifold with corners $Q$ is nice and $\lambda$ is unimodular, it is easy to see from the above definition that $\mathcal{Z}_{Q}$ is a manifold. \\
     
 \noindent \textbf{Convention:} In the rest of this paper, 
 we assume that any nice manifold with corners $Q$ can be equipped with a 
 CW-complex structure such that every face of $Q$ is a 
 subcomplex. 
 In addition, we assume that $Q$ has only finitely many faces.
Note that a compact smooth nice manifold with corners
always satisfies these two conditions since it is triangulable (see Johnson~\cite{John83}).  But in general we do not require $Q$ to be compact or smooth. We do not assume $Q$ to be connected either.\\

 Similarly to the stable decomposition of (generalized) moment-angle complexes obtained in~\cite{BBCG10}, we have the following stable decomposition of $\mathcal{Z}_Q$.
 
  \begin{thm}\label{Thm:Stable-Decomp-Main}
     Let $Q$ be a nice manifold with corners with facets $F_1,\cdots, F_m$. There is a homotopy equivalence
         \begin{equation}  
           \mathbf{\Sigma}( \mathcal{Z}_Q ) \simeq \bigvee_{J\subseteq [m]} 
           \mathbf{\Sigma}^{|J|+1}( Q\slash F_J)
           \end{equation}
          where $\bigvee$ denotes the wedge sum and $\mathbf{\Sigma}$ denotes the reduced suspension. 
       \end{thm}
       
       Here we will not explicitly write down the basepoints for our spaces unless it is necessary to do so.

          \begin{cor} \label{Cor:Main}
  Let $Q$ be a nice manifold with corners with facets $F_1,\cdots, F_m$. The integral (reduced) cohomology group of
  $\mathcal{Z}_Q$ is given by:
  \begin{equation} \label{Equ:Main-Result}
     H^p(\mathcal{Z}_Q)\cong \bigoplus_{J \subseteq [m]} H^{p-|J|}(Q,F_J),\ \ 
      \widetilde{H}^p(\mathcal{Z}_Q) \cong \bigoplus_{J \subseteq [m]} \widetilde{H}^{p-|J|}(Q\slash F_J),
     \, \forall p\in \Z. 
  \end{equation}   
\end{cor}

 Note that when $J=\varnothing$, $H^{*}(Q,F_{\varnothing})=H^{*}(Q,\varnothing)=H^{*}(Q)\cong\widetilde{H}^{*}(Q)\oplus \Z$. 

The term ``cohomology'' of a space $X$, denoted by $H^*(X)$, in this paper always means singular cohomology with integral coefficients if not specified otherwise.

 When $Q$ is acyclic (i.e. $\widetilde{H}^*(Q)=0$), we have
  $H^{p}(Q,F_J)\cong \widetilde{H}^{p-1}(F_J)$ by a cohomology long exact sequence for the pair $(Q,F_J)$. So in this case, 
  \[    H^p(\mathcal{Z}_Q)\cong \bigoplus_{J \subseteq [m]} \widetilde{H}^{p-|J|-1}(F_J),\, \forall p\in \Z. \]  
 This recovers Hochster's formula for the moment-angle manifold over a simple polytope in~\cite[Theorem 3.2.9]{BP15} (also see~\cite[Proposition 3.2.11]{BP15}). 
 
 \begin{rem} \label{Rem:Real-Moment-Angle}
  There is an analogue of $\mathcal{Z}_Q$ by replacing the group $(S^1)^m$ by $(\Z_2)^m$. The counterpart in the $(\Z_2)^m$ construction, denoted by
  $\R\mathcal{Z}_Q$,
  is a special case of the basic construction in~\cite[Ch.\,5]{Davisbook08} for a mirror space along with a Coxeter system.  A formula parallel to Corollary~\ref{Cor:Main} for computing the integral cohomology group of
  $\R\mathcal{Z}_Q$ is contained in~\cite[Theorem A]{Da87} (also see~\cite[Ch.\,8]{Davisbook08}). We call
  $\R\mathcal{Z}_Q$ the \emph{real moment-angle manifold over $Q$}.
\end{rem}
        
    Given a nice manifold with corners $Q$ with facets $F_1,\cdots, F_m$, define 
  \begin{equation} \label{Equ:Def-R-Q}
    \mathcal{R}^*_Q :=\bigoplus_{J \subseteq [m]} H^*(Q,F_J)
    \end{equation} 
  
  There is a graded ring structure $\Cup$ on $\mathcal{R}^*_Q $ defined as follows:
  \begin{itemize}
       \item If $J\cap J' \neq \varnothing$, then
     $ H^*(Q,F_J) \otimes H^{*}(Q,F_{J'}) \overset{\Cup}{\longrightarrow}
    H^{*}(Q,F_{J\cup J'})$ is trivial.
    
     \item If $J\cap J'=\varnothing$, then
  $ H^*(Q,F_J) \otimes H^{*}(Q,F_{J'}) \overset{\Cup}{\longrightarrow}
    H^{*}(Q,F_{J\cup J'})$ is the relative cup product $\cup$
     (see~\cite[p.\,209]{Hatcher02}).     
    \end{itemize}

 We can prove the following theorem
   via the above stable decomposition of $\mathcal{Z}_Q$.   
 
   \begin{thm}\label{Thm:Cohomology-Ring-Isom}
   Let $Q$ be a nice manifold with corners with $m$ facets
   $F_1,\cdots, F_m$. There exists a ring isomorphism (up to a sign)
     from $(\mathcal{R}^*_Q,\Cup)$ to the integral cohomology ring of $\mathcal{Z}_Q$. Moreover, we can make this ring isomorphism
     degree-preserving by shifting the degrees of all the elements in $H^*(Q,F_J)$ up by $|J|$ for every $J\subseteq [m]$.
  \end{thm}
  
   It is indicated in~\cite[Exercise\,3.2.14]{BP15} that Theorem~\ref{Thm:Cohomology-Ring-Isom} holds for any simple polytope.   
 Moreover, we can generalize
 Theorem~\ref{Thm:Cohomology-Ring-Isom} to describe the  cohomology ring of the polyhedral product of any
$(\mathbb{D},\mathbb{S}) = \{ \big( D^{n_j+1}, S^{n_j}, a_j \big) \}^m_{j=1}$ over $Q$ (see Theorem~\ref{Thm:Main-Dn-Sn-1}).       
In particular,
we have the following result for $\R\mathcal{Z}_Q$.

  \begin{thm}[Corollary~\ref{Cor:Real-Moment}]
    Let $Q$ be a nice manifold with corners with facets $F_1,\cdots, F_m$. Then 
     $$  \mathbf{\Sigma}( \R\mathcal{Z}_Q ) \simeq \bigvee_{J\subseteq [m]} 
           \mathbf{\Sigma}( Q\slash F_J), \ \ H^p(\R\mathcal{Z}_Q) \cong \bigoplus_{J\subseteq [m]} H^p(Q, F_J), \ \forall p\in \Z.$$
    Moreover, the integral cohomology ring of $\R\mathcal{Z}_Q$ is isomorphic as a graded ring to the ring $(\mathcal{R}^*_Q,\cup)$
     where $\cup$ is the relative cup product
      $$ H^*(Q,F_J) \otimes H^{*}(Q,F_{J'}) \overset{\cup}{\longrightarrow}
    H^{*}(Q,F_{J\cup J'}), \ \forall \, J,J'\subseteq [m].$$ 
  \end{thm}

   We can describe the equivariant cohomology ring of $\mathcal{Z}_Q$ with respect to the canonical action
    of $(S^1)^m$ as follows. 
    
  Let $\mathbf{k}$ denote a commutative ring with a unit.
  For any $J\subseteq [m]$, let $R^J_{\mathbf{k}}$ be the subring of
   the polynomial ring $\mathbf{k}[x_1,\cdots, x_m]$ defined by
    \begin{equation}  \label{Equ:RJ-K}
     R^J_{\mathbf{k}} : = 
      \begin{cases}
   \mathrm{span}_{\mathbf{k}}\{ x^{n_1}_{j_1}\cdots x^{n_s}_{j_s} \,|\,
       n_1>0, \cdots, n_s>0 \} ,  &  \text{if $J=\{ j_1,\cdots, j_s \}\neq \varnothing$}; \\
  \ \ \mathbf{k},  &  \text{if $J=\varnothing$}.
 \end{cases} 
  \end{equation}  
  
    We can multiply $f(x)\in  R^J_{\mathbf{k}}$ and 
    $f'(x)\in R^{J'}_{\mathbf{k}}$ in 
    $\mathbf{k}[x_1,\cdots, x_m]$ and obtain an element
    $f(x)f'(x)\in R^{J\cup J'}_{\mathbf{k}}$.
   
  \begin{defi}[Topological Face Ring] \label{Def:Top-Face-Ring}
    Let $Q$ be a nice manifold with corners with $m$ facets $F_1,\cdots, F_m$. For any coefficients ring $\mathbf{k}$, the \emph{topological face ring of $Q$} over $\mathbf{k}$ is defined to be
     \begin{equation} \label{Equ:Def-Top-Face-Ring}
      \mathbf{k}\langle Q\rangle := \bigoplus_{J\subseteq [m]} H^*( F_{\cap J};\mathbf{k})\otimes R^J_{\mathbf{k}}. \end{equation}
  Here if $F_{\cap J}=\varnothing$, we use the convention $H^*(\varnothing;\mathbf{k})=\{0\}$.

     For any $J,J'\subseteq [m]$, the product $\star$ on $\mathbf{k}\langle Q\rangle$,
     $$ \Big( H^*( F_{\cap J};\mathbf{k})\otimes R^J_{\mathbf{k}} \Big) \otimes
     \Big( H^*( F_{\cap J'};\mathbf{k})\otimes R^{J'}_{\mathbf{k}} \Big) \xlongrightarrow{\,\ \scalebox{0.9}{$\star$}\,\ }
     \Big( H^*( F_{\cap (J\cup J')};\mathbf{k})\otimes R^{J\cup J'}_{\mathbf{k}} \Big) $$
    for $\phi\in H^*( F_{\cap J};\mathbf{k})$, 
     $\phi'\in H^*( F_{\cap J'};\mathbf{k})$ and
     $f(x)\in R^J_{\mathbf{k}}$, $f'(x)\in R^{J'}_{\mathbf{k}}$,  is defined by:
     \begin{equation} \label{Equ:Face-Ring-Product}
      (\phi \otimes f(x))\star (\phi'\otimes f'(x)):=
      \big(\kappa^*_{J\cup J', J}(\phi)\cup
   \kappa^*_{J\cup J', J'}(\phi') \big) \otimes f(x)f'(x)   
    \end{equation}
    where $\kappa_{I', I} : F_{\cap I'}\rightarrow F_{\cap I}$ is the inclusion map for any subsets $I\subseteq I'\subseteq [m]$ and $\kappa^*_{I', I}:
    H^*( F_{\cap I};\mathbf{k}) \rightarrow H^*( F_{\cap I'};\mathbf{k})$ is the induced homomorphism on 
    cohomology.
  \end{defi}
  
  In addition, we can consider $\mathbf{k}\langle Q\rangle$ as a graded ring if we choose a degree for every indeterminate $x_j$
  in $\mathbf{k}[x_1,\cdots, x_m]$ and define
  $$\mathrm{deg}\big(\phi\otimes (x^{n_1}_{j_1}\cdots x^{n_s}_{j_s})\big) = \mathrm{deg}(\phi) + n_1\mathrm{deg}(x_{j_1})+\cdots + n_s\mathrm{deg}(x_{j_s}).$$

  \begin{thm} \label{Thm:Equivariant-Cohomology-Z-Q}
    Let $Q$ be a nice manifold with corners with facets $F_1,\cdots, F_m$. Then the equivariant cohomology ring of $\mathcal{Z}_Q$ (or $\R\mathcal{Z}_Q$) with $\Z$-coefficients (or $\Z_2$-coefficients) with respect to the canonical
     $(S^1)^m$-action
     (or $(\Z_2)^m)$-action) is isomorphic as a graded ring to the topological face ring $\Z\langle Q\rangle$ (or $\Z_2\langle Q\rangle$) of $Q$ by choosing
    $\mathrm{deg}(x_j)=2$ (or $\mathrm{deg}(x_j)=1$) for all $1\leq j \leq m$.
  \end{thm}
 
Moreover,
  the natural $H^*(BT^m)$-module structure on
  the integral equivariant cohomology ring $H^*_{T^m}(\mathcal{Z}_Q)$ is described in~\eqref{Equ-BT-module-struc} where $T^m=(S^1)^m$.
  
  \begin{rem}
 A calculation of the equivariant cohomology group of $\mathcal{Z}_Q$ with $\Z$-coefficients was announced earlier  by T.~Januszkiewicz in a talk~\cite{Jan20} in 2020. The formula given in Januszkiewicz's talk is equivalent to our $\Z\langle Q\rangle$. But the ring structure of the equivariant cohomology of $\mathcal{Z}_Q$ was not described in~\cite{Jan20}.
\end{rem}

    For a nice manifold with corners $Q$, 
    there are two other notions which reflect the
    stratification of $Q$.
    One is the \emph{face poset} of $Q$ which is 
    the set of all faces of $Q$ ordered by inclusion, denoted by $\mathcal{S}_Q$ (note that each connected component of $Q$ is also a face). 
    The other one is the \emph{nerve simplicial complex} of the covering of $\partial Q$ by its facets, denoted by
  $K_Q$. The \emph{face ring} (or \emph{Stanley-Reisner ring}) of 
  a simplicial complex is an important
  tool to study combinatorial objects in algebraic combinatorics and combinatorial commutative algebra  (see~\cite{MilSturm05} and~\cite{Stanley07}). 
    
   When $Q$ is a simple polytope, all faces
of $Q$, including $Q$ itself, and all their intersections are acyclic.
   Then it is easy to see that the topological face ring of $Q$ is isomorphic to 
   the face ring of $K_Q$ (see Example~\ref{Exam:Face-Ring-Polytope}). But in general, the topological face ring of $Q$ encodes more topological information of $Q$ than
  the face ring of $K_Q$.

  There is another way to think of the topological face ring $\mathbf{k}\langle Q\rangle$. Let 
  $$ \mathcal{R}^*_{\cap Q,\mathbf{k}}:= \bigoplus_{J\subseteq [m]} H^*( F_{\cap J};\mathbf{k}) $$
  where product $*$ on $\mathcal{R}^*_{\cap Q,\mathbf{k}}$ is defined by: for any $\phi\in H^*( F_{\cap J};\mathbf{k}), 
     \phi'\in H^*( F_{\cap J'};\mathbf{k})$,
  $$ \phi * \phi' :=
     \kappa^*_{J\cup J', J}(\phi)\cup
   \kappa^*_{J\cup J', J'}(\phi') \in H^*( F_{\cap (J\cup J')};\mathbf{k}).$$
  
   Moreover,  
   $$\mathbf{k}[x_1,\cdots, x_m]=\bigoplus_{J\subseteq [m]} R^J_{\mathbf{k}},$$   
     so we can think of both $\mathcal{R}^*_{\cap Q,\mathbf{k}}$ and 
   $\mathbf{k}[x_1,\cdots, x_m]$ as $2^{[m]}$-graded rings where $2^{[m]} = \{ J\subseteq [m]\}$ is the power set of $[m]$.
    Then the topological face ring $\mathbf{k}\langle Q\rangle$ can be thought of as the
  Segre product of $\mathcal{R}^*_{\cap Q,\mathbf{k}}$ and 
   $\mathbf{k}[x_1,\cdots, x_m]$ with respect to their
   $2^{[m]}$-gradings. By definition, the \emph{Segre product} of two rings 
   $R$ and $S$ graded by a common semigroup $\mathcal{A}$
   (using the notation in~\cite{Hoa88}) is:
   $$ R \, \scalebox{1.2}{$\underline{\otimes}$}\,  S =
     \bigoplus_{\mathbf{a}\in \mathcal{A}} R_{\mathbf{a}}\otimes S_{\mathbf{a}}. $$
     
  So $R \, \scalebox{1.2}{$\underline{\otimes}$}\,  S $ is a subring of the tensor product of $R$ and $S$ (as graded rings). The Segre product of two graded 
  rings (or modules) is studied in algebraic geometry
   and commutative algebra (see~\cite{Chou64} and~\cite{Hoa88,FrobHoa92} for example).
   
    Here we can think of $2^{[m]}$ as a semigroup where the product of two subsets of $[m]$ is just their union.   
   Then by this notation, we can write 
   $$ \mathbf{k}\langle Q\rangle = \mathcal{R}^*_{\cap Q,\mathbf{k}}\, \scalebox{1.2}{$\underline{\otimes}$}\,   \mathbf{k}[x_1,\cdots, x_m]. $$ 
   
   From this form, we see that $\mathbf{k}\langle Q\rangle$ is essentially determined by $\mathcal{R}^*_{\cap Q,\mathbf{k}}$.

       The paper is organized as follows.       
 In Section~\ref{Sec:Stable-Decomp},
we first construct an embedding of $Q$ into $Q\times [0,1]^m$ which is analogous to the embedding of a simple polytope into a cube. This induces an embedding 
of $\mathcal{Z}_Q$ into $Q\times (D^2)^m$ from which we can do the stable decomposition of $\mathcal{Z}_Q$ and give a proof of Theorem~\ref{Thm:Stable-Decomp-Main}. Our argument proceeds along the same line as the argument given in~\cite[Sec.\,6]{BBCG10} but with some extra ingredients. In fact, we will not do the stable decomposition of $\mathcal{Z}_Q$ directly, but the stable decomposition of the disjoint union of $\mathcal{Z}_Q$ 
with a point. In Section~\ref{Sec:Cohomology}, we obtain a description of the product structure of the cohomology of 
$\mathcal{Z}_Q$ using the stable decomposition of $\mathcal{Z}_Q$ and the partial diagonal map introduced in~\cite{BBCG12}. From this we give a proof of Theorem~\ref{Thm:Cohomology-Ring-Isom}.
In Section~\ref{Sec:Polyhedral-Prod}, we define the notion of polyhedral product of a sequence of based CW-complexes over a nice manifold with corners $Q$ and obtain some results parallel to
$\mathcal{Z}_Q$ for these spaces. In particular, we obtain a description of the integral cohomology ring of real moment-angle manifold
$\R\mathcal{Z}_Q$ (see Corollary~\ref{Cor:Real-Moment}).
   In Section~\ref{Sec:Equivariant-Cohom}, we compute the equivariant cohomology ring of $\mathcal{Z}_{Q}$ and prove Theorem~\ref{Thm:Equivariant-Cohomology-Z-Q}.
In Section~\ref{Sec:Generalization}, we discuss more generalizations of the construction of $\mathcal{Z}_Q$ and extend our main theorems to some wider settings.

   \section{Stable Decomposition of $\mathcal{Z}_Q$} \label{Sec:Stable-Decomp}

  Let $Q$ be a nice manifold with corners with $m$ facets. To obtain the stable decomposition of $\mathcal{Z}_Q$, we first construct a  special embedding of $Q$ into $Q\times [0,1]^m$, called the rim-cubicalization of $Q$.
  This construction can be thought of as a generalization of the embedding of a simple polytope with $m$ facets into 
  $[0,1]^m$ defined in~\cite[Ch.\,4]{BP02}. 
  
  \subsection{Rim-cubicalization of $Q$ in $Q\times [0,1]^m$} \label{Subsec:Rim-Cubical}
 \ \vskip .1cm
  
  Let $F_1,\cdots, F_m$ be all the facets of $Q$.
  For a face $f$ of $Q$, let $I_f$ be the following subset of $[m]$ called the \emph{strata index} of $f$.
      $$I_f =\{ i\in [m]\,|\, f\subseteq F_i\} \subseteq [m].$$
      
  Then we define a subset $\widehat{f}$ of $Q\times [0,1]^m$ associated to $f$ as follows. We write 
$$[0,1]^m = \prod_{j\in [m]} [0,1]_{(j)},$$
  and define
      \begin{equation*}
       \widehat{f} =  f 
       \times \prod_{j\in I_f} [0,1]_{(j)} 
   \times \prod_{j\in [m]\backslash I_f} 1_{(j)}.
   \end{equation*}
    In particular, 
   \begin{equation*}
     \widehat{F}_i =  F_i
       \times  [0,1]_{(i)} 
   \times \prod_{j\in [m]\backslash \{i\}} 1_{(j)},\ 1\leq i \leq m.
   \end{equation*}
    
  Let $\mathcal{S}_Q$ be the face poset of $Q$ and define
   \begin{equation} \label{Equ:Q-Hat}
     \widehat{Q} = \bigcup_{f\in \mathcal{S}_Q}  \widehat{f}\, \subseteq \, Q\times [0,1]^m.
   \end{equation}
   
   It is easy to see that $\widehat{Q}$ is a nice manifold with corners whose facets are $\widehat{F}_1,\cdots, \widehat{F}_m$. 
     If we identify $Q$ 
  with the subspace $Q\times \prod_{j\in [m]} 1_{(j)} \subseteq \widehat{Q}$, then 
  we can think of $\widehat{Q}$ as inductively gluing 
   the product of all codimension-$k$ strata of $Q$ with a $k$-cube to $\partial Q$ (see Figure~\ref{p:Extend-Cell}).
   $$ Q  \xlongleftarrow{\ \text{glue}\,} F_j \times [0,1] \xlongleftarrow{\ \text{glue}\,} \cdots
  \xlongleftarrow{\ \text{glue}\,}   f \times [0,1]^{|I_f|} \xlongleftarrow{\ \text{glue}\,} \cdots $$
   Due to this viewpoint, we call $\widehat{Q}$ the \emph{rim-cubicalization} of $Q$ in $Q\times [0,1]^m$.

\begin{figure}
         \includegraphics[width=0.36\textwidth]{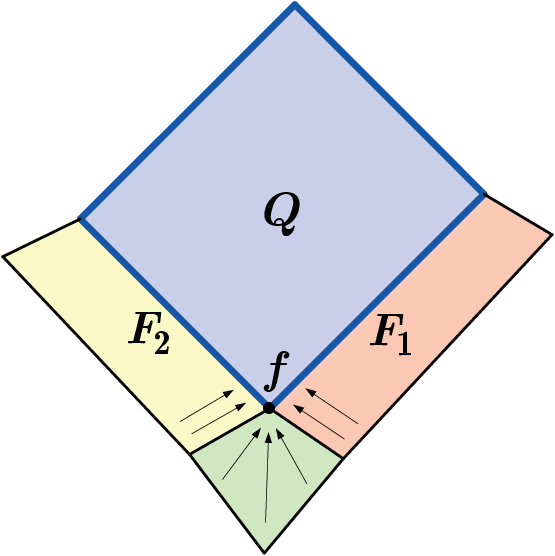}\\
          \caption{Rim-cubicalization of $Q$ in $Q\times [0,1]^m$}\label{p:Extend-Cell}
      \end{figure}

 \begin{lem}\label{Lem:Homeo-Q-Q-hat}
   $\widehat{Q}$ is homeomorphic to $Q$ as a manifold with corners.
 \end{lem} 
 \begin{proof}
  For any face $f$ of $Q$ and $0\leq t\leq 1$, let 
  \begin{align*}
        \widehat{f}(t) &=  f 
       \times \prod_{j\in I_f} [t,1]_{(j)} 
   \times \prod_{j\in [m]\backslash I_f} 1_{(j)}\\
     \widehat{Q}(t) &= \bigcup_{f\in \mathcal{S}_Q}  \widehat{f}(t) \, \subseteq\, Q\times [t,1]^m.
   \end{align*}
  
   Then $\widehat{Q}(t)$ determines an isotopy (see Figure~\ref{p:Extend-Cell}) from $\widehat{Q}(0) = \widehat{Q}$ to 
   $$\widehat{Q}(1)= \bigcup_{f\in \mathcal{S}_Q} \Big( f \times \prod_{j\in [m]} 1_{(j)} \Big) = Q\times \prod_{j\in [m]} 1_{(j)}\cong Q.$$    
   
  Around a codimension-$k$ stratum of $Q$, the isotopy $\widehat{Q}(t)$ is locally equivalent to the standard isotopy
    from $C^n_k(-1)$ to $C^n_k(0)$ defined in Example~\ref{Exam:Isotopy-Right-Angled-Cell}.

 Clearly, the isotopy $\widehat{Q}(t)$ sends each face
   $\widehat{f}$ of $\widehat{Q}$ to
 $f \times \prod_{j\in [m]} 1_{(j)}$. 
   So under the identification of $Q\times \prod_{j\in [m]} 1_{(j)}$ with $Q$, $\widehat{Q}$ is homeomorphic to $Q$ as a manifold with corners. The lemma is proved.
 \end{proof}

   \begin{exam} \label{Exam:Isotopy-Right-Angled-Cell}
  Let $C^n_k(0)$ and $C^n_k(-1)$ be two subspaces of $\R^n$ defined by
     $$ C^n_k(0) := \{ (x_1,\cdots, x_n)\in \R^n\,|\, 0\leq x_1,\cdots, x_k < 1, 
         -1 < x_{k+1},\cdots, x_n < 1 \}. $$
    $$ C^n_k(-1) := \{ (x_1,\cdots, x_n)\in \R^n \,|\, -1\leq x_1,\cdots, x_k < 1, 
         -1 < x_{k+1},\cdots, x_n < 1 \}. $$    
            
    There is a strong deformation retraction from $C^n_k(-1)$ to $C^n_k(0)$ 
    defined by 
    \begin{equation*} 
      H(x_1,\cdots,x_n,t) = \big(\delta_{x_1}(t)\cdot x_1, \cdots, \delta_{x_k}(t)\cdot x_k, x_{k+1},\cdots, x_n \big),
     \end{equation*}
 for any $(x_1,\cdots, x_n)\in C^n_k(-1)$ and $t\in [0,1]$, 
     where 
     $$\delta_x(t) =   \begin{cases}
   1-t ,  &  \text{if $x< 0$}; \\
   1 ,  &  \text{if $x\geq 0$}.
 \end{cases} $$
 
  It is easy to see that for any $t\in [0,1]$, the image of $H(\ ,t)$ is
 $$ C^n_k(t-1) = \{ (x_1,\cdots, x_n)\,|\, t-1\leq x_1,\cdots, x_k < 1, 
         -1 < x_{k+1},\cdots, x_n < 1 \}. $$
  So $H$ actually defines an isotopy from $C^n_k(-1)$ to $C^n_k(0)$ (see Figure~\ref{p:Isotopy}).  
    \end{exam}  
      
   \begin{figure}
         \includegraphics[width=0.44\textwidth]{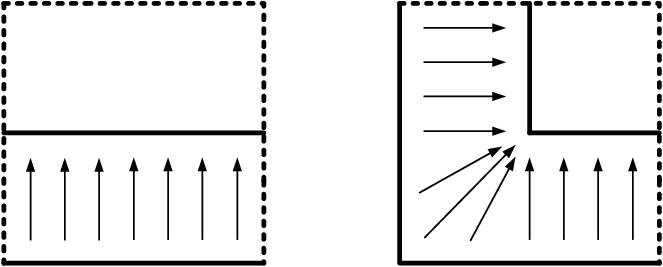}\\
          \caption{Isotopy from $C^n_k(-1)$ to $C^n_k(0)$}\label{p:Isotopy}
      \end{figure}

  \subsection{Embedding $\mathcal{Z}_Q$ into $Q\times (D^2)^m$}
  \ \vskip .1cm
  
  Using the above rim-cubicalization of $Q$ in $Q\times [0,1]^m$, we can embed the manifold $\mathcal{Z}_Q$ into $Q\times (D^2)^m$ where $D^2 =\{ z\in \mathbb{C}\,|\,\| z \| \leq 1\}$ is the unit disk.
   
   In the following, we consider $[0,1]$ as a subset of $D^2$ and the cube $[0,1]^m$ as a subset of $(D^2)^m\subset \mathbb{C}^m$.    
  For any $j\in [m]$, let $S^1_{(j)}$ and $D^2_{(j)}$ denote the corresponding spaces indexed by $j$.
  
  There is a canonical action of $(S^1)^m$ on $Q\times (D^2)^m$ defined by
  $$(g_1,\cdots, g_m)\cdot (x, z_1,\cdots, z_m) = (x,g_1z_1,\cdots, g_mz_m)$$
  where $x\in Q$, $g_j\in S^1_{(j)}$ and $z_j\in D^2_{(j)}$ for $1\leq j \leq m$. The orbit space of this action can be identified with $Q\times [0,1]^m$. We denote the quotient map by
  $$ p: Q\times (D^2)^m\rightarrow Q\times [0,1]^m.  $$
  
  For any face $f$ of $Q$, we define
    \begin{equation} \label{Equ:D2S1-Sigma-I}
      (D^2,S^1)^f := p^{-1}(\widehat{f}) = f
       \times \prod_{j\in I_f} D^2_{(j)} 
   \times \prod_{j\in [m]\backslash I_f} S^1_{(j)}\,\subseteq \, Q\times (D^2)^m,
   \end{equation}
     \begin{align} \label{Equ:D2S1-Q-Hat}
  (D^2,S^1)^{Q} 
  &:= \bigcup_{f\in \mathcal{S}_Q} (D^2,S^1)^f = \bigcup_{f\in \mathcal{S}_Q} \Big(
  f \times \prod_{j\in I_f} D^2_{(j)} 
   \times \prod_{j\in [m]\backslash I_f} S^1_{(j)} \Big).
   \end{align}
   
  There is a canonical action of $(S^1)^m$
   on $(D^2,S^1)^{Q}$ induced by the canonical action of $(S^1)^m$ on $Q\times (D^2)^m$. 
 
     \begin{lem} \label{Lem:Homeo-D2S1-Q-hat-ZQ}
     $(D^2,S^1)^{Q}$ is equivariantly homeomorphic to
     $\mathcal{Z}_Q$.
   \end{lem}
 \begin{proof}
      By Lemma~\ref{Lem:Homeo-Q-Q-hat}, it is equivalent to show that $(D^2,S^1)^{Q}$ is equivariantly homeomorphic to 
      $\mathcal{Z}_{\widehat{Q}}$. 
     We consider $[0,1)^m$ as a nice manifold with corners whose facets are $F^{\square}_1,\cdots, F^{\square}_m$ where
    $$ F^{\square}_i = 0_{(i)}\times \prod_{j\in [m]\backslash \{i\}} [0,1)_{(j)}, \ 1\leq i \leq m.$$
   
   It is clear that
  $\mathcal{Z}_{[0,1)^m} = [0,1)^m\times (S^1)^m\slash \sim$ is homeomorphic to $(D^2\backslash S^1)^m$.
  The quotient map $[0,1)^m\times (S^1)^m\rightarrow \mathcal{Z}_{[0,1)^m} =(D^2\backslash S^1)^m $ extends to a map $\pi: [0,1]^m\times (S^1)^m\rightarrow (D^2)^m$ which can be written explicitly as
  \begin{align} \label{Equ:Pi-square}
    \pi:  [0,1]^m \times (S^1)^m \ & \xlongrightarrow{}\
      (D^2)^m \\
  \big((t_1,\cdots, t_m), (g_1,\cdots g_m)\big)
   & \longmapsto (g_1t_1,\cdots, g_mt_m). \nonumber
    \end{align}
  Define
  $$\pi_Q = id_Q\times \pi : Q\times [0,1]^m\times (S^1)^m \rightarrow Q\times (D^2)^m$$
Notice that the facets of $\widehat{Q}$ are the intersections of 
  $\widehat{Q}$ with $ Q\times F^{\square}_1, \cdots,
  Q\times F^{\square}_m$:
   $$ \widehat{F}_i = \widehat{Q}\cap  \big(Q\times F^{\square}_i\big),\ 1\leq i \leq m. $$ 
 We can easily check that the restriction of $\pi_Q$
  to $\widehat{Q}\times (S^1)^m$ gives exactly $\mathcal{Z}_{\widehat{Q}}$, i.e.
   $$\mathcal{Z}_{\widehat{Q}} = \pi_Q(\widehat{Q}\times (S^1)^m).$$  
    
  Moreover, for any face $f$ of $Q$, we have
    \begin{align*}
      \pi_Q \big(\widehat{f} \times (S^1)^m \big) &= \pi_Q \Big(
      f
       \times \prod_{j\in I_f} \big( [0,1]_{(j)} \times S^1_{(j)} \big)
   \times \prod_{j\in [m]\backslash I_f} S^1_{(j)} \Big) \\
     &= f
       \times \prod_{j\in I_f} D^2_{(j)} 
   \times \prod_{j\in [m]\backslash I_f} S^1_{(j)}
      =  (D^2,S^1)^f. 
      \end{align*}
      
 So we have a homeomorphism
    \begin{align*}
      \mathcal{Z}_{\widehat{Q}} \cong \pi_Q(\widehat{Q}\times (S^1)^m) &=  \bigcup_{f\in \mathcal{S}_Q} \pi_Q \big( \widehat{f} \times (S^1)^m \big) =
   \bigcup_{f\in \mathcal{S}_Q} (D^2,S^1)^f =
     (D^2,S^1)^{Q}. 
     \end{align*}
    
    Clearly, the above homeomorphism
    is equivariant with respect to the canonical actions
    of $(S^1)^m$ on $\mathcal{Z}_{\widehat{Q}}$
    and $(D^2,S^1)^{Q}$. So the lemma is proved.
  \end{proof}

   \subsection{Viewing $\mathcal{Z}_Q$ as a colimit of CW-complexes}    
    \ \vskip .1cm
    
    By Lemma~\ref{Lem:Homeo-D2S1-Q-hat-ZQ}, studying the stable decomposition of $\mathcal{Z}_Q$ is equivalent to studying that for $(D^2,S^1)^{Q}$. To do the stable decomposition as in~\cite{BBCG10}, we want to first think of $(D^2,S^1)^{Q}$ as the colimit of 
a diagram of CW-complexes over a finite \emph{poset} (partially ordered set). 
The following are some basic definitions (see~\cite{ZiegZiv93}).
   
   \begin{itemize}
    \item Let $CW$ be the category of CW-complexes and continuous maps.
    
    \item  Let $CW_*$ be the category of based CW-complexes and based continuous maps. 
    
     \item A \emph{diagram} $\mathcal{D}$ of CW-complexes or based CW-complexes over a finite poset $\mathcal{P}$ is a functor 
   $$\mathcal{D} : \mathcal{P}\rightarrow CW \ 
   \text{or}\ CW_*$$
    such that for every $p\leq p'$ in $\mathcal{P}$, there is a map 
     $d_{pp'}: \mathcal{D}(p')\rightarrow \mathcal{D}(p)$ with
     $$d_{pp}=id_{\mathcal{D}(p)},\ \ d_{pp'} d_{p'p''} = d_{pp''},\ \forall\, p\leq p'\leq p''.$$

   \item The \emph{colimit} of $\mathcal{D}$ is the space
     \[ \mathrm{colim}(\mathcal{D}):= \Big(\coprod_{p\in \mathcal{P}} \mathcal{D}(p)\Big) \Big\slash \sim \]
     where $\sim$ denotes the equivalence relation generated by requiring that for each $x\in \mathcal{D}(p')$,
     $x\sim d_{pp'}(x)$ for every $p < p'$.  
    \end{itemize}
    
   To think of $(D^2,S^1)^{Q}$ as a colimit of CW-complexes, we need to
  introduce a finer decomposition of $(D^2,S^1)^{Q}$ as follows. By the notations in Section~\ref{Subsec:Rim-Cubical}, for any face $f$ of $Q$ and any subset $L\subseteq I_f\subseteq [m]$, let
   \begin{align} \label{Equ:D2S1-Sigma-I-L}
      (D^2,S^1)^{(f,L)} := f \times \prod_{j\in I_f\backslash L} D^2_{(j)} 
   \times \prod_{j\in [m]\backslash (I_f\backslash L)} S^1_{(j)}.
   \end{align}   
   Clearly,  $ (D^2,S^1)^{(f,L)} \subseteq   (D^2,S^1)^{(f,L')}$
   if and only if $L\supseteq L'$. So we have
     \[
        (D^2,S^1)^f = (D^2,S^1)^{(f,\varnothing)} =\bigcup_{L\subseteq I_f}  (D^2,S^1)^{(f,L)},
     \]
  \begin{equation} \label{Equ:D2S1-Q-Hat-2}
  (D^2,S^1)^{Q} 
  = \bigcup_{f \in \mathcal{S}_Q} (D^2,S^1)^f 
  = \bigcup_{f\in \mathcal{S}_Q}\bigcup_{L\subseteq I_f}  (D^2,S^1)^{(f,L)} .
   \end{equation}

  Corresponding to this decomposition, we define a poset associated to $Q$ by
  \begin{equation}\label{Equ:Poset-Q}
    \mathcal{P}_{Q} = \{ (f,L)\,|\,
    f\in \mathcal{S}_Q, L\subseteq I_f\subseteq[m] \}
    \end{equation}
    where $(f, L)\leq (f', L')$ if and only if
    $f\supseteq f'$ and $I_f\backslash L \supseteq I_{f'}\backslash L'$.
    It follows from the definition~\eqref{Equ:D2S1-Sigma-I-L} that:    
    $$(f, L)\leq (f', L') \Longleftrightarrow  (D^2,S^1)^{(f',L')} \subseteq (D^2,S^1)^{(f, L)}.$$
        Note that $\mathcal{P}_{Q}$ is a finite poset since by our convention $Q$ only has finitely many faces.

 \begin{defi}
 Let $\mathbf{D}: \mathcal{P}_{Q}\rightarrow CW$ be a diagram of CW-complexes where  
  $$\mathbf{D}((f,L)) = (D^2,S^1)^{(f,L)}, \ \forall
  (f,L)\in \mathcal{P}_{Q}. $$
 For any $(f,L)\leq (f',L')\in \mathcal{P}_{Q}$, $d_{(f,L),(f',L')}:
  \mathbf{D}((f',L')) \rightarrow \mathbf{D}((f,L))$ is the natural inclusion. 
 \end{defi} 

 Clearly, $(D^2,S^1)^{Q}$ is the colimit of the diagram $\mathbf{D}$. So we have
\begin{equation} \label{Equ:ZQ-Colimit}
 \mathcal{Z}_{Q} \cong (D^2,S^1)^{Q} = \mathrm{colim}(\mathbf{D}) = 
  \bigcup_{(f,L)\in\mathcal{P}_{Q}}  (D^2,S^1)^{(f,L)}.
\end{equation}
 
 \begin{rem} \label{Rem:Base-Points}
  Here we do not write $(D^2,S^1)^{Q}$ as the colimit of a diagram of based CW-complexes. This is because in general it is not possible to choose a basepoint
in each $(D^2,S^1)^{(f,L)}$ to adapt to the colimit construction of a diagram in $CW_*$.
 \end{rem}
 
 \subsection{Stable decomposition of $\mathcal{Z}_Q$}
 \label{Subsec:Stable-Decomp}
 \ \vskip .1cm
 
 First of all, let us recall a well-known theorem (see~\cite{James55, Whitehead56}) which allows us to decompose the
  Cartesian product of a collection of based CW-complexes
 into a wedge of spaces after doing a suspension. 
 
  Let $(X_i, x_i)$ for $1\leq i\leq m$ be based CW-complexes. For $I =\{ i_1,\cdots, i_k\}  \subseteq [m]$ with $1\leq i_1 < \cdots < i_k \leq m$, define
   \[ \widehat{X}^I = X_{i_1}\wedge \cdots \wedge X_{i_k}\] 
   which is the quotient space of $X^I=X_{i_1}\times \cdots \times X_{i_k}$ by the subspace given by 
    \[ \text{$FW(X^I) =\{ (y_{i_1},\cdots, y_{i_k}) \in X^I\,|\, y_{i_j} \ \text{is the basepoint $x_{i_j}\in X_{i_j}$ for at least one}\ i_j  \}$.}\]
    
 \begin{thm} \label{Thm:Stable-Decomp-Lemma}
  Let $(X_i, x_i)$, $1\leq i\leq m$, be based connected CW-complexes.
   There is a based, natural homotopy equivalence
  $$ h:  \mathbf{\Sigma}(X_1\times \cdots\times X_m)\rightarrow \mathbf{\Sigma}\Big(\bigvee_{\varnothing \neq I \subseteq [m]} \widehat{X}^I \Big) $$
  where $I$ runs over all the non-empty subsets of 
   $[m]$. Furthermore, the map $h$
commutes with colimits.
 \end{thm}
 
  In our proof later, we need a slightly generalized version of Theorem~\ref{Thm:Stable-Decomp-Lemma}. Before that, let us first prove three simple lemmas.
  
  \begin{lem} \label{Lem:Smash-Contract}
 If $(X,x_0)$ and $(Y,y_0)$ are based CW-complexes with $X$ contractible, then $X\wedge Y$ is also contractible.
\end{lem}
 \begin{proof} 
 The deformation retraction from $X$ to $x_0$ naturally induces a deformation retraction from 
 $X\wedge Y = X\times Y \slash (\{x_0\}\times Y)\cup (X\times\{y_0\})$ to its canonical basepoint 
$[(x_0,y_0)]=[(\{x_0\}\times Y)\cup (X\times\{y_0\})]$.
\end{proof}

  \begin{lem} \label{Lem:Suspen-Union}
   Suppose a CW-complex $X$ has $N$ connected components $X_1,\cdots, X_N$. Then there is a homotopy equivalence
   $$ \mathbf{\Sigma}(X)\simeq \mathbf{\Sigma}(X_1)\vee \cdots \vee \mathbf{\Sigma}(X_N) \vee \bigvee_{N-1} S^1. $$
   where $\bigvee_{N-1} S^1$ is the wedge sum of $N-1$ copies of $S^1$.
  \end{lem}
  \begin{proof}
    This follows easily from the definition of reduced suspension.
  \end{proof}
 
 \begin{lem} \label{Lem:Smash-Decomp}
   Let $X_1=X'_1\cup \{x_1\}$ where $X'_1$ is a connected based CW-complex and $x_1\notin X'_1$ is the basepoint of $X_1$.

  \begin{itemize}
   \item[(a)] For any connected based CW-complex $X_2$, 
   there is a homotopy equivalence 
   $$ \mathbf{\Sigma}(X_1\wedge X_2) \simeq \mathbf{\Sigma}(X_2)
  \vee \mathbf{\Sigma}(X'_1\wedge X_2).$$

  \item[(b)] If $X_2 = X'_2\cup \{x_2\}$ where 
  $X'_2$ is a connected based CW-complex and $x_2\notin X'_2$ is the basepoint of $X_2$, then 
   $ X_1\wedge X_2 $ is the disjoint union of $
         X'_1\wedge X'_2 $ and a point represented by
         $\{x_1\}\times \{x_2\}$.
      \end{itemize}   
 \end{lem}
 \begin{proof}
  (a) By the definition of smash product, we have a homeomorphism
     \begin{align*}
       X_1\wedge X_2 &= (X'_1\cup\{ x_1\})\times X_2\slash
          \big( \{x_1\}\times X_2 \cup X'_1\times \{x_2\}\big)\\
          &\cong X'_1\times X_2\slash X'_1\times \{x_2\}.
      \end{align*}
 
   Then we have    
     \begin{align*}
      \mathbf{\Sigma}(X_1\wedge X_2) &= \mathbf{\Sigma}\big( X'_1\times X_2\slash X'_1\times \{x_2\}\big) \\ 
      & \simeq  \mathbf{\Sigma}(X'_1\times X_2) \slash 
       \mathbf{\Sigma}(X'_1\times \{x_2\}) \\
       \text{(by Theorem~\ref{Thm:Stable-Decomp-Lemma})} \  &\simeq   \Big( \mathbf{\Sigma}(X'_1) \vee \mathbf{\Sigma}(X_2)
  \vee \mathbf{\Sigma}(X'_1\wedge X_2) \Big) \Big\slash \mathbf{\Sigma}(X'_1) \\
      &\simeq \mathbf{\Sigma}(X_2)
  \vee \mathbf{\Sigma}(X'_1\wedge X_2) .
    \end{align*}
    
   (b) It follows directly from the definition of smash product. 
   \end{proof}

We can generalize Theorem~\ref{Thm:Stable-Decomp-Lemma}
    to the following form.
  
 \begin{thm} \label{Thm:Generalized-Stable-Decomp}
   Let $(X_i, x_i)$ for $1\leq i\leq m$ be based CW-complexes. Assume that for some $1\leq n \leq m$,  
   \begin{itemize}
    \item $X_i = Y_i \cup \{x_i\}$, $1\leq i \leq n$, where $Y_i$ is a connected CW-complex and $x_i\notin Y_i$. 
    
    \item $X_i$, $n+1\leq i \leq m$, is a connected CW-complex.
    \end{itemize}
    There is a based, natural homotopy equivalence
     which commutes with colimits:
  $$ h:  \mathbf{\Sigma}(X_1\times \cdots\times X_m)\rightarrow \mathbf{\Sigma}\Big(\bigvee_{\varnothing\neq I \subseteq [m]} \widehat{X}^I \Big). $$
 \end{thm}
 \begin{proof}
   For brevity, let $[n_1, n_2]=\{ n_1,\cdots,n_2 \}$
  for any integer $n_1\leq n_2$. Let
   $$ x^I = \{x_{i_1}\}\times \cdots \times \{x_{i_k}\},\ 
   Y^I=Y_{i_1}\times \cdots \times Y_{i_k},\ I=\{i_1,\cdots, i_k\},\, i_1 <\cdots < i_k. $$
   
  There are $2^n$
  connected components in $X^{[m]}=X_1\times\cdots\times X_m$, which are
  $$  \{ x^{[n]\backslash I}\times Y^I \times X^{[n+1,m]} \,|\, I\subseteq [1,n]\}. $$
  We choose a basepoint for each $Y_i$, $1\leq i \leq n$. So by Lemma~\ref{Lem:Suspen-Union}, we have
  \begin{align*}
   \mathbf{\Sigma} ( X_1\times\cdots\times X_m )
   \simeq  \bigvee_{I\subseteq [1,n]}\mathbf{\Sigma}\big(Y^I \times X^{[n+1,m]}\big)  \vee \bigvee_{2^{n}-1} S^1.
  \end{align*}
 
 Since $Y_1,\cdots, Y_n, X_{n+1},\cdots, X_m$ are all connected based CW-complexes, we can apply Theorem~\ref{Thm:Stable-Decomp-Lemma} to each
 $Y^I \times X^{[n+1,m]}$ and
  obtain
  \begin{align} \label{Equ:Product-Decomp-Wedge}
  \mathbf{\Sigma} ( X_1\times\cdots\times X_m )
   \simeq \bigvee_{I\subseteq [1,n]}  
   \underset{J\subseteq [n+1,m]}{\bigvee_{L\cup J\neq\varnothing, L\subseteq I}} 
    \mathbf{\Sigma}(\widehat{Y}^L \wedge \widehat{X}^{J})  \vee \bigvee_{2^{n}-1} S^1.
  \end{align}
 
  On the other hand, for any $I=\{i_1,\cdots, i_k\}\subseteq [1,n]$ and $J\subseteq [n+1,m]$,
   $$\quad \widehat{X}^I\wedge \widehat{X}^{J} = X_{i_1}\wedge\cdots\wedge X_{i_k}\wedge \widehat{X}^{J} =  ( Y_{i_1}\cup \{x_{i_1}\} )\wedge\cdots\wedge (Y_{i_k}\cup \{x_{i_k}\})\wedge \widehat{X}^{J}. $$
   \begin{itemize}
    \item If $J\neq \varnothing$, $\widehat{X}^I\wedge \widehat{X}^{J}$ is a connected CW-complex. Then by iteratively using Lemma~\ref{Lem:Smash-Decomp}(a), we obtain
  \begin{align*}
      \mathbf{\Sigma}(\widehat{X}^I\wedge \widehat{X}^{J}) \simeq \bigvee_{L\subseteq I}   \mathbf{\Sigma}( \widehat{Y}^L\wedge\widehat{X}^{J}).
  \end{align*}
  \item If $J=\varnothing$ and $I\neq \varnothing$, by iteratively using Lemma~\ref{Lem:Smash-Decomp}(b), we can deduce that
   $\widehat{X}^I$ is the disjoint union of 
  $\widehat{Y}^I$ and a point represented by $x^I$.
  So by Lemma~\ref{Lem:Suspen-Union}, 
  $ \mathbf{\Sigma}(\widehat{X}^I ) \simeq 
     \mathbf{\Sigma}(\widehat{Y}^I) \vee S^1$.
  \end{itemize}
  
   So we have
   \begin{align*}
  \quad \underset{H\neq\varnothing}{\bigvee_{H \subseteq [m]}} \mathbf{\Sigma} (\widehat{X}^H) &=  \underset{J \subseteq [n+1,m]}{\bigvee_{I\cup J\neq\varnothing, I\subseteq [1,n]}}
    \mathbf{\Sigma} (\widehat{X}^I\wedge \widehat{X}^J )
    \\
    &= \Big( \underset{\varnothing\neq J \subseteq [n+1,m]}{\bigvee_{I\subseteq [1,n]}}
    \mathbf{\Sigma} (\widehat{X}^I\wedge \widehat{X}^J )
    \Big)
    \vee \Big( \bigvee_{\varnothing\neq I\subseteq [1,n]}
    \mathbf{\Sigma} (\widehat{X}^I) \Big)\\
    &\simeq  \Big( \bigvee_{I\subseteq [1,n]} \underset{\varnothing\neq J \subseteq [n+1,m]}{\bigvee_{L\subseteq I}} \mathbf{\Sigma}(\widehat{Y}^L \wedge \widehat{X}^{J}) \Big) 
    \vee \Big( \bigvee_{\varnothing\neq I\subseteq [1,n]}
    \big( \mathbf{\Sigma} (\widehat{Y}^I) \vee S^1 \big) \Big).
   \end{align*}
   
  By comparing the above expression with~\eqref{Equ:Product-Decomp-Wedge},
  we prove the theorem. 
   \end{proof}
 
 \begin{rem}
   By Theorem~\ref{Thm:Generalized-Stable-Decomp}, it is not hard to see that all the main theorems in~\cite{BBCG10} also hold for based CW-complex pairs $\{(X_i,A_i,a_i)\}^m_{i=1}$ where each of $X_i$ and $A_i$ is either connected or is a disjoint union of a connected CW-complex with its basepoint. In particular, \cite[Corollary 2.24]{BBCG10} also holds for $(D^1, S^0)$.
 \end{rem}

\begin{rem}
  It is possible to extend Theorem~\ref{Thm:Generalized-Stable-Decomp} further to 
  deal with spaces each of which is a disjoint union of a connected CW-complex with finitely many points. But since
  Theorem~\ref{Thm:Generalized-Stable-Decomp} is already
  enough for our discussion in this paper, we leave the
  more generalized statement to the reader.
  \end{rem}

 \begin{defi} \label{Def:rtimes-ltimes}
   For any based CW-complexes $(X,x_0)$ and $(Y, y_0)$, let
    $$  X\rtimes Y := X\times Y \big\slash x_0 \times Y,\ \ X\ltimes Y := X\times Y \big\slash X\times y_0. $$
    
   If each of $X$ and $Y$ is either connected or is a
     disjoint union of a connected CW-complex with its basepoint, there is a homotopy equivalence by Theorem~\ref{Thm:Generalized-Stable-Decomp} 
    \begin{align} 
      \mathbf{\Sigma}(X\rtimes Y) &\simeq  \mathbf{\Sigma}( X\times Y) \big\slash  \mathbf{\Sigma}(x_0 \times Y)\simeq \mathbf{\Sigma}(X)\vee \mathbf{\Sigma}(X\wedge Y) \label{Equ:rtimes-homotopy} \\
      \mathbf{\Sigma}(X\ltimes Y) &\simeq  \mathbf{\Sigma}( X\times Y) \big\slash  \mathbf{\Sigma}(X \times y_0)\simeq \mathbf{\Sigma}(Y)\vee \mathbf{\Sigma}(X\wedge Y).  \label{Equ:ltimes-homotopy}
      \end{align}   
   \end{defi}

 We can further generalize Theorem~\ref{Thm:Stable-Decomp-Lemma} to the following form. We will use the following convention in the
   rest of the paper:\\
     
  \noindent\textbf{Convention:} 
  For any based space $Y$, define $Y\wedge \widehat{X}^I := Y$ when $I=\varnothing$.

 \begin{thm} \label{Thm:Stabel-Decomp-rtimes}
     Let $(X_i, x_i)$, $1\leq i\leq m$ and $(B,b_0)$ be a collection of based CW-complexes where each of $X_i$ and $B$ is either connected or is a
     disjoint union of a connected CW-complex with its basepoint.
    Then 
    there is a based, natural homotopy equivalence which commutes with colimits:
  $$ h:  \mathbf{\Sigma} \big(B\rtimes (X_1\times \cdots\times X_m) \big) \rightarrow \mathbf{\Sigma}\Big(\bigvee_{I \subseteq [m]} B\wedge \widehat{X}^I \Big). $$
 \end{thm}
 \noindent \textit{Proof.}
  By definition, we have
  \begin{align*}
   \mathbf{\Sigma}\big(B\wedge (X_1\times\cdots\times X_m)\big) &=  \mathbf{\Sigma}\big( B\times (X_1\times\cdots\times X_m) \big\slash
    B\vee (X_1\times\cdots\times X_m) \big) \\
    &\simeq  \mathbf{\Sigma}\big( B\times X_1\times\cdots\times X_m \big) \big \slash
    \mathbf{\Sigma}(B) \vee   \mathbf{\Sigma}(X_1\times\cdots\times X_m)\\
  \text{(by Lemma~\ref{Thm:Generalized-Stable-Decomp})} \, &\simeq \bigvee_{\varnothing\neq I \subseteq [m]} \mathbf{\Sigma}\big( B\wedge \widehat{X}^I \big).
  \end{align*}
  Then by~\eqref{Equ:rtimes-homotopy}, we have
  \begin{align*}
  \mathbf{\Sigma} \big(B\rtimes (X_1\times \cdots\times X_m) \big) &= \mathbf{\Sigma}(B) \vee \mathbf{\Sigma}(B\wedge (X_1\times\cdots\times X_m)) \\
   &\simeq \mathbf{\Sigma}(B) \vee \bigvee_{\varnothing\neq I \subseteq [m]} \mathbf{\Sigma}\big( B\wedge \widehat{X}^I \big) \\
   &\simeq \mathbf{\Sigma}\Big(\bigvee_{I \subseteq [m]} B\wedge \widehat{X}^I \Big). \qquad\qquad \qquad\qquad \qed
    \end{align*}

 To apply the above stable decomposition lemmas to $(D^2,S^1)^{Q}$, we need to choose a basepoint for each
 $(D^2,S^1)^{(f, L)}$ in the first place. But by Remark~\ref{Rem:Base-Points}, there is no good way to choose a basepoint inside each
  $(D^2,S^1)^{(f,L)}$ to adapt to the colimit construction
  of $(D^2,S^1)^{Q}$.  So in the following, we add an auxiliary point to all $(D^2,S^1)^{(f,L)}$ as their common basepoint.
 
 \begin{itemize}
  \item Let $1_{(j)}$ be the basepoint of $S^1_{(j)}$ and 
 $D^2_{(j)}$ for every $j\in [m]$. 
 
 \item Let $Q_+=Q\cup q_0$ where $q_0 \notin Q$ is the basepoint of $Q_+$.
 
 \item For any face $f$ of $Q$, let $f_+ = f\cup q_0$ with basepoint $q_0$.
 
 \item For any $(f,L)\in \mathcal{P}_{Q}$, define 
   \[
   \qquad   (D^2,S^1)^{(f,L)}_+ :=   (D^2,S^1)^{(f,L)} \cup \widehat{q}_0, \ \text{where}\ 
   \widehat{q}_0 = q_0\times \prod_{j\in [m]} 1_{(j)}\ \,\text{is the basepoint}.
 \]
   
 \item Let
 $(D^2,S^1)^{Q}_+ = (D^2,S^1)^{Q}
 \cup \widehat{q}_0$ with basepoint $\widehat{q}_0$.
  \end{itemize}

 Let $\mathbf{D}_+: \mathcal{P}_{Q}\rightarrow CW_*$ be a diagram of based CW-complexes where  
  $$\mathbf{D}_+((f,L)) = (D^2,S^1)^{(f,L)}_+, \ \forall
  (f,L)\in \mathcal{P}_{Q}, $$
 and $(d_+)_{(f,L),(f',L')}:
  \mathbf{D}_+((f',L')) \rightarrow \mathbf{D}_+((f,L))$ is the natural inclusion for any $(f,L)\leq (f',L')\in \mathcal{P}_{Q}$. Then it is clear that
      \begin{equation} \label{Equ:Extra}
   (D^2,S^1)^{Q}_+ = (D^2,S^1)^{Q}
 \cup \widehat{q}_0 =  \mathrm{colim}(\mathbf{D})
   \cup \widehat{q}_0 = \mathrm{colim}(\mathbf{D}_+). 
   \end{equation}

 Next, we analyze the reduced suspension $\mathbf{\Sigma}\big(\mathrm{colim}(\mathbf{D}_+)\big)$ from the colimit point of view. 
  Since 
  all the $(D^2,S^1)^{(f,L)}_+$ share the same basepoint $\widehat{q}_0$, we have
  $$ \mathbf{\Sigma}(\mathrm{colim}(\mathbf{D}_+))=
   \mathbf{\Sigma}\Big(\bigcup_{(f,L)\in\mathcal{P}_{Q}} (D^2,S^1)^{(f,L)}_+ \Big) = \bigcup_{(f,L)\in\mathcal{P}_{Q}}   \mathbf{\Sigma}\big((D^2,S^1)^{(f,L)}_+\big) $$
     
   \begin{lem} \label{Lem:Stable-Decomp-plus}
   For any $(f,L)\in \mathcal{P}_{Q}$, there is a natural homeomorphism which commutes with taking the colimit:
    $$ (D^2,S^1)^{(f,L)}_+ \cong 
   f_+ \rtimes \Big(  \prod_{j\in I_f\backslash L} D^2_{(j)} 
   \times \prod_{j\in [m]\backslash (I_f\backslash L)} S^1_{(j)} \Big) . $$
   \end{lem}
   \begin{proof}
    By our definitions, 
     \begin{align*}
    &\quad \  f_+ 
       \rtimes \prod_{j\in I_f\backslash L} D^2_{(j)} 
   \times \prod_{j\in [m]\backslash (I_f\backslash L)} S^1_{(j)} \\
   &= \big( f \cup q_0  \big)
       \times \prod_{j\in I_f\backslash L} D^2_{(j)} 
   \times \prod_{j\in [m]\backslash (I_f\backslash L)} S^1_{(j)} \bigg\slash q_0\times \prod_{j\in I_f\backslash L} D^2_{(j)} 
   \times \prod_{j\in [m]\backslash (I_f\backslash L)} S^1_{(j)} \\
   &\cong \Big( f 
       \times \prod_{j\in I_f\backslash L} D^2_{(j)} 
   \times \prod_{j\in [m]\backslash (I_f\backslash L)} S^1_{(j)} \Big) \cup \widehat{q}_0 = (D^2,S^1)^{(f,L)}_+.
   \end{align*}
   
  The above homeomorphism ``$\cong$'' is induced by the global homeomorphism
   $$ Q_+\rtimes \prod_{j\in [m]} D^2_{(j)} \rightarrow \Big( Q\times \prod_{j\in [m]} D^2_{(j)}\Big) 
    \cup \widehat{q}_0   $$
    which identifies $q_0\times \prod_{j\in [m]} D^2_{(j)} \Big\slash q_0\times \prod_{j\in [m]} D^2_{(j)}$ with
    $\widehat{q}_0$. The lemma follows.   
 \end{proof}

  Since we assume that each face $f$ of $Q$ is a CW-complex in our convention, we can deduce from Theorem~\ref{Thm:Stabel-Decomp-rtimes} and Lemma~\ref{Lem:Stable-Decomp-plus} that 
  \begin{align}\label{Equ:term-Essential}
   \mathbf{\Sigma} \Big(  (D^2,S^1)^{(f,L)}_+ \Big) 
 &\cong \mathbf{\Sigma}\Big( f_+ 
       \rtimes \Big(\prod_{j\in I_f\backslash L} D^2_{(j)} 
   \times \prod_{j\in [m]\backslash (I_f\backslash L)} S^1_{(j)}\Big)  \Big) \\
    & \simeq \bigvee_{J \subseteq [m]} \mathbf{\Sigma}\Big(  f_+
       \wedge \bigwedge_{j\in J\cap (I_f\backslash L)} D^2_{(j)} 
   \wedge \bigwedge_{j\in  J\backslash (I_f\backslash L)} S^1_{(j)} \Big). \nonumber
\end{align}

 According to~\eqref{Equ:term-Essential}, we define a family of diagrams of based CW-complexes 
 $$ \widehat{\mathbf{D}}^J_+:  \mathcal{P}_{Q}\rightarrow CW_*, \ J \subseteq [m]$$
   \begin{equation} \label{Equ:E-hat-J-Sigma-I-L}
   \widehat{\mathbf{D}}^J_+((f,L)) := f_+
       \wedge \bigwedge_{j\in J\cap (I_f\backslash L)} D^2_{(j)} 
   \wedge \bigwedge_{j\in  J\backslash (I_f\backslash L)} S^1_{(j)} \, \subseteq Q_+\wedge \bigwedge_{j\in J} D^2_{(j)}.
    \end{equation}
where $(\widehat{d}^J_+)_{(f,L),(f',L')}: 
 \widehat{\mathbf{D}}^J_+((f',L'))\rightarrow  \widehat{\mathbf{D}}^J_+((f,L))$
 is the natural inclusion
  for any $(f,L) \leq (f',L')\in \mathcal{P}_{Q}$. The basepoint of 
 $ \widehat{\mathbf{D}}^J_+((f,L)) $ is
 $$[\widehat{q}^J_0] := \Big[ q_0\times \prod_{j\in J} 1_{(j)}\Big].$$

Since here the reduced suspension commutes with colimits up to homotopy equivalence (see~\cite[Theorem 4.3]{BBCG10}), we obtain a homotopy equivalence
\begin{equation} \label{Equ:Homotopy-Colimits-1}
    \mathbf{\Sigma} \big( \mathrm{colim}(\mathbf{D}_+) \big)
\simeq \mathrm{colim}(\mathbf{\Sigma}(\mathbf{D}_+)) \simeq 
  \bigvee_{J\subseteq [m]} \mathbf{\Sigma}\big( \mathrm{colim}(\widehat{\mathbf{D}}^J_+)\big).
\end{equation}

The following theorem from~\cite{BBCG10} will be useful in our proof of
 Theorem~\ref{Thm:Stable-Decomp-Main}.
 It is a modification of the ``Homotopy Lemma'' given in~\cite{Vogt73, ZiegZiv93, WelkZiegZiv99}.

\begin{thm}[Corollary 4.5 in~\cite{BBCG10}] \label{Thm:Homotopy-Colimit-Lemma}
 Let $D$ and $E$ be two diagrams over a finite poset $\mathcal{P}$ with values in $CW_*$ for which the maps
 $\mathrm{colim}_{q > p} D(q)\hookrightarrow D(p)$, and
 $\mathrm{colim}_{q > p} E(q)\hookrightarrow E(p)$ are all 
 closed cofibrations. If $f$ is a map of diagrams over $\mathcal{P}$ such that for every $p\in\mathcal{P}$,
 $f_p: D(p)\rightarrow E(p)$ is a homotopy equivalence,
 then $f$ induces a homotopy equivalence 
 $f: \mathrm{colim}(D(P)) \rightarrow \mathrm{colim}(E(P))$.
\end{thm}

 Now we are ready to give a proof of Theorem~\ref{Thm:Stable-Decomp-Main}.

\begin{proof}[\textbf{Proof of Theorem~\ref{Thm:Stable-Decomp-Main}.}]
   By~\eqref{Equ:Extra} and~\eqref{Equ:Homotopy-Colimits-1}, we obtain a homotopy equivalence
\begin{equation} \label{Equ:Iso-Wedge-Decomp}
 \mathbf{\Sigma}\big( (D^2,S^1)^{Q}_+ \big)   \simeq \bigvee_{J\subseteq [m]} \mathbf{\Sigma}\big( \mathrm{colim}(\widehat{\mathbf{D}}^J_+)\big)  
  \end{equation}

Notice that when $J\cap (I_f\backslash L) \neq \varnothing$,
 $\widehat{\mathbf{D}}^J_+((f,L))$ is contractible by Lemma~\ref{Lem:Smash-Contract}.
 So for any $J\subseteq [m]$, we define another diagram of based CW-complexes 
$$\widehat{\mathbf{E}}^J_+:  \mathcal{P}_{Q}\rightarrow CW_*$$ 
   \begin{equation}\label{Equ:F-Hat-J}
       \widehat{\mathbf{E}}^J_+((f,L)) :=
        \begin{cases}
  \widehat{\mathbf{D}}^J_+((f,L)) = f_+
       \wedge \underset{j\in  J}{\scalebox{1.3}{$\bigwedge$}} S^1_{(j)},  &  \text{if $J\cap (I_f\backslash L)=\varnothing$}; \\
  \quad   [\widehat{q}^J_0],  &  \text{if $J\cap (I_f\backslash L)\neq \varnothing$}.
 \end{cases} 
    \end{equation}
For $(f,L) \leq (f',L')\in \mathcal{P}_{Q}$, 
$(\widehat{e}^J_+)_{(f,L),(f',L')}: 
 \widehat{\mathbf{E}}^J_+((f',L'))\rightarrow  \widehat{\mathbf{E}}^J_+((f,L))$ is either the natural inclusion or the constant map $\mathbf{c}_{[\widehat{q}^J_0]}$ (mapping all points to $[\widehat{q}^J_0]$). The basepoint of $\widehat{\mathbf{E}}^J_+((f,L))$ is $[\widehat{q}^J_0]$.
 
  Moreover, let $\Phi^J: 
  \widehat{\mathbf{D}}^J_+\rightarrow \widehat{\mathbf{E}}^J_+$ be a map of diagrams over $\mathcal{P}_{Q}$ defined by:
   $$  \Phi^J_{(f,L)} : \widehat{\mathbf{D}}^J_+((f,L)) \rightarrow  \widehat{\mathbf{E}}^J_+((f,L)) $$
   $$ \Phi^J_{(f,L)} =  \begin{cases}
    id_{\widehat{\mathbf{D}}^J_+((f,L))} ,  &  
    \text{if $J\cap (I_f\backslash L)=\varnothing$}; \\
    \mathbf{c}_{[\widehat{q}^J_0]},  &  \text{if $J\cap (I_f\backslash L)\neq\varnothing$}.
 \end{cases}  $$

   Then by Theorem~\ref{Thm:Homotopy-Colimit-Lemma}, there exists a homotopy equivalence:
  $$\mathrm{colim}\big(\widehat{\mathbf{D}}^J_+\big) \simeq
    \mathrm{colim}\big(\widehat{\mathbf{E}}^J_+\big).$$
    
    Note that we always have 
 $\widehat{\mathbf{D}}^J_+((f,I_f))\subseteq \widehat{\mathbf{D}}^J_+((f,L))$ for any $L\subsetneq I_f \neq \varnothing$.
 So we can ignore the terms $\{ \widehat{\mathbf{D}}^J_+((f,I_f))\,|\, I_f\neq \varnothing, f\in \mathcal{S}_Q \}$ when computing
 $\mathrm{colim}\big(\widehat{\mathbf{D}}^J_+\big)$. 
 If $I_f =\varnothing$, then $f\subseteq Q^{\circ}$ and 
 $\widehat{\mathbf{D}}^J_+((f,L)) = \widehat{\mathbf{D}}^J_+((f,\varnothing)) = f_+
       \wedge \bigwedge_{j\in  J} S^1_{(j)} $.

 To understand $\mathrm{colim}\big(\widehat{\mathbf{E}}^J_+\big)$, we need to figure out in~\eqref{Equ:F-Hat-J} what are 
 those faces $f$ of $Q$ with some $L\subsetneq I_f$ such that $J\cap (I_f\backslash L) \neq \varnothing$.

  \begin{itemize}
   \item   
  There exists $L\subsetneq I_f$ with
 $J\cap (I_f\backslash L) \neq \varnothing$ if and only if $J\cap I_f\neq \varnothing$, which is equivalent to $f\subseteq F_J$.
  Conversely,  we have
$$ \qquad\ F_J =\bigcup_{f \in \mathcal{S}_Q} f\cap F_J =
\bigcup_{f \in \mathcal{S}_Q} \bigcup_{j\in J} f\cap F_j=
\bigcup_{f \in \mathcal{S}_Q} \bigcup_{j\in J\cap I_f} f\cap F_j \subseteq \bigcup_{f \in \mathcal{S}_Q} \bigcup_{J\cap I_f\neq \varnothing} f. $$ 
 This implies
 \begin{equation} \label{Equ:I-J-1}
  \bigcup_{f \in \mathcal{S}_Q} \underset{J\cap (I_f\backslash L)\neq \varnothing}{\bigcup_{\exists L\subsetneq I_f}} f = F_J.
  \end{equation} 
 
  \item There exists $L\subsetneq I_f$ with
 $J\cap (I_f\backslash L) = \varnothing$
 if and only if $f \subseteq F_{[m]\backslash J}$. So
 \begin{equation} \label{Equ:I-J-2}
   \bigcup_{f \in \mathcal{S}_Q} \underset{J\cap (I_f\backslash L)=\varnothing}{\bigcup_{\exists L\subsetneq I_f}} f  =  F_{[m]\backslash J}.
    \end{equation}
 \end{itemize}

 The above discussion implies:
 \begin{equation} \label{Equ:I-J-Com}
   \bigcup_{f \in \mathcal{S}_Q} \underset{J\cap (I_f\backslash L')\neq \varnothing}{\underset{J\cap (I_f\backslash L)= \varnothing}{\bigcup_{\exists L, L'\subsetneq I_f}}} f  =   F_{[m]\backslash J} \cap F_J.    
 \end{equation}
 
By the definition of $ \widehat{\mathbf{E}}^J_+$, if we have a face $f$ of $Q$ and two subsets $L,L'\subsetneq I_f$
such that $J\cap (I_f\backslash L)= \varnothing$ while $J\cap (I_f\backslash L') \neq \varnothing$,
then $J\cap (I_f\backslash (L\cup L')) = \varnothing$ and $J\cap (I_f\backslash (L\cap L')) \neq \varnothing$. So in this case, we have 
\begin{align*}
 &\bullet\, (\widehat{e}^J_+)_{(f,L'),(f,L\cup L')}:
 \widehat{\mathbf{E}}^J_+((f, L\cup L')) 
  \longrightarrow \widehat{\mathbf{E}}^J_+((f,L'))\ \text{is the constant map}\
   \mathbf{c}_{[\widehat{q}^J_0]}.\\
 &\qquad\qquad\qquad\qquad\quad \  \ f_+
       \wedge \bigwedge_{j\in  J} S^1_{(j)}\   
       \longrightarrow \ \,  [\widehat{q}^J_0] \\
 & \bullet\, (\widehat{e}^J_+)_{(f,L),(f,L\cup L')}:
 \widehat{\mathbf{E}}^J_+((f, L\cup L')) 
  \longrightarrow \widehat{\mathbf{E}}^J_+((f,L))\  \text{is identity map.}\\
 &\qquad\qquad\qquad\quad \qquad\   f_+
       \wedge \bigwedge_{j\in  J} S^1_{(j)}    \    \longrightarrow \  f_+
       \wedge \bigwedge_{j\in  J} S^1_{(j)}            
 \end{align*} 
Then in $\mathrm{colim}\big(\widehat{\mathbf{E}}^J_+\big)$,
the image of any of such $ f_+
       \wedge \bigwedge_{j\in  J} S^1_{(j)} $ is equivalent to the point
$[\widehat{q}^J_0]$. So we can deduce
\begin{itemize}
 \item For $J\neq \varnothing$,
 \begin{align*} 
 \qquad\ \ \mathrm{colim}\big(\widehat{\mathbf{E}}^J_+\big) \cong
   \Big( & Q^{\circ}\cup\underset{J\cap (I_f\backslash L)=\varnothing}{\bigcup_{\exists L\subsetneq I_f}} f_+ \Big) \wedge  \bigwedge_{j\in  J} S^1_{(j)} \bigg\slash 
    \Big( \underset{J\cap (I_f\backslash L')\neq \varnothing}{\underset{J\cap (I_f\backslash L)= \varnothing}{\bigcup_{\exists L, L'\subsetneq I_f}}} f_+ \Big) \wedge  \bigwedge_{j\in  J} S^1_{(j)}
  \\
   \overset{\eqref{Equ:I-J-2},\eqref{Equ:I-J-Com}}{\cong}& \Big( \big( (Q^{\circ} \cup F_{[m]\backslash J})\cup q_0 \big)  \Big\slash 
    \big( ( F_{[m]\backslash J} \cap F_J ) \cup q_0 \big)  \Big) \wedge  \bigwedge_{j\in  J} S^1_{(j)}  \\
   \cong  \  \big( & Q  \slash   F_J  \big)  \wedge  \bigwedge_{j\in  J} S^1_{(j)}  
   \cong \mathbf{\Sigma}^{|J|} ( Q\slash F_J ).
\end{align*}
\item For $J=\varnothing$,
$\mathrm{colim}\big(\widehat{\mathbf{E}}^{\varnothing}_+\big) \cong  Q^{\circ} \cup F_{[m]}\cup q_0 = Q\cup q_0=Q_+$.
\end{itemize}

 Combining all the above arguments, we obtain
 homotopy equivalences:
 \begin{align*}   
 &\quad\, \, \mathbf{\Sigma}\big( (D^2,S^1)^{Q}_+ \big)   \simeq  \bigvee_{J\subseteq [m]} \mathbf{\Sigma} \big( \mathrm{colim}\big(\widehat{\mathbf{D}}^J_+ \big) \big)
  \simeq 
  \bigvee_{J\subseteq [m]} \mathbf{\Sigma} \big( \mathrm{colim}\big(\widehat{\mathbf{E}}^J_+ \big) \big)\\
  &\simeq \mathbf{\Sigma}(Q_+) \vee \bigvee_{\varnothing \neq J\subseteq [m]}  \mathbf{\Sigma}^{|J|+1} ( Q\slash F_J ) \simeq  S^1\vee \mathbf{\Sigma}(Q) \vee \bigvee_{\varnothing\neq J\subseteq [m]}  \mathbf{\Sigma}^{|J|+1} ( Q\slash F_J ) \\
  &\simeq S^1\vee \bigvee_{J\subseteq [m]}  \mathbf{\Sigma}^{|J|+1} ( Q\slash F_J ).
   \end{align*}
 
   On the other hand, we have 
   $$  \mathbf{\Sigma}\big( (D^2,S^1)^{Q}_+ \big)  = 
   \mathbf{\Sigma}\big( (D^2,S^1)^{Q} \cup \widehat{q}_0 \big)
    \simeq  S^1 \vee
   \mathbf{\Sigma}\big( (D^2,S^1)^{Q} \big) \cong S^1 \vee \mathbf{\Sigma}( \mathcal{Z}_Q ) .$$
   
    Then the theorem follows.
\end{proof}

   \section{Cohomology ring structure of $\mathcal{Z}_Q$}
   \label{Sec:Cohomology}
   
   The cohomology ring of the moment-angle
complex over a simplicial complex $K$ was computed in Franz~\cite{Franz06} and  Baskakov-Buchstaber-Panov~\cite{BasBuchPanov04}.
The cohomology rings of a much wider class of spaces called
\emph{generalized moment-angle complexes} or \emph{polyhedral products} were computed in Bahri-Bendersky-Cohen-Gitler~\cite{BBCG12} via partial diagonal maps and in Bahri-Bendersky-Cohen-Gitler~\cite{BBCG17}
by a spectral sequence 
under certain freeness conditions (coefficients in a field for example).
The study in this direction is further extended in~\cite{BBCG20}.
A computation using different methods was carried out in Wang-Zheng~\cite{WangZheng15} and Zheng~\cite{Zheng16}.

 It was shown in Bahri-Bendersky-Cohen-Gitler~\cite{BBCG12} that the product structure on the cohomology of a polyhedral product over a simplicial complex can be formulated
 in terms of the stable decomposition and partial diagonal maps of the polyhedral product.
For a nice manifold with corners $Q$,
since we also have the stable decomposition of $\mathcal{Z}_Q$, we should be able to describe the cohomology ring of 
$\mathcal{Z}_Q$ in a similar way.

  Let us first recall the definition of partial diagonal in product spaces from~\cite{BBCG12}.
  Let $X_1,\cdots, X_m$ be a collection of based CW-complexes.
  Using the notations in Section~\ref{Subsec:Stable-Decomp}, for any $I\subseteq [m]$, there are natural projections $X^{[m]}\rightarrow \widehat{X}^I$ obtained as the composition
  \[ \widehat{\Pi}_I: X^{[m]} \overset{\Pi_I}{\longrightarrow} X^I
  \overset{\rho_I}{\longrightarrow} \widehat{X}^I
   \]
   where $\Pi_I : X^{[m]}\rightarrow X^I$ is the natural projection and $\rho_I$ is the quotient map
   in the definition of the smash product $\widehat{X}^I$. In addition, let
    $$ W^{J,J'}_I := \bigwedge_{|J|+|J'|} W_i,\ J\cup J' =I,\ \text{where} $$   
     $$   W_i =  \begin{cases}
    X_i ,  &  \text{if $i\in I\backslash (J\cap J')$}; \\
    X_i\wedge X_i ,  &  \text{if $i\in J\cap J'$}.
 \end{cases} $$
  
   Note that if $J\cup J'=I$ and $J\cap J'=\varnothing$, $ W^{J,J'}_I = \widehat{X}^I$. 
   
   Define
     $$ \psi^{J,J'}_{I} : \widehat{X}^I \rightarrow W^{J,J'}_I\ \,\text{as}\ \, \psi^{J,J'}_I=\bigwedge_{i\in I}\psi_i $$
     where $\psi_i: X_i\rightarrow W_i$ is defined by
     $$ \psi_i =  \begin{cases}
    id ,  &  \text{if $i\in I\backslash (J\cap J')$}; \\
    \Delta_i: X_i\rightarrow X_i\wedge X_i ,  &  \text{if $i\in J\cap J'$}.
 \end{cases}  $$
 where $\Delta_i : X_i\rightarrow X_i\times X_i\rightarrow X_i\wedge X_i $ is the \emph{reduced diagonal} of $X_i$.
 
  Note that the smash products $W^{J,J'}_I$ and $\widehat{X}^J\wedge \widehat{X}^{J'}$ have the same factors, but in a different order arising from the natural shuffles. Let
   \begin{equation} \label{Equ:Shuffle}
      \Theta^{J,J'}_{I}: W^{J,J'}_I \rightarrow \widehat{X}^J\wedge \widehat{X}^{J'}, \, J\cup J'=I, 
      \end{equation}
   be the natural homeomorphism given by a shuffle. Define
   the \emph{partial diagonal} 
   \begin{equation} \label{Equ:Partial-Diagonal}
     \widehat{\Delta}^{J,J'}_{I} :
     \widehat{X}^I \xlongrightarrow{\ \psi^{J,J'}_{I}} W^{J,J'}_I \xlongrightarrow{\ \Theta^{J,J'}_{I}} \widehat{X}^J\wedge \widehat{X}^{J'} 
    \end{equation}
    be the composition of $\Theta^{J,J'}_{I}$ and $ \psi^{J,J'}_{I}$.    
    There is a commutative diagram
    \[   \xymatrix{
          X^{[m]}\ar[d]_{\widehat{\Pi}_I} \ar[r]^{\Delta^X_{[m]}\quad \ }
                & X^{[m]}\wedge X^{[m]} \ar[d]^{\widehat{\Pi}_J\wedge \widehat{\Pi}_{J'}}  \\
            \widehat{X}^I  \ar[r]^{ \widehat{\Delta}^{J,J'}_{I}\ \ } & \widehat{X}^J\wedge \widehat{X}^{J'} 
                 }  \] 
  where $\Delta^X_{[m]}$ is the reduced diagonal map of $X^{[m]}$. 
  
  Let $\mathbf{k}$ denote a commutative ring with a unit. For any $J\subseteq [m]$, there is a homomorphism of rings given by the \emph{reduced cross product} $\times$ (see~\cite[p.\,223]{Hatcher02}):
  \[  \bigotimes_{j\in J} \widetilde{H}^*(X_j;\mathbf{k})
   \xlongrightarrow{\ \, \normaltimes\ \, } \widetilde{H}^*(\widehat{X}^{J};\mathbf{k}). \]
 In particular, this ring homomorphism becomes a ring isomorphism if all (possibly except one)
  $\widetilde{H}^*(X_j;\mathbf{k})$ are free $\mathbf{k}$-modules (see~\cite[Theorem\,3.21]{Hatcher02}).

\begin{lem} \label{Lem:Cross-Product-Pull-Back}
  For any $\phi_j \in \widetilde{H}^*(X_j;\mathbf{k})$,
  $j\in J$ and any $\phi'_j \in \widetilde{H}^*(X_j;\mathbf{k})$, $j\in J'$, 
  $$ \big( \widehat{\Delta}^{J,J'}_{I} \big)^*: 
  H^*\big(\widehat{X}^J\wedge \widehat{X}^{J'};\mathbf{k} \big) \longrightarrow H^*\big(\widehat{X}^I;\mathbf{k}\big),\  I=J\cup J', $$
  $$ \big( \widehat{\Delta}^{J,J'}_{I} \big)^* \Big(\big( \underset{j\in J}{\largetimes} \phi_j \big) \times 
  \big( \underset{j\in J'}{\largetimes} \phi'_j \big) \Big)=
  \big( \underset{j\in J\backslash J'}{\largetimes} \phi_j \big) \times  \big( \underset{j\in J'\backslash J}{\largetimes} \phi'_j \big) \times \big( 
  \underset{j\in J\cap J'}{\largetimes} \Delta^*_j(\phi_j \times \phi'_j)
  \big). $$
  \end{lem}
  \begin{proof}
   The above formula follows easily from the definition of $ \widehat{\Delta}^{J,J'}_{I}$. 
 Note that the shuffle $\Theta^{J,J'}_{I}$ (see~\eqref{Equ:Shuffle}) sorts all the
 cohomology classes $\{\phi_j\}_{j\in J}$ and $\{\phi'_j\}_{j\in J'}$ in order without introducing any $\pm$ sign. This is because for any space $X$ and $Y$,  
 \begin{align*}
   T: X\wedge Y &\longrightarrow Y\wedge X
   \\
    [(x,y)]&\longmapsto [(y,x)]
 \end{align*}
 induces a group isomorphism $T^*: H^*(Y\wedge X;\mathbf{k}) \longrightarrow H^*(X\wedge Y;\mathbf{k})$ such that
   \begin{align*}
    T^*(\phi_Y \times \phi_X) = \phi_X\times \phi_Y,\ 
     \phi_X \in H^*(X;\mathbf{k}), \phi_Y \in H^*(Y;\mathbf{k}). 
 \end{align*}
 So when $\Theta^{J,J'}_{I}$ transposes the space factors, 
 the cohomology classes in the reduced cross product are transposed accordingly.
 \end{proof} 
 
 The following lemma will be useful for our proof of
 Theorem~\ref{Thm:Cohomology-Ring-Isom} later.
 
 \begin{lem} \label{Lem:Induced-Cup-Prod}
   Let $X$ be a CW-complex and $A ,B$ be two subcomplexes of $X$. The relative cup product 
   $$ H^*(X,A)\otimes H^*(X,B) \overset{\cup}{\longrightarrow} H^*(X,A\cup B) $$ 
   induces a product 
   $$\widetilde{H}^*(X\slash A)\otimes \widetilde{H}^*(X\slash B) \overset{\widetilde{\cup}}{\longrightarrow} \widetilde{H}^*(X\slash ( A\cup B)),$$ 
   which can be factored as 
   $$ \phi\, \widetilde{\cup}\, \phi' = \Delta^*_X (\phi \times \phi'),\ \phi\in \widetilde{H}^*(X\slash A),
    \phi'\in \widetilde{H}^*(X\slash B) $$
    where $\Delta_X : X\rightarrow X\times X$ is the diagonal map and $\phi\times\phi'$ is the reduced cross product of $\phi$ and $\phi'$. 
 \end{lem}
 \begin{proof}
    This can be verified directly from the following
    diagram when $A,B$ are nonempty.
    \begin{equation*}\tiny
       \xymatrix{
           H^*(X,A)\otimes H^*(X,B)  \ar[r]^{\scalebox{1}{$\times$}\qquad}_{\text{relative cross product}\qquad\ }
                & H^*\big( X\times X, (A\times X)\cup (X\times B) \big)  \ar[r]^{\qquad\quad\ \
                \, \Delta^*_X} &
                 H^*(X,A\cup B)  \\
  \underset{\scalebox{1}{ $\widetilde{H}^*(X\slash A)$}}{  
   \underset{ \scalebox{1}{$\rotatebox{270}{$\cong$}$} }{H^*(X\slash A,A\slash A)} }\otimes  
       \underset{\scalebox{1}{ $\widetilde{H}^*(X\slash B)$}}{  
   \underset{ \scalebox{1}{$\rotatebox{270}{$\cong$}$} }{ H^*(X\slash B, B\slash B)}} \ar[u]_{\scalebox{1}{$\cong$}} \ar[r]^{\scalebox{1}{$\times$}\qquad\   
       \quad } & 
      \underset{\scalebox{1}{ $\widetilde{H}^*(X\slash A\wedge X\slash B)$}}{  
   \underset{ \scalebox{1}{$\rotatebox{270}{$\cong$}$} }{  H^*\big(X\slash A\times X\slash B, (A\slash A\times X\slash B)\cup (X\slash A \times B\slash B)\big) }} \ar[u]^{\scalebox{1}{$\cong$}}
       &  \underset{\scalebox{1}{ $\widetilde{H}^*(X\slash (A\cup B))$}}{  
   \underset{ \scalebox{1}{$\rotatebox{270}{$\cong$}$} }{ H^*\big( X\slash (A\cup B), A\cup B\slash A\cup B\big) }}
   \ar[u]^{\scalebox{1}{$\cong$}}
                 } 
    \end{equation*}
    where the lower $\overset{\times}{\longrightarrow}$ is the reduced cross product on $\widetilde{H}^*(X\slash A)\otimes \widetilde{H}^*(X\slash B)$. \vskip .1cm
    
    If $A$ or $B$ is empty, we should replace
    $\widetilde{H}^*(X\slash A)$ or $\widetilde{H}^*(X\slash B)$ by $H^*(X)$ in the above diagram.
    Moreover, since $H^*(X)\cong \widetilde{H}^*(X)\oplus \Z$, the $\widetilde{\cup}$ on
    $\widetilde{H}^*(X)$ is just the restriction of
    $\cup$ from $H^*(X)$. The lemma is proved.  
 \end{proof}

  Another useful fact is when 
 $X_i$ is the suspension of some space, the reduced diagonal $ \Delta_i : X_i \rightarrow X_i \wedge X_i $ is 
 null-homotopic (see~\cite{BBCG12}). So we have the following lemma. 
 \begin{lem} \label{Lem:Reduced-Diagonal-Null}
   If for some $j\in J\cap J'$, $X_j$ is a suspension space, then the partial diagonal 
  $ \widehat{\Delta}^{J,J'}_{I} :
     \widehat{X}^I \rightarrow \widehat{X}^J\wedge \widehat{X}^{J'}$ is null-homotopic, where $I= J\cup J'$.
     \end{lem}

    Now we are ready to give a proof of Theorem~\ref{Thm:Cohomology-Ring-Isom}.
     Our argument is parallel to the argument used in the proof of~\cite[Theorem 1.4]{BBCG12}.
  
 \begin{proof}[\textbf{Proof of Theorem~\ref{Thm:Cohomology-Ring-Isom}}]
  \ \vskip .1cm
  
  For brevity, we will use the following notation in the proof.
   \begin{align*}
         Q_+\times (D^2)^{[m]} &:= Q_+\times \prod_{j\in [m]} D^2_{(j)}; \ \   Q_+\wedge \widehat{D^2}^J :=
          Q_+\wedge \bigwedge_{j\in J} D^2_{(j)}.
    \end{align*}

  Considering the partial diagonals~\eqref{Equ:Partial-Diagonal} for $Q_+\times (D^2)^{[m]}$, we obtain a map 
  $$ \widehat{\Delta}^{J,J'}_{J\cup J',Q_+} :   Q_+\wedge \widehat{D^2}^{J\cup J'} \longrightarrow  \big( Q_+\wedge \widehat{D^2}^{J} \big) \wedge \big( Q_+\wedge \widehat{D^2} ^{J'}\big)$$
  for any $J, J'\subseteq [m]$ and a commutative diagram:
   \begin{equation} \label{Equ:Diagram-Q-D2-m}
       \xymatrix{
          Q_+\times (D^2)^{[m]}\ar[d]_{\widehat{\Pi}_{J\cup J'}} \ar[r]^{\Delta^{Q_+, D^2}_{[m]}\qquad \quad\ \ }
               \, & \, \big( Q_+\times (D^2)^{[m]} \big)\wedge \big( Q_+\times (D^2)^{[m]} \big) \ar[d]^{\widehat{\Pi}_{J}\wedge \widehat{\Pi}_{J'}}  \\
            Q_+\wedge \widehat{D^2}^{J\cup J'}  \ar[r]^{ \widehat{\Delta}^{J,J'}_{J\cup J',Q_+}\ \ \ \qquad }  &\,  \big(Q_+ \wedge \widehat{D^2}^{J}
            \big) \wedge \big(Q_+\wedge \widehat{D^2}^{J'} 
               \big)  }  
   \end{equation}      
       where $\Delta^{Q_+,D^2}_{[m]}$ is the reduced diagonal map of $Q_+\times (D^2)^{[m]}$.   
  By restricting the above diagram to $\mathrm{colim}(\mathbf{D}_+)$,    
   we obtain a commutative diagram for $\forall J , J'\subseteq [m]$:
   \begin{equation} \label{Equ:Diagram-Z-Q}
       \xymatrix{
         \mathrm{colim}(\mathbf{D}_+) \ar[d]_{\widehat{\Pi}_{J\cup J'}}  \ar[r]^{\Delta^{Q_+,D^2}_{[m]}\qquad \ \ }
               \, & \, \mathrm{colim}(\mathbf{D}_+) \wedge \mathrm{colim}(\mathbf{D}_+) \ar[d]^{\widehat{\Pi}_{J}\wedge \widehat{\Pi}_{J'}}  \\
        \mathrm{colim} \big(\widehat{\mathbf{D}}^{J\cup J'}_+\big)  \ar[r]^{ \widehat{\Delta}^{J,J'}_{J\cup J',Q_+} \ \qquad }\, & \, \mathrm{colim} \big(\widehat{\mathbf{D}}^{J}_+\big) \wedge 
        \mathrm{colim} \big(\widehat{\mathbf{D}}^{J'}_+\big).
         }  
   \end{equation}

       Given cohomology classes $u\in \widetilde{H}^*\big( \mathrm{colim} \big(\widehat{\mathbf{D}}^{J}_+\big) \big)$ and $v\in
 \widetilde{H}^*\big(\mathrm{colim} \big(\widehat{\mathbf{D}}^{J'}_+\big)\big)$, let
 \begin{equation}\label{Equ:Ring-Structure-2}
   u  \circledast v = \big( \widehat{\Delta}^{J,J'}_{J\cup J',Q_+} \big)^*
 (u\times v) \in \widetilde{H}^* \big(   \mathrm{colim} \big(\widehat{\mathbf{D}}^{J\cup J'}_+\big) \big) 
 \end{equation}
 where $u\times v\in \widetilde{H}^*\big( \mathrm{colim} \big(\widehat{\mathbf{D}}^{J}_+\big) \wedge 
        \mathrm{colim} \big(\widehat{\mathbf{D}}^{J'}_+\big) \big) $ is the reduced cross product of $u$ and $v$.
 This defines a ring structure on $\bigoplus_{J
 \subseteq [m]} \widetilde{H}^*\big( \mathrm{colim}\big(\widehat{\mathbf{D}}^{J}_+  \big) \big)$.
 
 The commutativity of diagram~\eqref{Equ:Diagram-Z-Q} implies
  \[ \widehat{\Pi}^*_{J\cup J'}(u \circledast v) = \widehat{\Pi}^*_{J}(u)\cup \widehat{\Pi}^*_{J'}(v),  \]
  where $\cup$ is the cup product for 
  $\mathrm{colim}(\mathbf{D}_+) $. 
  
  By~\eqref{Equ:Homotopy-Colimits-1}, the direct sum of $\widehat{\Pi}^*_{J}$ induces an additive isomorphism  
  \begin{equation} \label{Equ:Cohomology-Iso-All}
     \bigoplus_{J\subseteq [m]} \widehat{\Pi}^*_{J} : \bigoplus_{J \subseteq [m]} \widetilde{H}^*\big( \mathrm{colim}\big(\widehat{\mathbf{D}}^{J}_+  \big) \big) \longrightarrow \widetilde{H}^*\big(\mathrm{colim}(\mathbf{D}_+)\big)=\widetilde{H}^*\big((D^2,S^1)^{Q}_+\big).
      \end{equation}
    
 Then since
  $\widehat{\Pi}^*_{J}: \widetilde{H}^*\big( \mathrm{colim} \big(\widehat{\mathbf{D}}^{J}_+\big)  \big)\rightarrow \widetilde{H}^*\big(\mathrm{colim}(\mathbf{D}_+)\big)$
   is a ring homomorphism for every $J\subseteq [m]$, we can assert that $\bigoplus_{J\subseteq [m]} \widehat{\Pi}^*_{J}$ is a ring isomorphism. Then by the proof of Theorem~\ref{Thm:Stable-Decomp-Main}, this induces a ring isomorphism
   \begin{equation} \label{Equ:Cohomology-Ring-Isomor}
    \bigoplus_{J\subseteq [m]} \widehat{\Pi}^*_{J}:
         \Big( \bigoplus_{J\subseteq [m]}\widetilde{H}^*\big(Q\slash F_J   \wedge  \bigwedge_{j\in  J} S^1_{(j)} \big), \circledast \Big) \longrightarrow \widetilde{H}^*\big((D^2,S^1)^{Q}\big).
     \end{equation}

  Finally, let us show 
 how to define a ring isomorphism from $(\mathcal{R}^*_Q,\Cup)$ 
 to the cohomology ring $H^*\big((D^2,S^1)^{Q}\big)$ via $\bigoplus_{J\subseteq [m]} \widehat{\Pi}^*_{J}$.

    For any $1\leq j \leq m$, let $\iota^1_{(j)}$ denote a generator of $\widetilde{H}^1(S^1_{(j)})$. Then for any subset $J=\{j_1,\cdots,j_s\}\subseteq [m]$ with
  $j_1 <\cdots < j_s$, we have a generator
 $$ \iota^J = \iota^1_{(j_1)}\times \cdots\times \iota^1_{(j_s)} \in \widetilde{H}^{|J|}\Big(\bigwedge_{j\in  J} S^1_{(j)})\Big). $$

 For each $J\subseteq [m]$, there is a canonical linear isomorphism (see~\cite[p.\,223]{Hatcher02}):
   \begin{align*} 
      \widetilde{H}^*(Q\slash F_J) &\overset{\cong}{\longrightarrow} \widetilde{H}^*\Big(Q\slash F_J 
        \wedge \bigwedge_{j\in  J} S^1_{(j)}\Big) \cong \widetilde{H}^*\big(\mathbf{\Sigma}^{|J|}(Q\slash F_J)\big)\\
      \phi\ \ &\longmapsto\ \ \phi\times \iota^J. \nonumber
 \end{align*}

   Let   \begin{equation*}
    \widetilde{\mathcal{R}}^*_Q :=\bigoplus_{J \subseteq [m]} \widetilde{H}^*(Q\slash F_J), \ \text{then}\
    \mathcal{R}^*_Q =  \widetilde{\mathcal{R}}^*_Q \oplus \Z.
    \end{equation*} 
    
    By Lemma~\ref{Lem:Induced-Cup-Prod}, there is natural ring structure on $\widetilde{\mathcal{R}}^*_Q$,
    denoted by $\widetilde{\Cup}$,
    that is induced from the product $\Cup$ on $\mathcal{R}^*_Q$ (see~\eqref{Equ:Def-R-Q}).     
    We have a commutative diagram 
     \begin{equation}  \label{Equ:Diagram-Induced-Prod}
   \xymatrix{
           H^*(Q,F_J) \otimes H^*(Q,F_{J'}) \ar[d] \ar[r]^{\quad\ \ \ \ \scalebox{0.85}{$\Cup$}}
                &   H^*(Q,F_{J\cup J'}) \ar[d]  \\
           \widetilde{H}^*(Q\slash F_J) \otimes \widetilde{H}^*(Q\slash F_{J'}) \ar[r]^{\quad\ \ \ \  \, \scalebox{0.85}{$\widetilde{\Cup}$}} & \widetilde{H}^*(Q\slash F_{J\cup J'}) .
                 }  
            \end{equation}      
     \begin{itemize}
     \item For any $J,J'\subseteq [m]$ with $J\cap J'\neq \varnothing$, $\widetilde{\Cup}$ is trivial.
     
     \item For any $J,J'\subseteq [m]$ with $J\cap J=\varnothing$, $\widetilde{\Cup} =\widetilde{\cup}$ is induced from the relative cup product $\cup$ on $H^*(Q,F_J)\otimes H^*(Q,F_{J'})$ (see Lemma~\ref{Lem:Induced-Cup-Prod}).
   \end{itemize}
    
    It is clear that $\big(\mathcal{R}^*_Q, \Cup \big)$
    and $\big(\widetilde{\mathcal{R}}^*_Q, \widetilde{\Cup} \big)$ determine 
    each other. 
    
  \begin{itemize}
   \item When $J\cap J' \neq \varnothing$,
   since $(D^2,S^1) \cong (\Sigma D^1, \Sigma S^0)$ is a
  pair of suspension spaces, Lemma~\ref{Lem:Reduced-Diagonal-Null} implies that
  $$\widehat{\Delta}^{J,J'}_{J\cup J',Q_+}:
   \mathrm{colim} \big(\widehat{\mathbf{D}}^{J\cup J'}_+\big) \rightarrow 
   \mathrm{colim} \big(\widehat{\mathbf{D}}^{J}_+\big) \wedge 
        \mathrm{colim} \big(\widehat{\mathbf{D}}^{J'}_+\big)
        $$
    is null-homotopic. So by~\eqref{Equ:Ring-Structure-2},
    $\circledast$ is trivial in this case which corresponds to the definition of $\widetilde{\Cup}$ on $\widetilde{\mathcal{R}}^*_Q$.

  \item When $J\cap J'=\varnothing$, suppose in~\eqref{Equ:Cohomology-Ring-Isomor}, we have elements
  \begin{align*}
  \ \,  u &= \phi\times \iota^J \in \widetilde{H}^*\Big(Q\slash F_J\wedge \bigwedge_{j\in  J} S^1_{(j)} \Big) = \widetilde{H}^*\Big(\mathbf{\Sigma}^{|J|} \big( Q\slash F_J\big) \Big),\\
     v &= \phi'\times \iota^{J'} \in \widetilde{H}^*\Big( Q\slash F_{J'}   \wedge  \bigwedge_{j\in  J'} S^1_{(j)}) \Big) =  \widetilde{H}^*\Big(\mathbf{\Sigma}^{|J'|} \big( Q\slash F_{J'}\big) \Big).
    \end{align*}    
 Then Lemma~\ref{Lem:Cross-Product-Pull-Back} and Lemma~\ref{Lem:Induced-Cup-Prod} imply that
  \begin{align*}  
       u \circledast v = \big( \widehat{\Delta}^{J,J'}_{J\cup J',Q_+} \big)^*
  \big( ( \phi\times \iota^{J} ) \times (\phi'\times \iota^{J'}) \big) &= (-1)^{|J||\phi'|}
   (\phi\, \widetilde{\cup}\, \phi') \times 
  \iota^{J\cup J'}. 
  \end{align*} 
   So we have a commutative diagram below, which
    implies that the product $\widetilde{\Cup}$ on $\widetilde{\mathcal{R}}^*_{Q}$  corresponds to the product $\circledast$ in~\eqref{Equ:Cohomology-Ring-Isomor} in this case.        
   \begin{equation} 
    \label{Equ:Diagram-Commute}
   \xymatrix{
             \widetilde{H}^*(Q\slash F_J) \otimes \widetilde{H}^*(Q\slash F_{J'}) \ar[d]_{\scalebox{0.9}{$\times \iota^{J} \otimes \times \iota^{J'}$} }  \ar[r]^{\quad\ \ \ \ \scalebox{0.85}{$\widetilde{\Cup}$}} & \widetilde{H}^*(Q\slash F_{J\cup J'}) \ar[d]_{\scalebox{0.9}{$\times \iota^{J\cup J'}$}} \\
           \widetilde{H}^*\big(\mathbf{\Sigma}^{|J|}(Q\slash F_J)\big) \otimes  \widetilde{H}^*\big(\mathbf{\Sigma}^{|J’|}(Q\slash F_{J’})\big) \ar[r]^{\ \ \, \qquad \ \scalebox{0.85}{$\circledast$} } & \widetilde{H}^*\big(\mathbf{\Sigma}^{|J\cup J'|}(Q\slash F_{J\cup J'})\big)
                 }  
            \end{equation}   
  \end{itemize}

   Combining the above arguments, we obtain isomorphisms of rings:
    $$ \big(\widetilde{\mathcal{R}}^*_Q, \widetilde{\Cup} \big) \overset{\cong}{\longrightarrow} 
   \Big(  \bigoplus_{J\subseteq [m]}\widetilde{H}^*\big(Q\slash F_J   \wedge  \bigwedge_{j\in  J} S^1_{(j)} \big), \circledast \Big) 
   \xlongrightarrow{ 
   \underset{J\subseteq [m]}{\bigoplus}    
    \widehat{\Pi}^*_{J}    
   } \widetilde{H}^*\big((D^2,S^1)^{Q}\big)    
    = \widetilde{H}^*(\mathcal{Z}_Q). $$   
   
  It follows that there is a ring isomorphism from 
   $\big(\mathcal{R}^*_Q, \Cup \big)$
   to $H^*(\mathcal{Z}_Q)$ up to a sign. 
   
   Note that the above ring isomorphism is not degree-preserving. But by the diagram in~\eqref{Equ:Diagram-Commute},
  we can
    make this ring isomorphism degree-preserving by shifting
     the degrees of all the elements in $H^*(Q,F_J)$ up by $|J|$
     for every $J\subseteq [m]$. The theorem is proved.
  \end{proof}

  \section{Polyhedral product over a nice manifold with corners} \label{Sec:Polyhedral-Prod}
    Let $Q$ be a nice manifold with corners whose facets are $F_1,\cdots, F_m$. Let
       \begin{equation*} 
        (\mathbb{X},\mathbb{A}) = \{ (X_j,A_j,a_j) \}^m_{j=1}
      \end{equation*}  
     where $X_j$ and $A_j$ are CW-complexes with a basepoint $a_j\in A_j\subseteq X_j$. 
     
     For any face $f$ of $Q$, define
    \begin{align*}
      (\mathbb{X},\mathbb{A})^{f} &:= 
     f
       \times \prod_{j\in I_f} X_j 
   \times \prod_{j\in [m]\backslash I_f} A_j,\\
 (\mathbb{X},\mathbb{A})^{Q} 
  &:= \bigcup_{f\in \mathcal{S}_Q}   (\mathbb{X},\mathbb{A})^{f} 
   \subseteq Q\times \prod_{j\in [m]} X_j.
   \end{align*}
 
  If $(\mathbb{X},\mathbb{A})= \{(X_j,A_j,a_j)=(X,A,a_0)\}^m_{j=1}$,
  we also denote $ (\mathbb{X},\mathbb{A})^{Q}$ by $ (X, A)^{Q}$.

  We call $(\mathbb{X},\mathbb{A})^{Q}$ the
  \emph{polyhedral product} of $(\mathbb{X},\mathbb{A})$
  over $Q$. Note that in general, the homeomorphism type of $(\mathbb{X},\mathbb{A})^{Q}$ depends on the ordering of the facets of $Q$ and the ordering of the $X_j$'s. We consider $ (\mathbb{X},\mathbb{A})^{Q}$ as an analogue of polyhedral products over a simplicial complex (see~\cite{BP00}). 
  
  In the rest of this section, we assume that
  each of $X_j$ and $A_j$ in $(\mathbb{X},\mathbb{A})$
  is either connected or is a disjoint union of a connected
  CW-complex with its basepoint. Then we can study the stable decomposition and cohomology ring of
 $(\mathbb{X},\mathbb{A})^{Q}$ in the same way as we do for $\mathcal{Z}_Q$.
 
  \begin{itemize}
  \item Let $Q_+=Q\cup q_0$ where $q_0 \notin Q$ is the basepoint of $Q_+$.
  
   \item For any face $f$ of $Q$, let $f_+ = f\cup q_0$ with basepoint $q_0$.
 
 \item For any $(f,L)\in \mathcal{P}_{Q}$, define 
    \begin{align} 
    (\mathbb{X},\mathbb{A})^{(f,L)} &:= 
   f \times \prod_{j\in I_f\backslash L} X_j 
   \times \prod_{j\in [m]\backslash (I_f\backslash L)} A_j
   \label{Equ:X-A-sigma-I-L} \\
      (\mathbb{X},\mathbb{A})^{(f,L)}_+ &:=   (\mathbb{X},\mathbb{A})^{(f,L)} \cup \widehat{q}^{(\mathbb{X},\mathbb{A})}_0 \ \text{where
      $\widehat{q}^{(\mathbb{X},\mathbb{A})}_0$ is the
       basepoint defined by} \nonumber \\
   \widehat{q}^{(\mathbb{X},\mathbb{A})}_0 &:= q_0\times \prod_{j\in [m]} a_j. \nonumber
    \end{align}
 \item Let $(\mathbb{X},\mathbb{A})^{Q}_+ = (\mathbb{X},\mathbb{A})^{Q}
 \cup \widehat{q}^{(\mathbb{X},\mathbb{A})}_0 $ with basepoint $\widehat{q}^{(\mathbb{X},\mathbb{A})}_0$.
  \end{itemize}

  Let $\mathbf{D}_{(\mathbb{X},\mathbb{A})+}: \mathcal{P}_{Q}\rightarrow CW_*$ be the diagram of based CW-complexes where  
  $$\mathbf{D}_{(\mathbb{X},\mathbb{A})+}((f,L))  
  :=  (\mathbb{X},\mathbb{A})^{(f,L)}_+, \ \forall
  (f,L)\in \mathcal{P}_{Q} $$
 and let $(d_{(\mathbb{X},\mathbb{A})+})_{(f,L),(f',L')}:
  \mathbf{D}_{(\mathbb{X},\mathbb{A})+}((f',L')) \rightarrow \mathbf{D}_{(\mathbb{X},\mathbb{A})+}((f,L))$ be the natural inclusion for any $(f,L)\leq (f',L')\in \mathcal{P}_{Q}$. Then we have
 \begin{equation} \label{Equ:D-A-Extra}
   (\mathbb{X},\mathbb{A})^{Q}_+  = \mathrm{colim}(\mathbf{D}_{(\mathbb{X},\mathbb{A})+}) = \bigcup_{(f,L)\in \mathcal{P}_{Q}} (\mathbb{X},\mathbb{A})^{(f,L)}_+ . 
   \end{equation}
   
  By Theorem~\ref{Thm:Generalized-Stable-Decomp}, we can prove the following lemma
  parallel to Lemma~\ref{Lem:Stable-Decomp-plus}.
    \begin{lem} \label{Lem:Stable-Decomp-plus-X-A}
   For any $(f,L)\in \mathcal{P}_{Q}$, there is a natural homeomorphism which commutes with taking the colimit:
    $$ (\mathbb{X},\mathbb{A})^{(f,L)}_+ \cong 
    f_+ \rtimes \Big(  \prod_{j\in I_f\backslash L} X_j 
   \times \prod_{j\in [m]\backslash (I_f\backslash L)} A_j \Big) . $$
   \end{lem}
  
  So by Theorem~\ref{Thm:Stabel-Decomp-rtimes}, we have 
 \begin{align}\label{Equ:term-Essential-X-A}
   \mathbf{\Sigma} \Big(  (\mathbb{X},\mathbb{A})^{(f,L)}_+ \Big) 
 &\cong \mathbf{\Sigma}\Big( f_+ 
       \rtimes \Big(\prod_{j\in I_f\backslash L} X _j
   \times \prod_{j\in [m]\backslash (I_f\backslash L)} A_j \Big)  \Big) \\
    & \simeq \bigvee_{J \subseteq [m]} \mathbf{\Sigma}\Big(  f_+
       \wedge \bigwedge_{j\in J\cap (I_f\backslash L)} X_j 
   \wedge \bigwedge_{j\in  J\backslash (I_f\backslash L)} A_j \Big) \nonumber
\end{align}

 Accordingly, we define a family of 
 diagrams of 
based CW-complexes,
$$\widehat{\mathbf{D}}^J_{(\mathbb{X},\mathbb{A})+}:  \mathcal{P}_{Q}\rightarrow CW_*, \ J\subseteq [m]$$
   \begin{equation*} 
   \widehat{\mathbf{D}}^J_{(\mathbb{X},\mathbb{A})+}((f,L)) := f_+
       \wedge \bigwedge_{j\in J\cap (I_f\backslash L)} X_j 
   \wedge \bigwedge_{j\in  J\backslash (I_f\backslash L)} A_j, \ \forall
  (f,L)\in \mathcal{P}_{Q}.
    \end{equation*}
 Define $(\widehat{d}^J_{(\mathbb{X},\mathbb{A})+})_{(f,L),(f',L')}: 
 \widehat{\mathbf{D}}^J_{(\mathbb{X},\mathbb{A})+}((f',L'))\rightarrow  \widehat{\mathbf{D}}^J_{(\mathbb{X},\mathbb{A})+}((f,L))$
 to be the natural inclusion for any $(f,L) \leq (f',L')\in \mathcal{P}_{Q}$. The basepoint of $\widehat{\mathbf{D}}^J_{(\mathbb{X},\mathbb{A})+}((f,L))$ is $\big[q_0\times \prod_{j\in J} a_j \big]$. So we have the following
 theorem by~\cite[Theorem 4.3]{BBCG10}.
 
  \begin{thm} \label{Thm:Stable-Decomp-X-A}
  Let $(\mathbb{X},\mathbb{A})= \{ (X_j,A_j,a_j) \}^m_{j=1}$ where each $X_j$ and $A_j$ 
  is either connected or is a disjoint union of a connected
  CW-complex with its basepoint. Then there are homotopy equivalences:
  \begin{equation*} 
  S^1\vee \mathbf{\Sigma}\big( (\mathbb{X},\mathbb{A})^{Q} \big) \simeq
 \mathbf{\Sigma}\big( (\mathbb{X},\mathbb{A})^{Q}_+ \big) =
  \mathbf{\Sigma}\big( \mathrm{colim}(\mathbf{D}_{(\mathbb{X},\mathbb{A})+}) \big)   \simeq \bigvee_{J\subseteq [m]} \mathbf{\Sigma}\big( \mathrm{colim}\big(\widehat{\mathbf{D}}^J_{(\mathbb{X},\mathbb{A})+}\big)\big).  
  \end{equation*}
  This implies
  $$ H^*\big( (\mathbb{X},\mathbb{A})^{Q} \big)
  \cong \widetilde{H}^*( \mathrm{colim}(\mathbf{D}_{(\mathbb{X},\mathbb{A})+}))
  \cong \bigoplus_{J\subseteq [m]} \widetilde{H}^*\big( \mathrm{colim}\big(\widehat{\mathbf{D}}^{J}_{(\mathbb{X},\mathbb{A})+}  \big) \big).  $$
  \end{thm}

  Moreover, using the partial diagonal map for $Q_+\times \prod_{j\in [m]} X_j$ as in the proof of Theorem~\ref{Thm:Cohomology-Ring-Isom}, 
  we have a diagram parallel to diagram~\eqref{Equ:Diagram-Z-Q} for any $J , J'\subseteq [m]$:
   \begin{equation} \label{Equ:Diagram-Z-Q-X-A}
       \xymatrix{
         \mathrm{colim}(\mathbf{D}_{(\mathbb{X},\mathbb{A})+}) \ar[d]_{\widehat{\Pi}_{J\cup J'}}  \ar[r]^{\Delta^{Q_+,\mathbb{X}}_{[m]}\qquad \ \ \ \ \ }
               \, & \, \mathrm{colim}(\mathbf{D}_{(\mathbb{X},\mathbb{A})+}) \wedge \mathrm{colim}(\mathbf{D}_{(\mathbb{X},\mathbb{A})+}) \ar[d]^{\widehat{\Pi}_{J}\wedge \widehat{\Pi}_{J'}}  \\
        \mathrm{colim} \big(\widehat{\mathbf{D}}^{J\cup J'}_{(\mathbb{X},\mathbb{A})+}\big)  \ar[r]^{ \widehat{\Delta}^{J,J'}_{J\cup J',Q_+} \quad\ \ \qquad }\ & \ \mathrm{colim} \big(\widehat{\mathbf{D}}^{J}_{(\mathbb{X},\mathbb{A})+}\big) \wedge 
        \mathrm{colim} \big(\widehat{\mathbf{D}}^{J'}_{(\mathbb{X},\mathbb{A})+}\big).
         }  
   \end{equation}
    
   Similarly, we can obtain the following theorem parallel to Theorem~\ref{Thm:Cohomology-Ring-Isom}.
   
   \begin{thm} \label{Thm:Cohomology-Ring-Isom-X-A}
  Let $(\mathbb{X},\mathbb{A})= \{ (X_j,A_j,a_j) \}^m_{j=1}$ where each of $X_j$ and $A_j$ 
  is either connected or is a disjoint union of a connected CW-complex with its basepoint. Then
 there is a ring isomorphism  
  \begin{equation*} 
      \bigoplus_{J\subseteq [m]} \widehat{\Pi}^*_{J} : \bigoplus_{J\subseteq [m]} \widetilde{H}^*\big( \mathrm{colim}\big(\widehat{\mathbf{D}}^{J}_{(\mathbb{X},\mathbb{A})+}  \big) \big) \longrightarrow \widetilde{H}^*\big(\mathrm{colim}(\mathbf{D}_{(\mathbb{X},\mathbb{A})+})\big)\cong H^*\big( (\mathbb{X},\mathbb{A})^{Q} \big), 
  \end{equation*}  
    $$ \text{ where the product $\circledast$ on $ \bigoplus_{J\subseteq [m]} \widetilde{H}^*\big( \mathrm{colim}\big(\widehat{\mathbf{D}}^{J}_{(\mathbb{X},\mathbb{A})+}  \big) \big)$ is defined by} \qquad\qquad\qquad\quad\quad
    $$
    \begin{equation} \label{Equ:Product-General-X-A}
      \widetilde{H}^*\big( \mathrm{colim}\big(\widehat{\mathbf{D}}^{J}_{(\mathbb{X},\mathbb{A})+}  \big) \big) \otimes
    \widetilde{H}^*\big( \mathrm{colim}\big(\widehat{\mathbf{D}}^{J'}_{(\mathbb{X},\mathbb{A})+}  \big) \big)
      \xlongrightarrow{\circledast} \widetilde{H}^*\big( \mathrm{colim}\big(\widehat{\mathbf{D}}^{J\cup J'}_{(\mathbb{X},\mathbb{A})+}\big)\big) 
      \end{equation}
       $$  u \circledast v := 
      \big( \widehat{\Delta}^{J,J'}_{J\cup J',Q_+} \big)^*(u \times v). $$
      \end{thm}

      In the following two subsections, we will study the 
      stable decomposition and cohomology ring
      of $(\mathbb{X},\mathbb{A})^Q$ under some special 
      conditions on $(\mathbb{X},\mathbb{A})$.

      \subsection{The case of $(\mathbb{X},\mathbb{A})^Q$ with each $X_j$ contractible}
      \ \vskip .1cm
   
   Observe that in the proof of Theorem~\ref{Thm:Stable-Decomp-Main}, the only properties of $(D^2,S^1)$ that we actually use are that
    \begin{itemize}
    \item[(i)] $D^2$ is contractible;
    
    \item[(ii)] $X\wedge S^1$ is homeomorphic to $\mathbf{\Sigma}(X)$ for any based CW-complex $X$.
    \end{itemize}
    
    So if we assume that every $X_j$ in $(\mathbb{X},\mathbb{A})$ is contractible, we can obtain the following theorem parallel to Theorem~\ref{Thm:Stable-Decomp-Main}.

 \begin{thm} \label{Thm:Main-General-X-A}
  Let $Q$ be a nice manifold with corners with facets $F_1,\cdots, F_m$.
  Let $(\mathbb{X},\mathbb{A})= \{ (X_j,A_j,a_j) \}^m_{j=1}$ where each $X_j$ is contractible and each $A_j$ is either connected or is a disjoint union of a connected CW-complex with its basepoint. Then there is a homotopy equivalence
      \begin{equation*}  
         \mathbf{\Sigma}\big(  (\mathbb{X},\mathbb{A})^{Q} \big) \simeq \bigvee_{J\subseteq [m]} 
           \mathbf{\Sigma} \Big( Q\slash F_J \wedge\bigwedge_{j\in J} A_j \Big) .
           \end{equation*} 
 So the reduced cohomology group
  $$ \widetilde{H}^*\big(  (\mathbb{X},\mathbb{A})^{Q} \big) \cong \bigoplus_{J\subseteq [m]} \widetilde{H}^*\big(Q\slash F_J \wedge\bigwedge_{j\in J} A_j \big).$$    
    \end{thm}
    \begin{proof}
     We can easily extend the argument in the proof of Theorem~\ref{Thm:Stable-Decomp-Main} to show 
     \[  \mathrm{colim}\big(\widehat{\mathbf{D}}^{J}_{(\mathbb{X},\mathbb{A})+} \big) \simeq 
      \begin{cases}
    Q\slash F_J \wedge \underset{j\in J}{\scalebox{1.3}{$\bigwedge$}} A_j ,  &  \text{if $J\neq \varnothing$}; \\
    Q_+ ,  &  \text{if $J=\varnothing$}.
 \end{cases}  \]
 
 Then the statements of the theorem follow from Theorem~\ref{Thm:Stable-Decomp-X-A} and the fact that $\mathbf{\Sigma}(Q_+)\simeq
 S^1\vee \mathbf{\Sigma}(Q)$.
    \end{proof}
  
   Moreover, we have the following theorem which is parallel to~\cite[Theorem 1.4]{BBCG12}.
    
      \begin{thm} \label{Equ:Cohomology-Ring-Iso-X-A}
     Under the condition in Theorem~\ref{Thm:Main-General-X-A}, there is a ring isomorphism
      \[  
        \Big( \bigoplus_{J\subseteq [m]} \widetilde{H}^*\big(Q\slash F_J \wedge\bigwedge_{j\in J} A_j \big),\circledast \Big) \longrightarrow
         \widetilde{H}^*\big((\mathbb{X},\mathbb{A})^{Q} \big) \ 
         \text{induced by}\ \bigoplus_{J\subseteq [m]} \widehat{\Pi}^*_{J}. \]           
      \end{thm}

  \begin{rem}
  If any combination of $Q\slash F_{J}$ and $A_j$'s satisfies the
    \emph{strong smash form of the K\"unneth formula} as defined in~\cite[p.\,1647]{BBCG10} over a coefficient ring $\mathbf{k}$, i.e. the natural map
    $$ \widetilde{H}^*\big(Q\slash F_J;\mathbf{k}\big) \otimes \bigotimes_{j\in I } \widetilde{H}^*(A_j;\mathbf{k})  \longrightarrow \widetilde{H}^*\Big(Q\slash F_J\wedge\bigwedge_{j\in I} 
    A_j;\mathbf{k} \Big) $$
    is an isomorphism for any $I, J\subseteq [m]$,    
    we can write the cohomology ring structure of
    $(\mathbb{X},\mathbb{A})^Q$ with $\mathbf{k}$-coefficients more explicitly via 
    the formula in 
    Lemma~\ref{Lem:Cross-Product-Pull-Back}.
   \end{rem}

   In the following, we demonstrate the product $\circledast$ for $(\mathbb{D},\mathbb{S})^{Q}$
   where
      $$(\mathbb{D},\mathbb{S}) =
    \big\{ \big( D^{n_j+1}, S^{n_j}, a_j \big) \big\}^m_{j=1}.$$
     Here $D^{n+1}$ is the unit ball
      in $\R^{n+1}$ and $S^n=\partial D^{n+1}$. 
      
      In particular, if 
      $(\mathbb{D},\mathbb{S}) =
    \big\{ \big( D^{n_j+1}, S^{n_j}, a_j \big) =\big( D^{n+1},S^n,a_0\big) \big\}^m_{j=1}$, we also write
    $$(\mathbb{D},\mathbb{S})^{Q} =
    (D^{n+1},S^n)^{Q}.$$
      
  \begin{exam}
      $\mathcal{Z}_Q 
      \cong (D^2,S^1)^Q$ and 
      $\R\mathcal{Z}_Q 
      \cong (D^1,S^0)^Q$ (see Remark~\ref{Rem:Real-Moment-Angle}).
  \end{exam}  
  
   We define a graded ring structure $\Cup^{(\mathbb{D},\mathbb{S})}$ on $\mathcal{R}^*_Q $ according to
    $(\mathbb{D},\mathbb{S})$ as follows.
  \begin{itemize}
   \item If $J\cap J'=\varnothing$ or $J\cap J' \neq \varnothing$ but $n_j=0$ for all $j\in J\cap J'$, then\\   
  $ H^*(Q,F_J) \otimes H^{*}(Q,F_{J'}) \xlongrightarrow{\Cup^{(\mathbb{D},\mathbb{S})}}
    H^{*}(Q,F_{J\cup J'})$ is the relative cup product $\cup$. \vskip .1cm
          
     \item If $J\cap J' \neq \varnothing$ and there exists $n_j\geq 1$ for some $j\in J\cap J'$, then
    \\   
  $ H^*(Q,F_J) \otimes H^{*}(Q,F_{J'}) \xlongrightarrow{\Cup^{(\mathbb{D},\mathbb{S})}}
    H^{*}(Q,F_{J\cup J'})$ is trivial.
    \end{itemize}
      
      By Lemma~\ref{Lem:Induced-Cup-Prod}, the product $\Cup^{(\mathbb{D},\mathbb{S})}$ on $\mathcal{R}^*_Q $
      induces a product $\widetilde{\Cup}^{(\mathbb{D},\mathbb{S})}$ on $\widetilde{\mathcal{R}}^*_Q $.
            
             \begin{itemize}
   \item If $J\cap J'=\varnothing$ or $J\cap J' \neq \varnothing$ but $n_j=0$ for all $j\in J\cap J'$, then\\   
  $ \widetilde{H}^*(Q\slash F_J) \otimes \widetilde{H}^{*}(Q\slash F_{J'}) \xlongrightarrow{\widetilde{\Cup}^{(\mathbb{D},\mathbb{S})}}
    \widetilde{H}^{*}(Q\slash F_{J\cup J'})$ is the product $\widetilde{\cup}$
    induced from the relative cup product 
    $ H^*(Q,F_J) \otimes H^{*}(Q,F_{J'}) \xlongrightarrow{\cup }
    H^{*}(Q,F_{J\cup J'})$. \vskip .1cm
          
     \item If $J\cap J' \neq \varnothing$ and there exists $n_j\geq 1$ for some $j\in J\cap J'$, then
    \\   
  $ \widetilde{H}^*(Q\slash F_J) \otimes \widetilde{H}^{*}(Q\slash F_{J'}) \xlongrightarrow{\widetilde{\Cup}^{(\mathbb{D},\mathbb{S})}}
    \widetilde{H}^{*}(Q\slash F_{J\cup J'})$ is trivial.
    \end{itemize}

 We have the following theorem which generalizes Theorem~\ref{Thm:Stable-Decomp-Main} and Theorem~\ref{Thm:Cohomology-Ring-Isom}.

     \begin{thm} \label{Thm:Main-Dn-Sn-1}
  Let $Q$ be a nice manifold with corners with facets $F_1,\cdots, F_m$.
 Then for any $(\mathbb{D},\mathbb{S}) =
    \big\{ \big( D^{n_j+1}, S^{n_j} , a_j \big) \big\}^m_{j=1}$,
      \begin{itemize}  
      \item[(a)] There is a homotopy equivalence 
      $$\qquad \mathbf{\Sigma}\big( (\mathbb{D},\mathbb{S})^{Q} \big) \simeq \bigvee_{J\subseteq [m]} 
           \mathbf{\Sigma}\Big( Q\slash F_J \wedge
           \bigwedge_{j\in  J} S^{n_j} \Big) \cong \bigvee_{J\subseteq [m]} 
           \mathbf{\Sigma}^{1+\sum_{j\in J} n_j}\big( Q\slash F_J \big).$$     
     $$  \text{This implies}\  H^p \big( (\mathbb{D},\mathbb{S})^{Q} \big)\cong \bigoplus_{J \subseteq [m]} H^{p-\sum_{j\in J} n_j}(Q,F_J), \, \forall p\in \Z. \qquad\qquad\ $$ 
     
   \item[(b)] There is a ring isomorphism (up to a sign) from
   $(\mathcal{R}^*_Q,\Cup^{(\mathbb{D},\mathbb{S})})$ to the integral cohomology ring of $(\mathbb{D},\mathbb{S})^{Q}$. Moreover, we can make this ring isomorphism degree-preserving by shifting the degrees of the elements in $H^*(Q,F_J)$ for every $J\subseteq [m]$.
   \end{itemize}
      \end{thm}
 \begin{proof}  For brevity, we use the following notation in our proof.
    $$N_J = \sum_{j\in J} n_j, \ J\subseteq [m].$$ 
    
   Statement (a) follows from Theorem~\ref{Thm:Main-General-X-A} and the simple fact that:
    $$ Q\slash F_J   \wedge  \bigwedge_{j\in  J} S^{n_j} \cong Q\slash F_J   \wedge  S^{N_J} \cong \mathbf{\Sigma}^{N_J}(Q\slash F_J).$$

   For statement (b), note that by Theorem~\ref{Equ:Cohomology-Ring-Iso-X-A} we have a ring isomorphism
   \begin{equation} \label{Equ:Cohomology-Ring-Isomor-Dn-Sn-1}
    \bigoplus_{J\subseteq [m]} \widehat{\Pi}^*_{J}:
         \Big( \bigoplus_{J\subseteq [m]}\widetilde{H}^*\big(Q\slash F_J   \wedge  \bigwedge_{j\in  J} S^{n_j} \big), \circledast \Big) \longrightarrow \widetilde{H}^*\big((\mathbb{D},\mathbb{S})^{Q}\big).
     \end{equation}     
     
    \quad \ \ \, For any $1\leq j \leq m$, let $\iota^{n_j}$ denote a generator of $\widetilde{H}^{n_j}(S^{n_j})$. Let 
 $$ \iota^J_{(\mathbb{D},\mathbb{S})} =\underset{j\in J}{\largetimes} \iota^{n_j} \in \widetilde{H}^{N_J}\Big(\bigwedge_{j\in  J} S^{n_j})\Big) \ \text{be a generator.}$$
 
   \begin{itemize}
      \item[(i)] Assume $J\cap J'\neq \varnothing$ and there exists $n_j\geq 1$ for some $j\in J\cap J'$. Then since $S^{n_j}$ is a suspension space, the map $\widehat{\Delta}^{J,J'}_{J\cup J',Q_+}$ in~\eqref{Equ:Diagram-Z-Q-X-A} is null-homotopic. This implies that the product $\circledast$ in~\eqref{Equ:Cohomology-Ring-Isomor-Dn-Sn-1} is trivial 
      which corresponds to the definition of $\widetilde{\Cup}^{(\mathbb{D},\mathbb{S})}$
      on $\widetilde{\mathcal{R}}^*_Q$ in this case.
      
       \item[(ii)] Assume $J\cap J' \neq \varnothing$ but $n_j=0$ for all $j\in J\cap J'$.  Let 
       $$J_0 =\{ j \in [m]\,|\, n_j=0 \} \subseteq [m].$$        
       So the condition on $J$ and $J'$ is equivalent to $J\cap J'\subseteq J_0$ which implies
     \begin{equation} \label{Equ:Disjoint-Union-J}
      (J\backslash J_0)\cap (J'\backslash J_0)=\varnothing.
     \end{equation}

      Since $X\wedge S^0 \cong X$ for any based space $X$, we have for any $J\subseteq [m]$:
    $$\qquad  Q\slash F_J   \wedge  \bigwedge_{j\in  J} S^{n_j} \cong Q\slash F_J   \wedge  \bigwedge_{j\in  J\backslash J_0} S^{n_j} \cong \mathbf{\Sigma}^{N_{J\backslash J_0}}(Q\slash F_J). $$
    
 By Lemma~\ref{Lem:Cross-Product-Pull-Back} and Lemma~\ref{Lem:Induced-Cup-Prod}, we can derive an explicit formula for the product
  $\circledast$ in~\eqref{Equ:Cohomology-Ring-Isomor-Dn-Sn-1} as follows. 
  For any elements
        \begin{align*}
    u &= \phi\times \iota^{J\backslash J_0}_{(\mathbb{D},\mathbb{S})} \in \widetilde{H}^*\Big(Q\slash F_J   \wedge \bigwedge_{j\in  J\backslash J_0} S^{n_j} \Big)=
     \widetilde{H}^*\big(\mathbf{\Sigma}^{N_{J\backslash J_0}}(Q\slash F_J)\big),\\
     v &= \phi'\times \iota^{J'\backslash J_0}_{(\mathbb{D},\mathbb{S})} \in \widetilde{H}^*\Big(Q\slash F_{J'}   \wedge \bigwedge_{j\in  J'\backslash J_0} S^{n_j}  \Big)=\widetilde{H}^*\big(\mathbf{\Sigma}^{N_{J'\backslash J_0}}(Q\slash F_{J’})\big),\\
   \qquad   u \circledast v  &=  \big(  \widehat{\Delta}^{J\backslash J_0,J'\backslash J_0}_{(J\cup J')\backslash J_0,Q_+} \big)^*
  \Big( \big( \phi\times \iota^{J\backslash J_0}_{(\mathbb{D},\mathbb{S})} \big) \times \big(\phi'\times \iota^{J'\backslash J_0}_{(\mathbb{D},\mathbb{S})} \big) \Big) \overset{\eqref{Equ:Disjoint-Union-J}}{=} (-1)^{N_{J\backslash J_0}|\phi'|} (\phi\,\widetilde{\cup}\, \phi') \times 
 \iota^{(J\cup J')\backslash J_0}_{(\mathbb{D},\mathbb{S})}.  
    \end{align*}   
  So we have a commutative diagram parallel to diagram~\eqref{Equ:Diagram-Commute} below      
   \begin{equation*} 
  \qquad \ \xymatrix{
             \widetilde{H}^*(Q\slash F_J) \otimes \widetilde{H}^*(Q\slash F_{J'}) \ar[d]_{\scalebox{0.8}{$\times \iota^{J\backslash J_0}_{(\mathbb{D},\mathbb{S})}  \otimes \times \iota^{J'\backslash J_0}_{(\mathbb{D},\mathbb{S})} $} }  \ar[r]^{\qquad\ \ \ \ \scalebox{0.85}{$\widetilde{\Cup}^{(\mathbb{D},\mathbb{S})}$}} & \widetilde{H}^*(Q\slash F_{J\cup J'}) \ar[d]_{\scalebox{0.8}{$\times \iota^{(J\cup J')\backslash J_0}_{(\mathbb{D},\mathbb{S})} $}} \\
           \widetilde{H}^*\big(\mathbf{\Sigma}^{N_{J\backslash J_0}}(Q\slash F_J)\big) \otimes  \widetilde{H}^*\big(\mathbf{\Sigma}^{N_{J'\backslash J_0}}(Q\slash F_{J’})\big) \ar[r]^{ \qquad \quad\ \ \scalebox{0.85}{$\circledast$} } & \widetilde{H}^*\big(\mathbf{\Sigma}^{N_{(J\cup J')\backslash J_0}} (Q\slash F_{J\cup J'})\big).
                 }  
  \end{equation*}  
  This implies that the product $\widetilde{\Cup}^{(\mathbb{D},\mathbb{S})}$ on $\widetilde{\mathcal{R}}^*_Q$ corresponds to the product $\circledast$ in~\eqref{Equ:Cohomology-Ring-Isomor-Dn-Sn-1}
    in this case.

      \item[(iii)] When $J\cap J'=\varnothing$, the proof
      of the correspondence between
      the product $\widetilde{\Cup}^{(\mathbb{D},\mathbb{S})}$ on $\widetilde{\mathcal{R}}^*_Q$ and the product $\circledast$ in~\eqref{Equ:Cohomology-Ring-Isomor-Dn-Sn-1} is the same as case (ii). 
 \end{itemize}  
 
  The above discussion implies that there is an isomorphism of rings:
  $$\big(\widetilde{\mathcal{R}}^*_Q, \widetilde{\Cup}^{(\mathbb{D},\mathbb{S})} \big) \longrightarrow
  \Big( \bigoplus_{J\subseteq [m]}\widetilde{H}^*\big(Q\slash F_J   \wedge  \bigwedge_{j\in  J} S^{n_j} \big), \circledast \Big) \xlongrightarrow{\ 
  \underset{J\subseteq [m]}{\bigoplus} \widehat{\Pi}^*_{J}\ }  \widetilde{H}^*\big((\mathbb{D},\mathbb{S})^{Q} \big) .$$
  
 This implies that
   $(\mathcal{R}^*_Q,\Cup^{(\mathbb{D},\mathbb{S})})$ is isomorphic (up to a sign) to the integral cohomology ring $H^*\big((\mathbb{D},\mathbb{S})^{Q} \big)$. Moreover, according to the above diagram, we can make the ring isomorphism between $(\mathcal{R}^*_Q,\Cup^{(\mathbb{D},\mathbb{S})})$ and $H^*\big((\mathbb{D},\mathbb{S})^{Q} \big)$ degree-preserving by
   shifting the degrees of all the elements in 
   $H^*(Q,F_J)$ up by $N_{J\backslash J_0}$ for every $J\subseteq [m]$.
    The theorem is proved.
 \end{proof}
     
   \begin{rem}
    $S^0$ is not a suspension of any space and the reduced diagonal map $\Delta_{S^0} = id_{S^0}: S^0\rightarrow S^0\wedge S^0\cong S^0$ is not null-homotopic. This is the essential reason why for a general $(\mathbb{D},\mathbb{S})$, the cohomology ring of $(\mathbb{D},\mathbb{S})^{Q}$ is more subtle than that of $\mathcal{Z}_Q$.
    \end{rem}
   
      A very special case of Theorem~\ref{Thm:Main-Dn-Sn-1} is $(D^1,S^0)^Q=\R\mathcal{Z}_Q$ where 
         the product $\Cup^{(D^1,S^0)}$ on $\mathcal{R}^*_Q$ is exactly the relative cup product for all $J,J'\subseteq [m]$.

    \begin{cor} \label{Cor:Real-Moment}
     Let $Q$ be a nice manifold with corners with facets $F_1,\cdots, F_m$. Then we have
     $$  \mathbf{\Sigma}( \R\mathcal{Z}_Q ) \simeq \bigvee_{J\subseteq [m]} 
           \mathbf{\Sigma}( Q\slash F_J), \ \ H^p(\R\mathcal{Z}_Q) \cong \bigoplus_{J\subseteq [m]} H^p(Q, F_J), \ \forall p\in \Z.$$
    Moreover, the integral cohomology ring of $\R\mathcal{Z}_Q$ is isomorphic as a graded ring to the ring $(\mathcal{R}^*_Q,\cup)$
     where $\cup$ is the relative cup product
      $$ H^*(Q,F_J) \otimes H^{*}(Q,F_{J'}) \overset{\cup}{\longrightarrow}
    H^{*}(Q,F_{J\cup J'}), \ \forall J,J'\subseteq [m].$$ 
    \end{cor}
   Note that in the case of $\R\mathcal{Z}_Q$, the sign factor of the isomorphism between 
    $(\mathcal{R}^*_Q,\cup)$ and $H^*(\R\mathcal{Z}_Q)$ is trivial because the degree of $\iota^J_{(D^1,S^0)}$ is always zero.
 \begin{rem}
    When $Q$ is a simple polytope, 
    the ring structure of the integral cohomology of $\R\mathcal{Z}_Q$ was studied in~\cite{CaiLi17} via a different method.
    \end{rem}

     \subsection{The case of $(\mathbb{X},\mathbb{A})^Q$ with each $A_j$ contractible} 
     \ \vskip .1cm
     
    If in $(\mathbb{X},\mathbb{A})= \{ (X_j,A_j,a_j) \}^m_{j=1}$, each $A_j$ is contractible, we can derive the
    stable decomposition of $(\mathbb{X},\mathbb{A})^Q$ from Theorem~\ref{Thm:Stable-Decomp-X-A} as follows.
    
   \begin{thm} \label{Thm:Stable-Decomp-A-Contract}
   Let $Q$ be a nice manifold with corners with facets $F_1,\cdots, F_m$.
  Let $(\mathbb{X},\mathbb{A})= \{ (X_j,A_j,a_j) \}^m_{j=1}$ where each $A_j$ is contractible and each $X_j$ is either connected or is a disjoint union of a connected CW-complex with its basepoint. Then there is a homotopy equivalence
  \begin{align*}
    S^1\vee \mathbf{\Sigma}\big( (\mathbb{X},\mathbb{A})^{Q} \big)  \simeq
 \mathbf{\Sigma}\big( (\mathbb{X},\mathbb{A})^{Q}_+ \big) 
  & \simeq \bigvee_{J\subseteq [m]} 
 \mathbf{\Sigma}  \Big(  \big( F_{\cap J}\cup q_0\big)  \wedge \bigwedge_{j\in J}  X_j \Big). 
  \end{align*} 
  \[  \text{So we have}\ \, \mathbf{\Sigma}\big( (\mathbb{X},\mathbb{A})^{Q} \big)\simeq  \bigvee_{J\subseteq [m]} 
 \mathbf{\Sigma}  \Big(   F_{\cap J}  \ltimes \bigwedge_{j\in J}  X_j \Big), \ \text{and} \qquad\qquad\qquad\qquad\quad\ \]
   \[ H^*\big((\mathbb{X},\mathbb{A})^{Q} \big) \cong
   \bigoplus_{J\subseteq [m]} \widetilde{H}^* \Big(  \big( F_{\cap J}\cup q_0\big)  \wedge \bigwedge_{j\in J}  X_j \Big).  \]
  \end{thm}
    \begin{proof}
       By Lemma~\ref{Lem:Smash-Contract} and our assumption on $A_j$, when $J\backslash (I_f\backslash L) \neq \varnothing$,
    \begin{equation*} 
   \widehat{\mathbf{D}}^J_{(\mathbb{X},\mathbb{A})+}((f,L)) = f_+
       \wedge \bigwedge_{j\in J\cap (I_f\backslash L)} X_j 
   \wedge \bigwedge_{j\in  J\backslash (I_f\backslash L)} A_j
    \end{equation*}
  is contractible. So for any $J\subseteq [m]$, we define a diagram of based CW-complexes 
$$\widehat{\mathbf{G}}^J_{(\mathbb{X},\mathbb{A})+}:  \mathcal{P}_{Q}\rightarrow CW_*$$ 
   \begin{equation} \label{Equ:G-J}
       \widehat{\mathbf{G}}^J_{(\mathbb{X},\mathbb{A})+}((f,L)) :=
        \begin{cases}
  \widehat{\mathbf{D}}^J_{(\mathbb{X},\mathbb{A})+}((f,L)) = f_+
       \wedge \underset{j\in  J}{\scalebox{1.3}{$\bigwedge$}} X_j,  &  \text{if $J\backslash (I_f\backslash L)=\varnothing$}; \\
  \quad   [\widehat{q}^J_0],  &  \text{if $J\backslash (I_f\backslash L) \neq \varnothing$}.
 \end{cases} 
    \end{equation}
 where
$(\widehat{g}^J_{(\mathbb{X},\mathbb{A})+})_{(f,L),(f',L')}: 
 \widehat{\mathbf{G}}^J_{(\mathbb{X},\mathbb{A})+}((f',L'))\rightarrow  \widehat{\mathbf{G}}^J_{(\mathbb{X},\mathbb{A})+}((f,L))$ is either the natural inclusion or the constant map $\mathbf{c}_{[\widehat{q}^J_0]}$ for any $(f,L) \leq (f',L')\in \mathcal{P}_{Q}$. The basepoint of $\widehat{\mathbf{G}}^J_{(\mathbb{X},\mathbb{A})+}((f,L))$ is $[\widehat{q}^J_0]$.
 
  Let $\Psi^J_{(\mathbb{X},\mathbb{A})+}: 
  \widehat{\mathbf{D}}^J_{(\mathbb{X},\mathbb{A})+}\rightarrow \widehat{\mathbf{G}}^J_{(\mathbb{X},\mathbb{A})+}$ be a map of diagrams over $\mathcal{P}_{Q}$ defined by:
   $$  (\Psi^J_{(\mathbb{X},\mathbb{A})+})_{(f,L)} : \widehat{\mathbf{D}}^J_{(\mathbb{X},\mathbb{A})+}((f,L)) \rightarrow  \widehat{\mathbf{G}}^J_{(\mathbb{X},\mathbb{A})+}((f,L)) $$
   $$ (\Psi^J_{(\mathbb{X},\mathbb{A})+})_{(f,L)} =  \begin{cases}
    id_{\widehat{\mathbf{D}}^J_{(\mathbb{X},\mathbb{A})+}((f,L))} ,  &  
  \text{if $J\backslash (I_f\backslash L)=\varnothing$}; \\
    \mathbf{c}_{[\widehat{q}^J_0]},  &  \text{if $J\backslash (I_f\backslash L)\neq\varnothing$}.
 \end{cases}  $$

   Then by Theorem~\ref{Thm:Homotopy-Colimit-Lemma}, there exists a homotopy equivalence:
  $$\mathrm{colim}\big(\widehat{\mathbf{D}}^J_{(\mathbb{X},\mathbb{A})+}\big) \simeq
    \mathrm{colim}\big(\widehat{\mathbf{G}}^J_{(\mathbb{X},\mathbb{A})+}\big), \ J\subseteq [m].$$

  To understand $\mathrm{colim}\big(\widehat{\mathbf{G}}^J_{(\mathbb{X},\mathbb{A})+}\big)$, we need to figure out in~\eqref{Equ:G-J} what are 
 those faces $f$ of $Q$ with some $L\subseteq I_f$ such that $J\backslash (I_f\backslash L) = \varnothing$. 
 \begin{itemize}
  \item There exists $L\subseteq I_f$ such that
  $J\backslash (I_f\backslash L) = \varnothing$ if and only if $J\subseteq I_f$. So 
  \begin{equation} \label{Equ:sum-B-sigma-J}
   \bigcup_{f \in \mathcal{S}_Q} \underset{J\backslash (I_f\backslash L) = \varnothing}{\bigcup_{\exists L\subseteq I_f}} f  = \bigcup_{f\in \mathcal{S}_Q} \bigcup_{J\subseteq I_f} f =
    F_{\cap J}.
    \end{equation}

    \item There exists $L\subseteq I_f$ such that
  $J\backslash (I_f\backslash L) \neq \varnothing$ if and only if $J\neq \varnothing$.
    \end{itemize}

 Then for any $\varnothing\neq J \subseteq [m]$, we have
  \begin{align}
   \mathrm{colim}\big(\widehat{\mathbf{G}}^J_{(\mathbb{X},\mathbb{A})+}\big) &\cong \underset{J\backslash (I_f\backslash L)=\varnothing}{\bigcup_{\exists L\subseteq I_f}}  \Big( f_+
       \wedge \bigwedge_{j\in J} X_j  \Big) 
       \bigg\slash \underset{J\backslash (I_f\backslash L)\neq \varnothing}{\bigcup_{\exists L\subseteq I_f}}  \Big( f_+
       \wedge \bigwedge_{j\in I_f\backslash L} X_j 
   \wedge \bigwedge_{j\in  J\backslash ( I_f\backslash L)} A_j
       \Big)  \nonumber \\
    &\simeq \Big( \big( \bigcup_{J\subseteq I_f} f_+ \big)
       \wedge \bigwedge_{j\in J} X_j  \Big) 
       \bigg\slash \underset{J\backslash (I_f\backslash L)\neq \varnothing}{\bigcup_{f\in\mathcal{S}_Q, L\subseteq I_f}}  \Big( f_+
       \wedge \bigwedge_{j\in I_f\backslash L} X_j 
   \wedge \bigwedge_{j\in  J\backslash ( I_f\backslash L)} a_j
       \Big)  \label{Equ:G-Deform-Retract} \\
    &\cong   \big( \bigcup_{J\subseteq I_f} f_+ \big) \wedge \bigwedge_{j\in J}  X_j \bigg\slash [\widehat{q}^J_0]\
    \overset{\eqref{Equ:sum-B-sigma-J}}{=} \ \big(F_{\cap J} \cup q_0\big) \wedge \bigwedge_{j\in J}  X_j. \label{Equ:G-colimit}
  \end{align}     
  
  The ``$\simeq$'' in~\eqref{Equ:G-Deform-Retract}
  is because each $A_j$ is contractible, so $A_j$ deformation retracts to its basepoint $a_j$.
 
  The ``$\cong$'' in~\eqref{Equ:G-colimit} is because
   $f_+
      \times \underset{j\in I_f\backslash L}{\scalebox{1.1}{$\prod$}} X_j 
   \times  \underset{j\in  J\backslash (I_f\backslash L)}{\scalebox{1.1}{$\prod$}} a_j$ 
   is equivalent to the basepoint $[\widehat{q}^J_0]$
   in $f_+
       \wedge \underset{j\in J}{\scalebox{1.3}{$\bigwedge$}} X_j$ since
       $a_j$ is the basepoint of $X_j$.
       
    When $J=\varnothing$, we have
    $$\mathrm{colim}\big(\widehat{\mathbf{G}}^J_{(\mathbb{X},\mathbb{A})+}\big) = \bigcup_{f\in \mathcal{S}_Q} f_+ = 
    Q \cup q_0 = F_{\cap \varnothing}\cup q_0.$$

  So by Theorem~\ref{Thm:Stable-Decomp-X-A}, we have homotopy equivalences:
  \begin{align*} 
  & \quad\ \,
 \mathbf{\Sigma}\big( (\mathbb{X},\mathbb{A})^{Q}_+ \big) =
  \mathbf{\Sigma}\big( \mathrm{colim}(\mathbf{D}_{(\mathbb{X},\mathbb{A})+}) \big) \simeq \bigvee_{J\subseteq [m]} \mathbf{\Sigma}\big( \mathrm{colim}(\widehat{\mathbf{D}}^J_{(\mathbb{X},\mathbb{A})+})\big) \\
   & \simeq  \bigvee_{J\subseteq [m]} \mathbf{\Sigma}\big( \mathrm{colim}\big(\widehat{\mathbf{G}}^J_{(\mathbb{X},\mathbb{A})+}\big)\big) \simeq  \bigvee_{J\subseteq [m]} 
  \mathbf{\Sigma} \Big( \big( F_{\cap J}\cup q_0\big)  \wedge \bigwedge_{j\in J}  X_j \Big). 
  \end{align*}
 
\[ \text{By Definition~\ref{Def:rtimes-ltimes}, we have
   $\big( F_{\cap J}\cup q_0\big)  \wedge \bigwedge_{j\in J}  X_j \cong 
    \begin{cases}
   F_{\cap J} \ltimes \underset{j\in J}{\scalebox{1}{$\bigwedge$}}  X_j ,  &  \text{if $J\neq \varnothing$}; \\
   Q\cup q_0 ,  &  \text{if $J=\varnothing$}.
 \end{cases} $} \]
  Then since $\mathbf{\Sigma}(Q\cup q_0)\simeq 
  S^1\vee\mathbf{\Sigma}(Q)$, the theorem is proved.
    \end{proof}
    
   The cohomology ring structure of
    $(\mathbb{X},\mathbb{A})^Q$ can be computed by Theorem~\ref{Thm:Cohomology-Ring-Isom-X-A}.
   In particular, if any combination of $F_{\cap J}$ and $X_j$'s satisfies the
    strong smash form of the K\"unneth formula over
    a coefficient ring $\mathbf{k}$,    
    we can give an explicit description of the cohomology ring of
    $(\mathbb{X},\mathbb{A})^Q$ with $\mathbf{k}$-coefficients. Indeed, by Theorem~\ref{Thm:Cohomology-Ring-Isom-X-A} and Theorem~\ref{Thm:Stable-Decomp-A-Contract} we obtain an isomorphism of rings
    \begin{equation} \label{Equ:Isom-Cohom-A-Contract}
     \bigoplus_{J\subseteq [m]} \widehat{\Pi}^*_{J}: \bigoplus_{J\subseteq [m]}\Big( H^*\big( F_{\cap J}; \mathbf{k} \big)\otimes \bigotimes_{j\in J} 
    \widetilde{H}^*(X_j;\mathbf{k})   \Big)
     \longrightarrow      
     H^*\big((\mathbb{X},\mathbb{A})^{Q};\mathbf{k} \big) 
     \end{equation}
   where the product $\circledast$ on the left-hand side is defined by~\eqref{Equ:Product-General-X-A} via the partial diagonal maps.    
    We will do some computation of this kind 
    in the next section to describe the equivariant cohomology ring of the moment-angle manifold $\mathcal{Z}_Q$.

 \section{Equivariant cohomology ring of $\mathcal{Z}_Q$
 and $\R\mathcal{Z}_Q$}
 \label{Sec:Equivariant-Cohom}
 
 Let $Q$ be a nice manifold with corners whose facets are $F_1,\cdots, F_m$. Since there is a canonical action of $(S^1)^m$ on $\mathcal{Z}_{Q}$ 
 (see~\eqref{Equ:Canon-Action-Complex}), it is a natural
 problem to compute the equivariant 
 cohomology ring of $\mathcal{Z}_{Q}$ with respect to this action.
 
  For a simple polytope $P$, it is shown in Davis-Januszkiewicz~\cite{DaJan91} that the equivariant cohomology of $\mathcal{Z}_{P}$ with integral coefficients
 is isomorphic to the face ring (or Stanley-Reisner ring)
 $\Z[P]$ of $P$ defined by
 \begin{equation*}
     \Z[P] = \Z[x_1,\cdots, x_m]\slash \mathcal{I}_P,
 \end{equation*} 
    where $\mathcal{I}_P$ is the ideal generated by all
    square-free monomials $x_{i_1}x_{i_2}\cdots x_{i_s}$
    such that $F_{i_1}\cap\cdots\cap F_{i_s} =\varnothing$ in $P$. A liner basis of $\Z[P]$ is given by
    \begin{equation} \label{Equ:Face-Ring-Basis}
     \{1\}\cup   \big\{  x^{n_1}_{i_1}\cdots x^{n_s}_{i_s}\,|\,
        F_{i_1}\cap\cdots\cap F_{i_s} \neq \varnothing,   \,  n_1>0,\cdots, n_s>0  \big\}. 
      \end{equation}
   We can also think of $\Z[P]$ as the face ring of $\partial P^*$ where $P^*$ is the dual simplicial polytope 
   of $P$ (see~\cite[Ch.\,3]{BP02}).

    For brevity, let $T^m=(S^1)^m$. By definition, the \emph{equivariant cohomology} of 
     $\mathcal{Z}_{Q}$, denoted by $H^*_{T^m}(\mathcal{Z}_{Q})$ is the cohomology of the \emph{Borel construction} 
     $$  ET^m \times_{T^m} \mathcal{Z}_{Q} = 
     ET^m \times \mathcal{Z}_{Q} \big\slash \sim $$
    where $(e,x)\sim (eg,g^{-1}x)$ for any $e\in 
    ET^m$, $x\in \mathcal{Z}_{Q}$ and $g\in T^m$.
    Here we let 
    $$ET^m = (ES^1)^m = (S^{\infty})^m.$$ 
    
    Associated to the Borel construction, there is a canonical fiber bundle
    \begin{equation} \label{Equ:Borel-Fiber}
      \mathcal{Z}_{Q} \rightarrow ET^m \times_{T^m} \mathcal{Z}_{Q} \rightarrow BT^m 
      \end{equation}
     where $BT^m = (BS^1)^m = (S^{\infty}\slash S^1)^m=(\mathbb{C}P^{\infty})^m$ is the \emph{classifying space} of $T^m$.

 By Lemma~\ref{Lem:Homeo-D2S1-Q-hat-ZQ}, $\mathcal{Z}_Q$ is equivariantly homeomorphic to $(D^2,S^1)^{Q}$. So computing the equivariant cohomology of $\mathcal{Z}_Q$ is equivalent to computing that for $(D^2,S^1)^{Q}$.

   By the colimit construction of $(D^2,S^1)^{Q}$ in~\eqref{Equ:ZQ-Colimit} and our notation for polyhedral products~\eqref{Equ:X-A-sigma-I-L}, the Borel construction
    \begin{align*} \label{Equ:ZQ-Equiv-Colimit}
  ET^m \times_{T^m} (D^2,S^1)^{Q} &=  
  \bigcup_{(f,L)\in\mathcal{P}_{Q}} 
 ET^m \times_{T^m} (D^2,S^1)^{(f,L)} \\
  &=  \bigcup_{(f,L)\in\mathcal{P}_{Q}} 
   (S^{\infty}\times_{S^1} D^2, S^{\infty}\times_{S^1} 
   S^1)^{(f,L)} \\
   &= (S^{\infty}\times_{S^1} D^2, S^{\infty}\times_{S^1} 
   S^1)^Q.
\end{align*}
  
    Then by the homotopy equivalence of the pairs
    $$(S^{\infty}\times_{S^1} D^2, S^{\infty}\times_{S^1} 
   S^1) \rightarrow (\mathbb{C}P^{\infty}, * )$$
   we can derive from Theorem~\ref{Thm:Homotopy-Colimit-Lemma} that there is a homotopy equivalence
   $$  (S^{\infty}\times_{S^1} D^2, S^{\infty}\times_{S^1} 
   S^1)^Q \simeq (\mathbb{C}P^{\infty}, * )^Q.$$
   
   We call $(\mathbb{C}P^{\infty}, * )^Q$ the \emph{Davis-Januszkiewicz space} of $Q$, denoted by $\mathcal{DJ}(Q)$.
   So the equivariant cohomology ring of $\mathcal{Z}_Q$
   is isomorphic to the ordinary cohomology ring of $\mathcal{DJ}(Q)$.  
   
   Similarly, we can prove that the Borel construction of $\R\mathcal{Z}_Q$
   with respect to the canonical $(\Z_2)^m$-action is
   $(\mathbb{R}P^{\infty}, * )^Q$.

 \begin{proof}[\textbf{Proof of Theorem~\ref{Thm:Equivariant-Cohomology-Z-Q}}]
 \ \vskip .1cm
 
   By the proof of Theorem~\ref{Thm:Stable-Decomp-A-Contract} and the fact that $H^*(\mathbb{C}P^{\infty})$ is torsion free, we can deduce from~\eqref{Equ:G-colimit} that
    $$  \widetilde{H}^*\big( \widehat{\mathbf{D}}^{J}_{(\mathbb{C}P^{\infty}, *)+} \big) \cong
     \widetilde{H}^*\big( \widehat{\mathbf{G}}^{J}_{(\mathbb{C}P^{\infty}, *)+} \big) \cong
   H^*( F_{\cap J})\otimes \bigotimes_{j\in J} 
    \widetilde{H}^*(\mathbb{C}P^{\infty}_{(j)}) , \ \forall J\subseteq [m] $$
  where $(\mathbb{C} P^{\infty})^m = \prod_{j\in [m]}
  \mathbb{C}P^{\infty}_{(j)}$. Then we obtain a ring isomorphism from~\eqref{Equ:Isom-Cohom-A-Contract}:
     \begin{equation*} 
     \bigoplus_{J\subseteq [m]} \widehat{\Pi}^*_{J}: \bigoplus_{J\subseteq [m]}\Big( H^*( F_{\cap J})\otimes \bigotimes_{j\in J} 
    \widetilde{H}^*(\mathbb{C}P^{\infty}_{(j)})   \Big)
     \longrightarrow      
     H^*\big((\mathbb{C}P^{\infty}, *)^{Q}\big)
     \cong H^*_{T^m}\big(\mathcal{Z}_{Q} \big) 
     \end{equation*}
     where the product $\circledast$ on the left-hand side is defined by~\eqref{Equ:Product-General-X-A} via the partial diagonal maps:
     
   \qquad  $    \mathrm{colim} \big(\widehat{\mathbf{D}}^{J\cup J'}_{(\mathbb{C}P^{\infty}, *)+}\big)  \xlongrightarrow{\ \ \, \widehat{\Delta}^{J,J'}_{J\cup J',Q_+}  \ } \mathrm{colim} \big(\widehat{\mathbf{D}}^{J}_{(\mathbb{C}P^{\infty}, *)+}\big) \wedge 
        \mathrm{colim} \big(\widehat{\mathbf{D}}^{J'}_{(\mathbb{C}P^{\infty}, *)+}\big) $.  
  
  \begin{exam} \label{Exam:CP-infty}
    If $Q = [0,1)$, the moment-angle manifold $\mathcal{Z}_{[0,1)} = D^2\backslash S^1$ whose Borel construction is homotopy equivalent to $\mathbb{C}P^{\infty}$. Then we have
    $$ H^*_{S^1}(\mathcal{Z}_{[0,1)}) \cong H^*(\mathbb{C}P^{\infty})\cong \Z[x], \ \mathrm{deg}(x)=2.$$
    
    The above ring isomorphism implies that the homomorphism
    $\Delta^*_{\mathbb{C}P^{\infty}}$ induced by the reduced diagonal map
    $\Delta_{\mathbb{C}P^{\infty}} : \mathbb{C}P^{\infty} \rightarrow 
  \mathbb{C}P^{\infty} \wedge \mathbb{C}P^{\infty}$
   on the integral cohomology is given by
\[
   \Delta^*_{\mathbb{C}P^{\infty}} : \widetilde{H}^*(\mathbb{C}P^{\infty} \wedge \mathbb{C}P^{\infty}) \cong 
  \underset{\overset{\qquad\qquad}{\scalebox{0.94}{$\qquad\quad\ \theta \otimes \theta' \xlongrightarrow{\quad \quad \ \ \ \ \ \qquad\ }  \theta\cup \theta'.$}}}{ \widetilde{H}^*(\mathbb{C}P^{\infty})\otimes 
  \widetilde{H}^*(\mathbb{C}P^{\infty}) \longrightarrow \widetilde{H}^*(\mathbb{C}P^{\infty}) }
\]
 
  \end{exam}

  Then by Lemma~\ref{Lem:Cross-Product-Pull-Back}
  and the above example, for any elements
   \begin{align*}
       u &= \phi \otimes \bigotimes_{j\in J} \theta_j,\ 
       \phi\in  H^*( F_{\cap J}), \theta_j\in 
       \widetilde{H}^*(\mathbb{C}P^{\infty}_{(j)}), \\      
       v &= \phi' \otimes \bigotimes_{j\in J'} \theta'_j,
       \  \phi'\in  H^*( F_{\cap J'}), \theta'_j\in 
       \widetilde{H}^*(\mathbb{C}P^{\infty}_{(j)}),
   \end{align*}
   $$u\circledast v = \big(\kappa^*_{J\cup J', J}(\phi)\cup
   \kappa^*_{J\cup J', J'}(\phi') \big) \otimes \bigotimes_{J\backslash J'} \theta_j\otimes 
   \bigotimes_{J'\backslash J} \theta'_j \otimes \bigotimes_{j\in J\cap J'} \big(\theta_j\cup \theta'_{j}\big),$$
   where $\kappa_{I', I} : F_{\cap I'}\rightarrow F_{\cap I}$ is the inclusion map for any subsets $I\subseteq I'\subseteq [m]$.
   
  Finally, since there is a graded ring isomorphism 
    $$ H^*\big((\mathbb{C}P^{\infty})^m \big) \cong \Z[x_1,\cdots, x_m], \ \mathrm{deg}(x_1)=\cdots=\mathrm{deg}(x_m)=2,$$    
   it is easy to check that $\bigoplus_{J\subseteq [m]}\Big( H^*( F_{\cap J})\otimes \bigotimes_{j\in J} 
    \widetilde{H}^*(\mathbb{C}P^{\infty}_{(j)}) \Big)$
    with the product $\circledast$
    is isomorphic to the topological face ring 
    $\Z\langle Q\rangle =
    \bigoplus_{J\subseteq [m]} H^*( F_{\cap J})\otimes R^J_{\Z}$ where $\bigotimes_{j\in J} 
    \widetilde{H}^*(\mathbb{C}P^{\infty}_{(j)})$ corresponds to $R^J_{\Z}$ (see~\eqref{Equ:RJ-K}). 
    
    By replacing $(D^2,S^1)$ with $(D^1,S^0)$, $(S^1)^m$ with $(\Z_2)^m$ and $\mathbb{C}P^{\infty}$ with $\R P^{\infty}$ in the above argument, and by the fact $H^*(\R P^{\infty};\Z_2)\cong \Z_2[x]$, $\mathrm{deg}(x)=1$,
   we obtain the parallel result for $\R\mathcal{Z}_Q$.     
   \end{proof}
   
   From the canonical fiber bundle associated to the Borel construction in~\eqref{Equ:Borel-Fiber}, we 
   have a natural $H^*(BT^m)$-module structure on
   $H^*_{T^m}(\mathcal{Z}_Q)$. By the identification
   $$H^*(BT^m) = \Z[x_1,\cdots, x_m],$$
    we can write the $H^*(BT^m)$-module structure on $H^*_{T^m}(\mathcal{Z}_Q)$
   as: for each $1\leq i \leq m$,
   \begin{equation} \label{Equ-BT-module-struc}
    x_i \cdot (\phi \otimes f(x)) = (1\otimes x_i)\star (\phi \otimes f(x)) \overset{\eqref{Equ:Face-Ring-Product}}{=} 
   \kappa^*_{J\cup\{i\}, J}(\phi) \otimes x_i f(x)
\end{equation}
   where $\phi\in  H^*(F_{\cap J})$ and
   $f(x)\in R^J_{\Z}$, $J\subseteq [m]$.

 \begin{exam}\label{Exam:Face-Ring-Polytope}
   Let $P$ be a simple polytope with facets $F_1,\cdots, F_m$. For a subset $J\subseteq [m]$, $F_{\cap J}$ is either empty or a face of $P$ and hence acyclic. So we can write the topological face ring of $P$ as
   \begin{align*}
      \Z\langle P\rangle  &\cong  \Big( \underset{J\subseteq[m]}{\bigoplus_{F_{\cap J}\neq \varnothing}}  R^J_{\Z},\ \star \Big) 
      \end{align*}
   where for any $f(x)\in R^J_{\Z}, f'(x)\in R^{J'}_{\Z}$ 
     with $F_{\cap J}\neq \varnothing, F_{\cap J'}\neq \varnothing$,   
   \begin{align*}   
    f(x) &\star f'(x)  =  \begin{cases}
   f(x)f'(x) ,  &  \text{if $F_{\cap(J\cup J')}\neq \varnothing$}; \\
   \   0 ,  &  \text{otherwise}.
 \end{cases}
 \end{align*}
 
  According to the linear basis of the face ring $\Z[P]$ in~\eqref{Equ:Face-Ring-Basis}, we can easily check that 
  $\Z\langle P\rangle$ is isomorphic to $\Z[P]$.
 \end{exam}
 
  \begin{thm} \label{Thm:Equiv-Free-Quotient}
   Let $Q$ be a nice manifold with corners with $m$ facets. 
  If a subtorus $H \subseteq T^m=(S^1)^m$ acts freely on
   $\mathcal{Z}_Q$ through the canonical action, 
    the equivariant cohomology ring with $\Z$-coefficients of the quotient space
   $\mathcal{Z}_Q\slash H$ with respect to the induced
   action of $T^m\slash H$ is isomorphic to 
   the topological face ring $\Z\langle Q\rangle$ of $Q$. 
  \end{thm}
  \begin{proof}
   Suppose $T^m \slash H \cong T^k$. Since $H$ acts freely on $\mathcal{Z}_Q$, the Borel constructions of $\mathcal{Z}_{Q}\slash H$ and $\mathcal{Z}_{Q}$ are homotopy equivalent by
  \begin{equation} \label{Equ-Homotopy-Borel}
        ET^m\times_{T^m} \mathcal{Z}_{Q} \cong 
     E H \times \Big( E\big(T^m\slash H \big) \times_{T^m\slash H} \big(\mathcal{Z}_{Q} \slash  H \big) \Big) \simeq ET^k\times_{T^k} 
     \mathcal{Z}_{Q}\slash H.
   \end{equation}
   
     So the equivariant cohomology ring of
   $\mathcal{Z}_Q\slash H$ is isomorphic to the equivariant cohomology ring of $\mathcal{Z}_Q$.  
    Then the theorem follows from 
 Theorem~\ref{Thm:Equivariant-Cohomology-Z-Q}.
  \end{proof}
  
  In Theorem~\ref{Thm:Equiv-Free-Quotient},
  the group homomorphism $T^m \rightarrow T^m\slash H\cong T^k$ induces a commutative diagram
  \[
   \xymatrix{
           E T^m \ar[d] \ar[r]
                & E \big( T^m\slash H \big) \ar[d]  \\
           BT^m  \ar[r] & B \big( T^m\slash H \big).
                 } 
  \]
  which, along with the maps in~\eqref{Equ-Homotopy-Borel}, induce the diagram 
  \begin{equation} \label{Diag-Borel-Construc}
   \xymatrix{
           E T^m \times_{T^m} \mathcal{Z}_Q \ar[d] \ar[r]^{\simeq\qquad\quad }
                & E\big( T^m\slash H \big)\times_{T^m\slash H} \big(\mathcal{Z}_{Q} \slash  H \big)\ar[d]  \\
           BT^m  \ar[r] & B \big(T^m\slash H \big).
                 } 
  \end{equation}
  
  We can describe the natural $H^*\big(B \big(T^m\slash H \big)\big)$-module structure of the integral equivariant cohomology ring of $\mathcal{Z}_Q\slash H$ as follows.
  The inclusion $H \hookrightarrow T^m$ 
  induces a monomorphism $\varphi_H: \Z^{m-k}\rightarrow \Z^m$ whose image is a direct summand in $\Z^m$. This determines an integer $m\times (m-k)$ matrix
  $S=(s_{ij})$ if we choose a basis for each of $\Z^{m-k}$
  and $\Z^m$. Then since the image of $\varphi_H$ is 
  a direct summand in $\Z^m$, there is an integer $k\times m$ matrix $R=(r_{ij})$ of rank $k$ such that 
  $R\cdot S=0$ which defines the homomorphism $T^m\rightarrow T^m\slash H$.
  
 If we write
  $H^*\big(B \big(T^m\slash H \big)\big)=H^*(B T^k) =\Z[y_1,\cdots,y_k]$, it follows from the diagram~\eqref{Diag-Borel-Construc} that the natural $H^*\big(B \big(T^m\slash H \big)\big)$-module structure of the integral equivariant cohomology ring of 
  $\mathcal{Z}_Q\slash H$ is determined by the formula in~\eqref{Equ-BT-module-struc} along with the map $H^*\big(B \big(T^m\slash H \big)\big)\rightarrow H^*(B T^m)$ given by:
   \begin{align*}
    \Z[y_1,\cdots,y_k] &\longrightarrow \Z[x_1,\cdots, x_m]\\
    y_i &\longmapsto r_{i1}x_1 +\cdots + r_{im}x_m.
   \end{align*}
     The above formula is parallel to the formula given  in~\cite[Theorem 7.37]{BP02} (where $Q$ is a simple polytope).
   
  \begin{rem}
   If a subtorus $H\subseteq T^m$ of dimension $m-\dim(Q)$ 
   acts freely on $\mathcal{Z}_Q$ through the canonical action, the quotient space
   $\mathcal{Z}_Q\slash H$ with the induced
   action of $T^m\slash H\cong T^{\dim(Q)}$ can be considered as a 
   generalization of \emph{quasitoric manifold} over a simple polytope defined by Davis and Januszkiewicz~\cite{DaJan91}.  
  \end{rem}
  
   The following is an application of Theorem~\ref{Thm:Equiv-Free-Quotient}
   to locally standard torus actions on closed manifolds. Recall that an action of $T^n$ on a closed $2n$-manifold $M^{2n}$ is called \emph{locally standard} (see~\cite[\S\,1]{DaJan91}) if every point in $M^{2n}$ has a $T^n$-invariant neighborhood that is weakly equivariantly diffeomorphic an open subset of
   $\mathbb{C}^n$ invariant under the \emph{standard $T^n$-action}:
   \[  (g_1,\cdots, g_n)\cdot (z_1,\cdots, z_n) = (g_1z_1,\cdots,g_nz_n), \ g_i\in S^1, z_i\in \mathbb{C}, 1\leq i \leq n.  \]
   
  \begin{cor} \label{Cor:Equiv-Cohom-Free-Quotient}
   Let $M^{2n}$ be a closed smooth $2n$-manifold with a smooth locally standard $T^n$-action and the free part of the action is a trivial $T^n$-bundle. Then the integral equivariant cohomology ring $H^*_{T^n}(M^{2n})$ of $M^{2n}$ is isomorphic to the topological face ring $\Z\langle M^{2n}\slash T^n \rangle$.
  \end{cor} 
  \begin{proof}
    The orbit space $Q=M^{2n}\slash T^n$ is a smooth nice manifold with corners since the $T^n$-action is locally standard and smooth. Then $Q$ is triangulable (by~\cite{John83}) and hence all our theorems can be applied to $Q$.
   In addition, using the characteristic function argument
   in Davis-Januszkiewicz~\cite{DaJan91} (also see~\cite[\S\,4.2]{MasPanov06} or~\cite{Yos11}), we can prove that $M^{2n}$ is a free quotient space
   of $\mathcal{Z}_{Q}$ by a canonical action of some torus.
   Then this corollary follows from Theorem~\ref{Thm:Equiv-Free-Quotient}. 
   \end{proof}
   
   \begin{rem} 
    The equivariant cohomology ring
   $H^*_{T^n}(M^{2n})$ in the above corollary was also computed by
   Ayzenberg-Masuda-Park-Zeng~\cite[Proposition 5.2]{AntonMasParkZeng17}
   under an extra assumption that 
   all the proper faces of $M^{2n}\slash T^n$ are acyclic.
   We leave it as an exercise for the reader to check that the formula for $H^*_{T^n}(M^{2n})$ given in~\cite{AntonMasParkZeng17} is isomorphic to $\Z\langle M^{2n}\slash T^n \rangle$.  
   \end{rem}

    \section{Generalizations} \label{Sec:Generalization}
      
        Let $Q$ be a nice manifold with corners
        with facets
        $\mathcal{F}(Q) =\{F_1,\cdots, F_m \}$.  
      Observe that neither in the construction of $\mathcal{Z}_Q$ nor in the proof of Theorem~\ref{Thm:Stable-Decomp-Main} and Theorem~\ref{Thm:Cohomology-Ring-Isom} do we really use the connectedness of each facet $F_j$. So we have the following generalization of $\mathcal{Z}_Q$.

   Let $\mathcal{J}=\{ J_1,\cdots, J_k\}$ 
       be a \emph{partition} of $[m]=\{1,\cdots, m\}$.
        i.e. the $J_i$'s are disjoint subsets of $[m]$
       with $J_1 \cup\cdots \cup J_k =[m]$. So $\partial Q = F_{J_1}\cup \cdots \cup F_{J_k}$. Moreover, we require
    $\mathcal{J}$ to satisfy the following condition in our discussion: 
      \begin{equation}\label{Condition-J}
      \text{for any $1\leq i \leq k$, if $j, j' \in J_i$, then $F_j\cap F_{j'}=\varnothing$. }
       \end{equation}      
       From $Q$ and the partition $\mathcal{J}$, we can construct the following manifold. 
         
    Let $\{ e_1,\cdots, e_k\}$ be a unimodular basis of $\Z^k$. Let $\mu: \mathcal{F}(Q) \rightarrow \Z^k$ be the map which sends all the facets in $F_{J_i}$ to $e_i$ for every $1\leq i \leq k$. Define         
       \begin{equation*} 
       \mathcal{Z}_{Q,\mathcal{J}} := Q\times (S^1)^k \slash \sim
     \end{equation*} 
     where $(x,g) \sim (x',g')$ if and only if $x=x'$ and $g^{-1}g' \in \mathbb{T}^{\mu}_x$
   where $\mathbb{T}^{\mu}_x$ is the subtorus of $(S^1)^k=\R^k\slash \Z^k$ determined by
   the linear subspace of $\R^k$ spanned by the set $\{ \mu(F_j) \, |\, x\in F_j \}$.
  There is a canonical action of $(S^1)^k$ on $\mathcal{Z}_{Q,\mathcal{J}}$ defined by:
   \begin{equation} \label{Equ:Canon-Action-Complex-J}
    g' \cdot [(x,g)] = [(x,g'g)], \ x\in Q,\ g,g'\in (S^1)^k. 
   \end{equation} 
   
   If $\mathcal{J}_0=\big\{ \{1\},\cdots, \{m\}\big\}$
    is the trivial partition of $[m]$, then $  \mathcal{Z}_{Q,\mathcal{J}_0}  = \mathcal{Z}_Q$.  
      
   Note that here $\{ F_{J_i}\}$
    play the role of facets $\{ F_j\}$ in the definition of $\mathcal{Z}_Q$. But $F_{J_i}$
    may not be connected. Using the term defined in Davis~\cite{Da83}, the decomposition of $\partial Q$ into $\{ F_{J_i}\}$ is called a \emph{panel structure on $Q$} and each
    $F_{J_i}$ is called a \emph{panel}.

 \begin{rem}
   For a general partition $\mathcal{J}$ of $[m]$, it is possible that $F_j\cap F_{j'}\neq \varnothing$ for some $j,j' \in J_i$. Although the definition of $\mathcal{Z}_{Q,\mathcal{J}}$ still makes sense in the general setting,
   the orbit space of the $(S^1)^k$-action on $\mathcal{Z}_{Q,\mathcal{J}}$ may not be $Q$ (as a
manifold with corners). It would be $Q$ with some corners smoothed. But for a general partition of $[m]$, one can always reduce to the case where the  condition~\eqref{Condition-J} is satisfied by smoothing the corners of the orbit space. 
\end{rem}

    For any subset $\omega\subseteq [k] =\{1,\cdots, k\}$, let 
      $$F_{\omega} = \bigcup_{i\in \omega} F_{J_i},\ F_{\varnothing}=\varnothing,
       \ \ F_{\cap\omega} = \bigcap_{i\in \omega} F_{J_i}, \ F_{\cap\varnothing}=Q.$$

  \begin{thm} \label{Thm:Main-General-2}
   Let $Q$ be a nice manifold with corners with facets
   $F_1,\cdots, F_m$.
    For any partition $\mathcal{J}=\{ J_1,\cdots, J_k\}$ 
     of $[m]=\{1,\cdots, m\}$, we have
      \begin{equation*}  
           \mathbf{\Sigma}\big( \mathcal{Z}_{Q,\mathcal{J}} \big) \simeq \bigvee_{\omega\subseteq [k]} 
           \mathbf{\Sigma}^{|\omega|+1}( Q\slash F_{\omega}),\ \
           H^p \big( \mathcal{Z}_{Q,\mathcal{J}} \big)\cong \bigoplus_{\omega\subseteq [k]} H^{p-|\omega|}(Q,F_{\omega}), \, \forall p\in \Z. 
  \end{equation*} 
      \end{thm}     
  \begin{proof} 
       We can generalize the rim-cubicalization of $Q$ in Section~\ref{Subsec:Rim-Cubical} as follows. For any
       face $f$ of $Q$, let
       $$I^{\mathcal{J}}_f=\{ i\in [k]\,|\, f\subseteq F_{J_i} \} \subseteq [k]. $$
       
       Then define
        \begin{equation*} 
       \widehat{f}^{\mathcal{J}} = f 
       \times \prod_{i\in I^{\mathcal{J}}_f} [0,1]_{(i)} 
   \times \prod_{i\in [k]\backslash I^{\mathcal{J}}_f} 1_{(i)}
   \end{equation*}
   \begin{equation*}
      \widehat{Q}^{\mathcal{J}} = \bigcup_{f\in \mathcal{S}_Q}   \widehat{f}^{\mathcal{J}}  \subseteq Q\times [0,1]^k.\
   \end{equation*}

  By the same argument as in the proof of Lemma~\ref{Lem:Homeo-Q-Q-hat}, we can show that 
 $\widehat{Q}^{\mathcal{J}}$ with faces $\widehat{f}^{\mathcal{J}}$ is homeomorphic to
 $Q$ as a manifold with corners.
 The partition $\mathcal{J}$ of the 
 facets of $Q$ naturally induces a partition of
 the corresponding facets of $\widehat{Q}^{\mathcal{J}}$,  also denoted by $\mathcal{J}$. So we have
 $\mathcal{Z}_{\widehat{Q}^{\mathcal{J}},\mathcal{J}} \cong \mathcal{Z}_{Q,\mathcal{J}}$.
      
  For any
       face $f$ of $Q$, let 
   \begin{align*} 
      (D^2,S^1)^f_{\mathcal{J}} &:= f
       \times \prod_{i\in I^{\mathcal{J}}_f} D^2_{(i)} 
   \times \prod_{i\in [k]\backslash I^{\mathcal{J}}_f} S^1_{(i)},\\ 
  (D^2,S^1)^{Q}_{\mathcal{J}} 
  &:= \bigcup_{f\in \mathcal{S}_Q}   (D^2,S^1)^f_{\mathcal{J}} \subseteq Q\times (D^2)^k .
     \end{align*}  
      
   There is a canonical $(S^1)^k$-action on $ (D^2,S^1)^{Q}_{\mathcal{J}} $ induced from the
   canonical $(S^1)^k$-action on $Q\times (D^2)^k$.
   And parallel to Lemma~\ref{Lem:Homeo-D2S1-Q-hat-ZQ}, we can prove that there is an equivariant homeomorphism
      from $ (D^2,S^1)^{Q}_{\mathcal{J}} $ to 
      $\mathcal{Z}_{\widehat{Q}^{\mathcal{J}},\mathcal{J}} \cong \mathcal{Z}_{Q,\mathcal{J}}$.
      
      For any subset $L\subseteq I^{\mathcal{J}}_f$, let
      \[   (D^2,S^1)^{(f,L)}_{\mathcal{J}} := f
       \times \prod_{i\in I^{\mathcal{J}}_f\backslash L} D^2_{(i)} 
   \times \prod_{i\in [k]\backslash (I^{\mathcal{J}}_f \backslash L)} S^1_{(i)}. \]
      
       We can easily translate the proof of 
  Theorem~\ref{Thm:Stable-Decomp-Main} to obtain the
   desired
   stable decomposition of $
   \mathcal{Z}_{Q,\mathcal{J}}\cong (D^2,S^1)^Q_{\mathcal{J}}$ by the following correspondence of symbols.      
     \begin{align*}
  \text{The proof of Theorem~\ref{Thm:Stable-Decomp-Main}}\quad\ &\  
      \text{The proof of Theorem~\ref{Thm:Main-General-2}}\\
         J \subseteq [m]  \qquad \qquad\ \ \ & \qquad\qquad\ \  \omega \subseteq [k]\\
       F_J \  \ \ \qquad\qquad\quad\ & \qquad\qquad\;\,\; \ \ \, F_{\omega}\\
        I_f  \subseteq [m]  \qquad \qquad\ \ \ & \qquad\qquad\   I^{\mathcal{J}}_f \subseteq [k] \\
     D^2_{(j)}, S^1_{(j)}, j\in [m]  \quad \quad\ \ \ & \quad \quad \  \ \  D^2_{(i)}, S^1_{(i)}, i\in [k]\\
       (D^2,S^1)^{(f,L)}   \qquad\qquad \, & \qquad\quad\
       \ (D^2,S^1)^{(f,L)}_{\mathcal{J}}
        \end{align*} 
   \end{proof}
   
  \begin{rem}
    Theorem~\ref{Thm:Main-General-2}
     is an analogue of~\cite[Theorem 1.3]{Yu19}.   
   \end{rem}

  To describe the cohomology ring of $\mathcal{Z}_{Q,\mathcal{J}}$, let
  \begin{equation} \label{Equ:Def-R-Q}
    \mathcal{R}^*_{Q,\mathcal{J}} :=\bigoplus_{\omega \subseteq [k]} H^*(Q,F_{\omega}).
    \end{equation} 
  
  There is a graded ring structure $\Cup_{\mathcal{J}}$ on $\mathcal{R}^*_{Q,\mathcal{J}}$ defined as follows.
  \begin{itemize}
       \item If $\omega\cap \omega' \neq \varnothing$,
       then
     $ H^*(Q,F_{\omega}) \otimes H^{*}(Q,F_{\omega'}) \overset{\Cup_{\mathcal{J}}}{\longrightarrow}
    H^{*}(Q,F_{\omega\cup \omega'})$ is trivial.
    
     \item If $\omega\cap \omega'=\varnothing$, then
  $ H^*(Q,F_{\omega}) \otimes H^{*}(Q,F_{\omega'}) \overset{\Cup_{\mathcal{J}}}{\longrightarrow}
    H^{*}(Q,F_{\omega\cup \omega'})$ is the relative cup product $\cup$.
     
    \end{itemize}

   To describe the equivariant cohomology ring of $\mathcal{Z}_{Q,\mathcal{J}}$, let 
     $$ \mathbf{k}^{\mathcal{J}}\langle Q\rangle := \bigoplus_{\omega\subseteq [k]} H^*( F_{\cap \omega};\mathbf{k})\otimes R^{\omega}_{\mathbf{k}}. $$
   where the product on $\mathbf{k}^{\mathcal{J}}\langle Q\rangle $
   is defined in the same way as $\mathbf{k}\langle Q\rangle$ in
   Definition~\ref{Def:Top-Face-Ring}.

    The following theorem generalizes
 Theorem~\ref{Thm:Cohomology-Ring-Isom} and 
 Theorem~\ref{Thm:Equivariant-Cohomology-Z-Q}. The proof 
  is omitted since it is completely parallel to the proof of these two theorems.
 
 \begin{thm}  \label{Thm:Main-General-Cohomology-2}
    Let $Q$ be a nice manifold with corners with $m$ facets $F_1,\cdots, F_m$ and let $\mathcal{J}=\{J_1,\cdots, J_k\}$ be a partition of $[m]$. 
    \begin{itemize}
     \item There is a ring isomorphism (up to a sign)
     from $(\mathcal{R}^*_{Q,\mathcal{J}},\Cup_{\mathcal{J}})$ to the integral cohomology ring of $\mathcal{Z}_{Q,\mathcal{J}}$. Moreover,
    we can make this ring isomorphism degree-preserving by shifting 
    the degrees of all
    the elements in $H^*(Q,F_{\omega})$ up by $|\omega|$ for every $\omega\subseteq [k]$.
    
     \item There is a graded ring isomorphism from the equivariant cohomology ring of $\mathcal{Z}_{Q,\mathcal{J}}$ with integral coefficients to $\Z^{\mathcal{J}}\langle Q\rangle$ by
     choosing $\mathrm{deg}(x_i)=2$ for all $1\leq i \leq k$.
     \end{itemize}  
 \end{thm}

     By combining the constructions in Theorem~\ref{Thm:Main-General-X-A} and Theorem~\ref{Thm:Main-General-2}, we have the following  definitions which provide the most general setting for our study.  
      
       Let $\mathcal{J}=\{ J_1,\cdots, J_k\}$ 
       be a \emph{partition} of $[m]=\{1,\cdots, m\}$ and
       let 
       $$(\mathbb{X},\mathbb{A}) = \{ (X_i,A_i,a_i) \}^k_{i=1}$$
   where $X_i$ and $A_i$ are CW-complexes with a basepoint $a_i\in A_i\subseteq X_i$. 
   
   For any face $f$ of $Q$, let
   \begin{equation*} 
      (\mathbb{X},\mathbb{A})^f_{\mathcal{J}} := f 
       \times \prod_{i\in I^{\mathcal{J}}_f} X_i 
   \times \prod_{i\in [k]\backslash I^{\mathcal{J}}_f} A_i,
   \end{equation*}
     \begin{equation*} 
  (\mathbb{X},\mathbb{A})^Q_{\mathcal{J}}
  := \bigcup_{f\in \mathcal{S}_Q}   (\mathbb{X},\mathbb{A})^f_{\mathcal{J}}  
   \subseteq Q\times \prod_{i\in [k]} X_i.
   \end{equation*}

   The following theorem generalizes 
   Theorem~\ref{Thm:Main-General-X-A} and
   Theorem~\ref{Equ:Cohomology-Ring-Iso-X-A}.

     \begin{thm} \label{Thm:Main-General-J-1}
     Let $Q$ be a nice manifold with corners
      with facets $F_1,\cdots, F_m$. Let $ (\mathbb{X},\mathbb{A}) = \{ (X_i,A_i,a_i) \}^k_{i=1}$ where each $X_i$ is contractible and each $A_i$ is either connected or is a disjoint union of a connected CW-complex with its basepoint. Then for any partition $\mathcal{J}=\{J_1,\cdots,J_k\}$ of $[m]$, there is a homotopy equivalence
      \begin{equation*}  
           \mathbf{\Sigma}\Big(  (\mathbb{X},\mathbb{A})^Q_{\mathcal{J}}  \Big) \simeq \bigvee_{\omega\subseteq [k]} 
           \mathbf{\Sigma} \Big( Q\slash F_{\omega}\wedge\bigwedge_{i\in\omega} A_i \Big).
           \end{equation*} 
           
  In addition,  there is a ring isomorphism
      \[  
         \Big( \bigoplus_{\omega\subseteq [k]} \widetilde{H}^*\big(Q\slash F_{\omega}\wedge\bigwedge_{i\in\omega} A_i \big), \circledast \Big) \longrightarrow
         \widetilde{H}^*\big( (\mathbb{X},\mathbb{A})^Q_{\mathcal{J}} \big) \]  
         where $\circledast$ is defined in the same way
         as in~\eqref{Equ:Product-General-X-A}.
      \end{thm}

      In particular, for
      $(\mathbb{D},\mathbb{S}) =
    \big\{ \big( D^{n_i+1}, S^{n_i}, a_i \big)\big\}^k_{i=1}$, we can describe the integral cohomology ring of $(\mathbb{D},\mathbb{S})^Q_{\mathcal{J}} $ explicitly as follows.      
   Define a graded ring structure $\Cup^{(\mathbb{D},\mathbb{S})}_{\mathcal{J}}$ on $\mathcal{R}^*_{Q,\mathcal{J}} $ according to
    $(\mathbb{D},\mathbb{S})$ by:
  \begin{itemize}
   \item If $\omega\cap \omega'=\varnothing$ or $\omega\cap \omega' \neq \varnothing$ but $n_i=0$ for all $i\in \omega\cap \omega'$, then \\   
  $ H^*(Q,F_{\omega}) \otimes H^{*}(Q,F_{\omega'}) \xlongrightarrow{\ \Cup^{(\mathbb{D},\mathbb{S})}_{\mathcal{J}}}
    H^{*}(Q,F_{\omega\cup \omega'})$ is the relative cup product.\vskip .1cm
          
     \item If $\omega\cap \omega' \neq \varnothing$ and there exists $n_i\geq 1$ for some $i\in \omega\cap \omega'$, then
    \\   
  $ H^*(Q,F_{\omega}) \otimes H^{*}(Q,F_{\omega'}) \xlongrightarrow{\ \Cup^{(\mathbb{D},\mathbb{S})}_{\mathcal{J}}}
    H^{*}(Q,F_{\omega\cup \omega'})$ is trivial.
    \end{itemize}
    
  \begin{thm} \label{Thm:Main-General-J-2}
     Let $Q$ be a nice manifold with corners with facets $F_1,\cdots, F_m$. For any partition $\mathcal{J}=\{J_1,\cdots,J_k\}$ of $[m]$ and $(\mathbb{D},\mathbb{S}) =
    \big\{ \big( D^{n_i+1}, S^{n_i}, a_i \big)\big\}^k_{i=1}$, there is a homotopy equivalence
      \begin{equation*}  
           \mathbf{\Sigma}\big(  (\mathbb{D},\mathbb{S})^Q_{\mathcal{J}}  \big) \simeq \bigvee_{\omega\subseteq [k]} 
            \mathbf{\Sigma}^{1+\sum_{i\in \omega} n_i} \big( Q\slash F_{\omega} \big).
           \end{equation*}                       
   And there is a ring isomorphism (up to a sign) from
      $ 
         \big(\mathcal{R}^*_{Q,\mathcal{J}},\Cup^{(\mathbb{D},\mathbb{S})}_{\mathcal{J}} \big)$ to the integral cohomology ring of $(\mathbb{D},\mathbb{S})^Q_{\mathcal{J}}$. Moreover, we can make this 
         ring isomorphism degree-preserving by shifting the degrees
         of the elements in $H^*(Q,F_{\omega})$ for every $\omega\subseteq [k]$.
      \end{thm}   
      
    When $(\mathbb{D},\mathbb{S})=\{ (D^1,S^0,a_0)\}^k_{i=1}$, we denote $(\mathbb{D},\mathbb{S})^Q_{\mathcal{J}}$ also by $\R\mathcal{Z}_{Q,\mathcal{J}}$ which is the real analogue of $\mathcal{Z}_{Q,\mathcal{J}}$. Then we have the following corollary which generalizes
      Corollary~\ref{Cor:Real-Moment} and Theorem~\ref{Thm:Equivariant-Cohomology-Z-Q}.
      
  \begin{cor} \label{Cor:Real-Moment-General}
     Let $Q$ be a nice manifold with corners with facets $F_1,\cdots, F_m$. Then for any partition $\mathcal{J}=\{J_1,\cdots,J_k\}$ of $[m]$, we have
     $$  \mathbf{\Sigma}( \R \mathcal{Z}_{Q,\mathcal{J}} ) \simeq \bigvee_{\omega\subseteq [k]} 
           \mathbf{\Sigma}( Q\slash F_{\omega}), \ \ H^p(\R\mathcal{Z}_{Q,\mathcal{J}}) \cong \bigoplus_{\omega\subseteq [k]} H^p(Q, F_{\omega}), \ \forall p\in \Z.$$
\begin{itemize}
     \item The integral cohomology ring of $\R\mathcal{Z}_{Q,\mathcal{J}}$ is isomorphic as a graded ring to the ring $(\mathcal{R}^*_{Q,\mathcal{J}}, \cup)$
     where $\cup$ is the relative cup product
      $$ H^*(Q,F_{\omega}) \otimes H^{*}(Q,F_{\omega'}) \overset{\cup}{\longrightarrow}
    H^{*}(Q,F_{\omega\cup \omega'}), \ \forall\, \omega,\omega'\subseteq [k].$$ 
     
     \item There is a graded ring isomorphism from the equivariant $\Z_2$-cohomology ring of $\R\mathcal{Z}_{Q,\mathcal{J}}$ to $\Z_2^{\mathcal{J}}\langle Q\rangle$ by
     choosing $\mathrm{deg}(x_i)=1$ for all $1\leq i \leq k$.
     \end{itemize} 
   \end{cor}

     The proofs of Theorem~\ref{Thm:Main-General-J-1}, Theorem~\ref{Thm:Main-General-J-2} and Corollary~\ref{Cor:Real-Moment-General} are almost the same as their counterparts
     in Section~\ref{Sec:Polyhedral-Prod} and Section~\ref{Sec:Equivariant-Cohom}, hence omitted.
  
  For any partition $\mathcal{J}=\{ J_1,\cdots, J_k\}$ of $[m]$,
  \begin{itemize}
  \item We can think of $\mathcal{Z}_{Q,\mathcal{J}}$ as the quotient space of $\mathcal{Z}_Q$ by the canonical action of an
  $(m-k)$-dimensional subtorus $\mathbb{T}^{\mathcal{J}}$ of $(S^1)^m$ determined by (see~\eqref{Equ:glue-back-complex})
   $$\qquad\quad \ \{ \lambda(F_{j}) - \lambda(F_{j'}) \,|\, j, j' \ \text{belong to the same}\ J_i \ \text{for some}\  1\leq i \leq k \}  \subseteq \Z^m.$$

  \item Similarly, we can think of $\R\mathcal{Z}_{Q,\mathcal{J}}$ as the quotient space of $\R\mathcal{Z}_Q$ by the canonical action of a subgroup of rank $m-k$ in $(\Z_2)^m$.
  \end{itemize}
  
  Note that the canonical action of
  $\mathbb{T}^{\mathcal{J}}$ on $\mathcal{Z}_Q$ may not be free. But when the action is free, the integral equivariant cohomology ring of 
  $\mathcal{Z}_Q \slash \mathbb{T}^{\mathcal{J}} = 
  \mathcal{Z}_{Q,\mathcal{J}}$ is isomorphic to
 $\Z\langle Q \rangle$ by Theorem~\ref{Thm:Equiv-Free-Quotient}. So by Theorem~\ref{Thm:Main-General-Cohomology-2}, $\Z^{\mathcal{J}}\langle Q \rangle$ is isomorphic as a ring to $\Z\langle Q \rangle$ 
  in this case. But this ring isomorphism is not obvious from the algebraic definition of $\Z^{\mathcal{J}}\langle Q \rangle$ and $\Z\langle Q \rangle$.
  
  \begin{rem}
  For any partition $\mathcal{J}=\{ J_1,\cdots, J_k\}$ of $[m]$ with $k=\dim(Q)$, the
  $\mathcal{Z}_{Q,\mathcal{J}}$ and $\R\mathcal{Z}_{Q,\mathcal{J}}$ can be considered as a generalization of
  the \emph{pull-back from the linear model} (see~\cite[Example 1.15]{DaJan91}) in the study of quasitoric manifolds and small covers.
  \end{rem}
     
   \section*{Acknowledgment}
  This work is partially supported by the National 
 Natural Science Foundation of China (grant no.11871266) and 
 the PAPD (priority academic program development) of Jiangsu higher education institutions.  The author wants to thank Professor Tony Bahri for some
inspiring comment on an earlier version of the paper
and thank Feifei Fan for pointing out a sign problem in the paper.
The author also wants to thank the referee for some valuable suggestions.


\begin{thebibliography}{99}

 \bibitem{AntonBuch11}
 A.~Ayzenberg and V.~M.~Buchstaber, \emph{Nerve complexes and moment-angle spaces of convex polytopes}, Proc. Steklov Inst. Math. 275 (2011), no.~\textbf{1}, 15--46.
 
 \bibitem{AntonMasParkZeng17}
 A.~Ayzenberg, M.~Masuda, S.~Park and H.~Zeng, \emph{Cohomology of toric origami manifolds with acyclic proper faces}, 
 J. Symplectic Geom. 15 (2017), no.~\textbf{3}, 645--685.
 
 \bibitem{BBCG10}
   A.~Bahri, M.~Bendersky, F.~Cohen and S.~Gitler, 
    \emph{The Polyhedral Product Functor: a method of decomposition for moment-angle complexes, 
     arrangements and related spaces},
      Adv. Math.~\textbf{225} (2010) no.~\textbf{3}, 1634--1668. 
      
\bibitem{BBCG12}
   A.~Bahri, M.~Bendersky, F.~Cohen and S.~Gitler, 
   \emph{Cup-products for the polyhedral product functor}, Math.~Proc.~Cambridge~Philos.~Soc. 153 (2012), no.~\textbf{3}, 457--469.    

\bibitem{BBCG17}
 A.~Bahri, M.~Bendersky, F.~Cohen and S.~Gitler,
\emph{A spectral sequence for polyhedral products},
Adv. Math. 308 (2017), 767--814.

\bibitem{BBCG20}
 A.~Bahri, M.~Bendersky, F.~Cohen and S.~Gitler, \emph{A Cartan formula for the cohomology of polyhedral products and its application to the ring structure},
  arXiv:2009.06818.
 

\bibitem{BasBuchPanov04}
 I.~V.~Baskakov, V.~M.~Buchstaber and T.~E.~Panov,  \emph{Algebras of cellular cochains, and torus actions}, (Russian) Uspekhi Mat. Nauk 59 (2004), no.~\textbf{3} 
 (357), 159--160; translation in Russian Math. Surveys 59 (2004), no.~\textbf{3}, 562--563. 
 
\bibitem{BP00}
 V.~M.~Buchstaber and T.~E.~Panov, \emph{Actions of tori, combinatorial topology and homological algebra}, Russian
Math. Surveys, 55 (2000), 825--921. 
 
\bibitem{BP02} 
  V.~M.~Buchstaber and T.~E.~Panov,
 \emph{Torus actions and their applications in topology and
 combinatorics}, University Lecture Series, \textbf{24}.
 American Mathematical Society, Providence, RI, 2002.   

 \bibitem{BP15}
V.~M.~Buchstaber and T.~E.~Panov, \emph{Toric Topology},
 Mathematical Surveys and Monographs, Vol.~204, 
 American Mathematical Society, Providence, RI, 2015.
 
 \bibitem{CaiLi17}
   L.~Cai, \emph{On products in a real moment-angle manifold}, J. Math. Soc. Japan 69 (2017), no.~\textbf{2}, 503--528.
   
  \bibitem{CanGuilPires11}  
  A.~Cannas da Silva, V.~Guillemin and A.~R.~Pires, \emph{Symplectic origami},
Int. Math. Res. Not. IMRN 2011, no.~\textbf{18}, 4252--4293.    
 
\bibitem{Chou64}
   W.~Chow, \emph{On unmixedness theorem}.
Amer. J. Math. 86 (1964), 799--822.

\bibitem{Da83}
 M.~W.~Davis, 
 \emph{Groups generated by reflections and aspherical manifolds not covered by
Euclidean space}, Ann. of Math. 117 (1983), 293--324. 

\bibitem{Da87}
 M.~W.~Davis, \emph{The homology of a space on which a reflection group acts}, Duke Math. J. 55 (1987), no.~\textbf{1}, 97--104.  
 
\bibitem{Davisbook08}
M.~W.~Davis, \emph{The geometry and topology of Coxeter groups}, London Mathematical Society Monographs Series, 32. Princeton University Press, Princeton, NJ, 2008. 
  
\bibitem{DaJan91}  M.~W.~Davis and T.~Januszkiewicz, \textit{Convex polytopes,
Coxeter orbifolds and torus actions}, Duke Math. J. \textbf{62}
(1991), no.~\textbf{2}, 417--451.

\bibitem{Franz06}
  M.~Franz, \emph{On the integral cohomology of smooth toric varieties}, Proc. Steklov Inst. Math. 252 (2006), 53--62.
  
\bibitem{FrobHoa92}
  R.~Fr\"oberg and L.~Hoa, \emph{Segre products and Rees algebras of face rings}, Comm. Algebra 20 (1992), no. 11, 3369--3380.

  
\bibitem{Hatcher02}
   A.~Hatcher, \emph{Algebraic topology}, Cambridge University Press, Cambridge, 2002. 
     
   
\bibitem{Hoa88}
  L.~Hoa, \emph{On Segre products of affine semigroup rings}, Nagoya Math. J. 110 (1988), 113--128.   
  
\bibitem{HolmPires13}
  T.~Holm and A.~R.~Pires, \emph{The topology of toric origami manifolds},
    Math. Res. Lett. 20 (2013), no.~\textbf{5}, 885--906.   
     
   
\bibitem{James55}
  I.~James, \emph{Reduced product spaces}, 
   Ann. of Math. 62 (1955), 170--197.

\bibitem{Jan20}
  T.~Januszkiewicz, \emph{Actions with section},
  a talk in ``Workshop on Torus Actions in Topology'' held at the Fields Institute for
Research in Mathematical Sciences, May 11--15 (2020).
http://www.fields.utoronto.ca/talks/Actions-section

\bibitem{John83}
  F.~E.~A.~Johnson, 
  \emph{On the triangulation of stratified sets and singular varieties},
Trans. Amer. Math. Soc. 275 (1983), no.~\textbf{1}, 333--343.
 
\bibitem{LuPanov11}
   Z.~L\"u and T.~E.~Panov, \emph{Moment-angle complexes from simplicial posets},
     Cent. Eur. J. Math. 9 (2011), no.~\textbf{4}, 715--730.
     
\bibitem{MasPanov06}
 M.~Masuda and T.~Panov, \textit{On the cohomology of torus manifolds}, Osaka J.~Math.~43
  (2006), no.~\textbf{3}, 711--746.     
     
\bibitem{MilSturm05}
   E.~Miller and B.~Sturmfels, \emph{Combinatorial Commutative Algebra}, Graduate Texts in Mathematics~\textbf{227} (2005), Springer, Berlin.    
     
\bibitem{PodSark15} 
 M.~Poddar and S.~Sarkar, \emph{A class of torus manifolds with nonconvex orbit space}, Proc. Amer. Math. Soc. 143 (2015), no.~\textbf{4}, 1797--1811. 
 
\bibitem{Stanley07}
   R.~Stanley, \emph{Combinatorics and commutative algebra}, 2nd edition. Birkh\"auser Boston, 2007.
   
\bibitem{Vogt73}   
R.~Vogt, \emph{Homotopy limits and colimits},
 Math. Z. 134 (1973) 11--52.   
  
\bibitem{WangZheng15}
  X.~J.~Wang, Q.~B.~Zheng, \emph{The homology of simplicial complements and the cohomology of polyhedral products},
Forum Math. 27 (2015), no.~\textbf{4}, 2267--2299.  
  
\bibitem{WelkZiegZiv99}
V.~Welker, G.~Ziegler, R.~\v{Z}ivaljevi\'{c}, \emph{Homotopy colimits-comparison lemmas for combinatorial applications},
J. Reine Angew. Math. 509 (1999) 117--149. 


\bibitem{Whitehead56}
 G.~W.~Whitehead, \emph{Homotopy groups of joins and unions}, Trans. Amer. Math. Soc. 83 (1956), 55--69.  
 
\bibitem{Yos11}
 T.~Yoshida, \emph{Local torus actions modeled on the standard representation},  Adv. Math., 227 (2011),
1914--1955. 
  
\bibitem{Yu19}
  L.~Yu, \emph{On Hochster's formula for a class of quotient spaces of moment-angle complexes}, 
  Osaka J. Math. 56 (2019), no.~\textbf{1},  33--50. 
  
  
\bibitem{Zheng16}
  Q.~B.~Zheng, \emph{The cohomology algebra of polyhedral product spaces}, J. Pure Appl. Algebra 220 (2016), no.~\textbf{11}, 3752–3776.
  
  
\bibitem{ZiegZiv93}
G.~Ziegler, R.~\v{Z}ivaljevi\'{c}, \emph{Homotopy types of sub-space arrangements via diagrams of spaces}, Math. Ann. 295 (1993), 527--548.  

  
\end{thebibliography}
\end{document}